\documentclass[11pt]{article}

\usepackage[margin=1in]{geometry}
\usepackage{amssymb, color}
\usepackage{hyperref}
\usepackage{todonotes}
\usepackage{graphicx,xcolor}
\usepackage{amsthm,amsfonts,euscript,amsmath,comment,slashed,here,todonotes}
\usepackage[affil-it]{authblk}

\newtheorem{theorem}{Theorem}[section]
\newtheorem{lemma}[theorem]{Lemma}
\newtheorem{proposition}{Proposition}[section]

\theoremstyle{definition}
\newtheorem{definition}[theorem]{Definition}

\theoremstyle{remark}
\newtheorem{remark}[theorem]{Remark}

\numberwithin{equation}{section}

\newcommand\be{\begin{equation}}
\newcommand\ee{\end{equation}}
\newcommand\bea{\begin{eqnarray}}
\newcommand\eea{\end{eqnarray}}
\newcommand\bi{\begin{itemize}}
\newcommand\ei{\end{itemize}}
\newcommand\ben{\begin{enumerate}}
\newcommand\bena{\begin{enumerate}[(a)]}
\newcommand\een{\end{enumerate}}

\newcommand\bp{\begin{proof}}
\newcommand\ep{\end{proof}}

\newcommand{\R}{\ensuremath{\mathbb{R}}}

\newcommand{\N}{\mathbb{N}}

\renewcommand{\H}{\mathbb{H}}

\newcommand\de{\text{d}}
\newcommand\p{\partial}

\newcommand{\Ti}{\boldsymbol{T}}
\newcommand{\Ga}{\boldsymbol{\Gamma}}

\newcommand{\der}{V}
\newcommand{\BA}{\boldsymbol{A}}

\newcommand{\eps}{\varepsilon}

\allowdisplaybreaks

\title{Global stability for nonlinear wave equations with~multi-localized~initial~data}

\author[1]{John Anderson\thanks{jranders@math.princeton.edu}}
\author[1]{Federico Pasqualotto\thanks{fp2@math.princeton.edu}}
\affil[1]{\small  Department of Mathematics, Princeton University, Washington~Road,~Princeton~NJ~08544,~United~States~of~America \vskip.1pc \ }

\begin{document}

\maketitle

\begin{abstract}
In this paper, we initiate the study of the global stability of nonlinear wave equations with initial data that are not required to be localized around a single point. More precisely, we allow small initial data localized around any finite collection of points which can be arbitrarily far from one another. Existing techniques do not directly apply to this setting because they require norms with radial weights away from some center to be small. The smallness we require on the data is measured in a norm which does not depend on the scale of the configuration of the data. 

Our method of proof relies on a close analysis of the geometry of the interaction between waves originating from different sources. We prove estimates on the bilinear forms encoding the interaction, which allow us to show improved bounds for the energy of the solution. We finally apply a variant of the vector field method involving modified Klainerman--Sobolev estimates to prove global stability. As a corollary of our proof, we are able to show global existence for a class of data whose $H^1$ norm is arbitrarily large.
\end{abstract}

\tableofcontents

\section{Introduction}
The study of nonlinear hyperbolic equations is intimately tied with many fundamental physical phenomena. For example, the irrotational compressible Euler equations, which are useful for describing the dynamics of a compressible gas, can be realized as a system of quasilinear hyperbolic equations, see e.g.~\cite{courantfriedrichs}. Similarly, the dynamics of elastic materials can be described using systems of quasilinear hyperbolic equations \cite{Sideris2000}. We finally mention that the Einstein vacuum equations in general relativity can be realized as a system of quasilinear hyperbolic equations after choosing an appropriate gauge, see e.g.~\cite{choquetbruhat1952}.

In this paper, we initiate the study of global stability for systems of quasilinear wave equations with small initial data localized around $N$ fixed points which are allowed to be arbitrarily far away from each other. This extends classical results, which require initial data to be highly localized around a single point. One can regard the present study as a model problem to understand the interaction of two or more gravitational waves originating from far away sources, and we expect the methods developed in the present paper to extend to other physical equations such as the Einstein equations~\cite{forthcoming}.

\subsection{Formulation of the problem and historical remarks}
In this part of the introduction, we will first state a rough version of our result in the case of data localized around two points (Section~\ref{sub:resultsrough}). We will then proceed to describe the existing theory in Section~\ref{sub:existing}. Moreover, we will remark on the early developments concerning the classical null condition in Section~\ref{sub:nullintro}. Furthermore, in Section~\ref{sub:overview}, we will explain how our work relates to classical small data results, giving an overview of the problem at hand. 
In Section~\ref{sub:nirenberg}, as a motivating example, we will provide a concise proof of our theorem for a particular semilinear wave equation, to which the so-called Nirenberg trick applies.
In Section~\ref{sub:secondv}, we provide a rough statement for the general version of our main theorem, which applies to initial data localized around $N$ points far away from each other (in a sense which we are going to make precise later). 
In Section~\ref{sub:largedata}, we describe a large data  existence result for nonlinear wave equations satisfying the null condition which follows from the methods introduced in the present work. Finally, in Section \ref{structure}, we shall describe the structure of the rest of the paper.

\subsubsection{Rough description of the results}\label{sub:resultsrough}

This paper deals with the global stability of nonlinear wave equations with small data localized around $N$ distinct points.

To gain some intuition, we shall now state a special case of the theorem in which the data are localized around only two points. A rough version of the main result in full generality will be given in Theorem~\ref{thm:rough2}.

\begin{figure}[H]
\def\svgwidth{\linewidth}
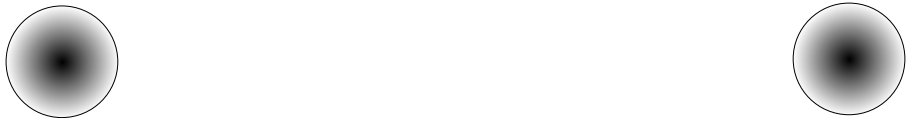
\caption{Initial data configuration in the special case $N = 2$.}
\label{fig:zerozero}
\end{figure}

In our main theorem, we are going to prove global stability for a class of systems of quasilinear wave equations of the following form:
\begin{equation} \label{eq:nonlwaverough}
    \begin{aligned}
        \Box \phi_A + \sum_{B, C= 1}^M F^{A}_{BC} (d \phi_B, d^2 \phi_C) = \sum_{B,C = 1}^M G^A_{BC} (d \phi_B, d \phi_C), \quad A = 1, \ldots, M.
    \end{aligned}
\end{equation}
Here, $\phi_A: [0, T]\times \R^{3} \to \R^M$ is the unknown, and $F^{A}_{BC}$ and $G^{A}_{BC}$ are collections of resp.~trilinear and bilinear forms satisfying the null condition:
\begin{equation}\label{eq:nullintro}
    F^{A}_{BC} (N,N,N) = 0, \quad  G^{A}_{BC} (N,N) = 0 \qquad \text{for all } \ N \in T\R^4 \ \text{ such that } m(N,N) =0.
\end{equation}
Here, $m$ is the Minkowski metric, and $\Box$ denotes the wave operator induced by $m$ on $\R^{3 + 1}$. See Section~\ref{sec:null} for additional details. For this class of equations, we have the following rough version of the theorem on global stability:

\begin{theorem} [Rough version of Theorem~\ref{thm:main} for $N = 2$]\label{thm:roughtwo}
For all sufficiently small data supported in two unit sized balls, nonlinear wave equations \eqref{eq:nonlwaverough} satisfying the classical null condition~\eqref{eq:nullintro} admit global-in-time solutions. Moreover, the smallness of initial data is measured in a norm which does not depend on the distance between the two balls.
\end{theorem}

Note that the theorem does not follow from the classical theory (see Section \ref{sub:existing} below) because we are not assuming the data to be localized around a single point.

We now proceed to review existing theory and connect it to the present work.

\subsubsection{Nonlinear waves: the classical small data theory}\label{sub:existing}

Despite the physical interest attached to nonlinear hyperbolic equations, their rigorous mathematical study is a relatively young field with comparatively few general results, especially outside of the $1 + 1$ dimensional setting of hyperbolic conservation laws (see e.g.~\cite{bressannotes, majdabook}). Specifically, it is only since the $1980$'s that small data global stability results have been rigorously established in more than one space dimension. Indeed, after proving local well-posedness, a very natural question is whether certain special solutions (such as the trivial solution) are globally stable under suitable perturbations of the initial data.

Motivated by physical models such as the irrotational compressible Euler equations, the equations of elasticity, and the Einstein vacuum equations, we are led to consider equations of the form
\begin{equation}\label{eq:nlwclassical}
    \begin{aligned}
        \Box_{g(\phi,d \phi)} \phi = N(d \phi),
    \end{aligned}
\end{equation}
where $g$ is a Lorentzian metric depending on the solution $\phi$ and its first derivatives $d \phi$, and where $N(d \phi)$ is a semilinear term.

Small data long time existence results for wave equations originate with the works of John and Klainerman. In \cite{Klainerman1980}, Klainerman was able to establish small data global existence results for nonlinear wave equations of the form~\eqref{eq:nlwclassical} in sufficiently high space dimensions (see also \cite{KlainermanPonce1983}). Already in these works, it is clear that the mechanism exploited for global existence is the decay of the solutions. This allows the contribution of the nonlinearity to the dynamics to be controlled upon integration. In the physical case of $3 + 1$ dimensions, ``almost global existence'' in the case of quadratic nonlinearities was established by John and Klainerman \cite{johnklainerman1984} (see also the earlier spherically symmetric case \cite{Klainerman1983}).

In \cite{Klainerman1985}, Klainerman developed a way of proving pointwise decay of solutions to wave equations by commuting with Lorentz vector fields (see Section \ref{sec:vectors}), circumventing the use of the fundamental solution that was common in previous work. This breakthrough already led to 
a sharp global existence result in dimensions $n + 1$ for $n \ge 3$. Indeed, he was able to show global existence in the cases of general quadratic nonlinearities for $n \ge 4$ and cubic nonlinearities for $n = 3$. This result is sharp since certain nonlinear wave equations with quadratic nonlinearities do not admit global-in-time solutions in $3 + 1$ dimensions for small initial data. The work of John in \cite{John1979} and \cite{John1981} already showed that global existence cannot be expected for general wave equations with quadratic nonlinearities in $\R^{3 + 1}$ (see also the work of Sideris \cite{sideris91}).\footnote{More recently, a mechanism of shock formation leading to blowup was investigated by Alinhac in~\cite{Alinhac99}. In the monumental work \cite{christodoulou2007}, Christodoulou was able to provide a complete description of shock formation, and subsequently, he was able to understand the shock development problem for the compressible Euler equations \cite{Christodoulou2017}. The understanding of shock formation was later generalized to a large class of wave equations by Speck in~\cite{speck2016}. This continues to be an active area of study.}

\subsubsection{Global existence in three space dimensions and the null condition}\label{sub:nullintro}

Given that general quadratic nonlinearities can lead to blowup in finite time, it is natural to look for a condition which can distinguish between nonlinearities that will result in small data global existence and nonlinearities that will not. This led to the discovery of the null condition~\eqref{eq:nullintro}, which was first discussed by Klainerman in~\cite{icm1982}.

Already in \cite{Klainerman1980}, a particular example of a nonlinear wave equation satisfying what became known as the null condition appeared, the so-called \emph{Nirenberg example}. This is the simple equation
\begin{equation} \label{nirenbergex1}
    \begin{aligned}
        \Box \phi = m(d \phi,d \phi),
    \end{aligned}
\end{equation}
where $m$ is the Minkowski metric. The equation \eqref{nirenbergex1} admits a representation formula that allows for a precise analysis of the solution (see Section \ref{sub:nirenberg} below). Because \eqref{nirenbergex1} is a wave equation with a quadratic nonlinearity admitting global solutions for small data, it becomes apparent that it is possible for such equations to have globally stable trivial solutions in $\R^{3 + 1}$. This provided an early clue that some condition on the nonlinearity may result in global stability. Indeed, the trivial solution to nonlinear wave equations satisfying the null condition~\eqref{eq:nullintro} is globally stable in $3 + 1$ dimensions. This was first shown by Klainerman using the vector field method \cite{Klainerman1986}, and Christodolou gave a different proof using conformal compactification in \cite{christodoulou1986}. Since then, there have been other proofs, as well as several generalizations and extensions.\footnote{For wave equations on Minkowski space, we mention the work of Klainerman--Sideris \cite{KS1996}, Katayama \cite{Katayama2004}, \cite{Katayama2005}, and Katayama--Yokoyama \cite{KY2006}. We also mention the work of Yang \cite{Yang2013}, \cite{yang2016} concerning quasilinear wave equations on perturbations of Minkowski space and that by Keir~\cite{keir2018} on equations satisfying the weak null condition on a general class of asymptotically flat manifolds. Both of these works are based on the $r^p$ method due to Dafermos--Rodnianski \cite{DR2010}. We also mention the work of Pusateri--Shatah \cite{PS2013} on equations satisfying the null condition and Deng--Pusateri \cite{DP2018} on equations satisfying the weak null condition using the method of spacetime resonances (see, for example, \cite{germain2011} and the references therein).} We discuss the null condition further in Section~\ref{sec:null}.

Among the physical systems used as motivation above, the null condition appears in certain models of elasticity. In this context, the null condition was used by Sideris to prove a global stability result in \cite{Sideris2000} dealing with elastic materials that are homogeneous, isotropic, and hyperelastic. However, the null condition is not satisfied by the irrotational compressible Euler equations.

The case of the Einstein vacuum equations is more subtle. The classical null condition is not satisfied by the Einstein vacuum equations in wave coordinates. Nonetheless, identifying the presence of a form of the null condition in a more geometric setting was an important step in the monumental work by Christodoulou--Klainerman \cite{globalnon} showing that Minkowski space is globally nonlinearly stable as a solution to the Einstein vacuum equations under suitable perturbations. We note that later, Lindblad and Rodnianski identified a generalization of the classical null condition which is referred to as the weak null condition. The weak null condition is satisfied by the Einstein vacuum equations in wave coordinates, and this observation allowed Lindblad and Rodnianski to prove the global nonlinear stability of Minkowski space using wave coordinates in \cite{LR2010}.

Compared to generic quadratic nonlinearities, those satisfying the null condition are better behaved because the null structure allows stronger estimates to be proven. Such improved estimates are an essential ingredient in proving global stability. Other examples of the importance of estimates taking advantage of null structure concern low regularity results. In this context, there is the work of Klainerman--Machedon \cite{klainermanmachedon1993}, \cite{klainermanmachedon1995}, Klainerman--Rodnianski--Tao \cite{KRT2002}, and Klainerman--Rodnianski \cite{KR2005} on bilinear estimates. Then, in a monumental series of papers, Klainerman, Rodnianski, and Szeftel were able to prove the Bounded $L^2$ Curvature Conjecture (see \cite{KRS2015} and the references therein). Bilinear estimates (in the form of those proved in \cite{KR2005}) were a fundamental tool in proving this conjecture.

The present work also relies heavily on null structure. Indeed, the most important estimate in the following argument is one that controls the nonlinear interaction between waves originating from distant sources when the nonlinearity satisfies the null condition. These trilinear estimates involve controlling expressions that arise from using a multiplier on a null form. This is described in more detail in Section \ref{subsub:quadraticimp}.

\subsubsection{Overview of the problem: why the classical theory does not apply}\label{sub:overview}

As was already mentioned, previous results concerning global existence for nonlinear wave equations require the initial data to be localized around a single point in a quantitative sense. For example, in \cite{Klainerman1986}, the requirement for global existence is that norms of the form $\Vert \partial \Gamma^\alpha \phi \Vert_{L^2 (\Sigma_0)}$ be of size $\eps$, where $\alpha$ is a multi-index, $\Gamma^\alpha$ is a string of Lorentz vector fields, $\phi$ is the solution to the equation at hand, and $\Sigma_0$ is the hypersurface defined by $t = 0$ (thus, the above is a requirement on the initial data). We note that Lorentz vector fields in general will produce radial weights at $t = 0$, which implies that the norm of initial data will be large if the data itself is not localized around the origin. Indeed, the Lorentz vector fields are given by
\begin{equation} \label{lfields1}
    \begin{aligned}
        x^i \partial_j - x^j \partial_i, \qquad t \partial_i + x^i \partial_t, \qquad t \partial_t + r \partial_r.
    \end{aligned}
\end{equation}
At $t = 0$, these vector fields manifestly have weights that grow away from the fixed origin.

We now consider the following situation. We take the points $(R,0,0)$ and $(-R,0,0)$ on the $x$-axis in $\R^3$. We then fix two pairs of real valued functions $(\phi_0^R,\phi_1^R)$ and $(\phi_0^{-R},\phi_1^{-R})$. The first pair of functions is supported in the unit ball centered at $(R,0,0)$, and the second pair of functions is supported in the unit ball centered at $(-R,0,0)$. 

We then consider an arbitrary nonlinear wave equation of the type~\eqref{eq:nlwclassical} satisfying the null condition~\eqref{eq:nullintro}, and we impose the sum of the above functions as initial data:

\begin{equation}
    \begin{aligned}
        & \Box_{g(\phi,d \phi)} \phi = N(d \phi),\\
        &\phi|_{t=0} = \phi_0^R + \phi_0^{-R},\\
        &\partial_t \phi|_{t=0} = \phi_1^R + \phi_1^{-R}.
    \end{aligned}
\end{equation}

We study the resulting initial value problem with $R$ as a parameter. For $R = 1$, we know that, after scaling the initial data by some small parameter $\eps$, this initial value problem has a globally stable trivial solution by known results. If we take $R$ large, however, we note that the initial data will grow like some power of $R$ in the norms required by existing results on nonlinear wave equations satisfying the null condition. For example, the weighted norms appearing in~\cite{Klainerman1986} satisfy:
\begin{equation*}
    \Vert \p \Gamma^\alpha \phi  \Vert_{L^2(\Sigma_0)} \sim \eps R^{|\alpha|}.
\end{equation*}
This holds because the vector fields in \eqref{lfields1} have weights proportional to the distance between the supports of $\phi_i^R$ and $\phi_i^{-R}$, which is of size $R$. Thus, for $R$ sufficiently large, the initial data would appear large from the point of view of the norms used for global stability in the classical theory. Because of this, the classical results do not apply for this class of initial data. The difficulties introduced by these weights are discussed in more detail in Section~\ref{subsub:badrsobolevs} below.

\subsubsection{A special case: the Nirenberg example}\label{sub:nirenberg}

As a motivating example, we now prove a version of our Theorem~\ref{thm:roughtwo} for a specific equation, the \emph{Nirenberg example}, which we recall was discussed in the context of the classical small data theory in Section \ref{sub:nullintro}. Proving this result for this particular equation is significantly easier than addressing the general case of \eqref{eq:nonlwaverough} because it has a representation formula. Thus, studying this particular example serves as a preliminary check before trying to prove a more general result.

Let us consider the initial value problem for the Nirenberg example:
\begin{equation}\label{eq:nlw}
\begin{aligned}
    &\Box \phi = m(d \phi, d \phi),\\
    &\phi|_{t=0} = \phi_0^{R} + \phi_0^{-R},\\
    &\p_t \phi|_{t=0}  = 0.
\end{aligned}
\end{equation}
We recall that $m$ is the Minkowski metric, and we take $\phi_0^{R}$ to be a smooth function supported in the unit ball centered at the point $(x,y,z)=(R, 0, 0)$, while taking $\phi_0^{-R}$ to be a smooth function supported in the unit ball centered at the point $(x,y,z)= (-R, 0,0)$. Furthermore, we assume that both $\phi_0^{R}$ and $\phi_0^{-R}$ have small $H^{N_0}(\R^3)$ norm, where $N_0$ is assumed to be a large positive integer.

\begin{remark}
	Note that this norm does not depend on the distance ($2 R$) between the supports of $\phi_0^{R}$ and~$\phi_0^{-R}$.
\end{remark}

It is then straightforward to note that, upon defining
$$
u := \exp(-\phi),
$$
the initial value problem~\eqref{eq:nlw} is equivalent to the following linear problem:
\begin{equation}\label{eq:nirenberg}
\begin{aligned}
    &\Box u = 0,\\
    &u|_{t=0} = \exp(-\phi^R_{0} - \phi^{-R}_{0}) =: u_0,\\
    &\p_t u|_{t=0} = 0.
\end{aligned}
\end{equation}
The solution $\phi$ to the original nonlinear wave equation can be recovered by taking $\phi = \log u$. It is clear that $u_0$ is positive everywhere. Furthermore, it is evident that, as long as $u$ solving~\eqref{eq:nirenberg} remains positive, the solution $\phi$ to problem~\eqref{eq:nlw} will not develop singularities.

The function $u_0$, in addition to being positive everywhere, is identically 1 outside of the union of the supports of $\phi_0^{R}$ and $\phi_0^{-R}$. This set is contained, by our choice, in the union of two unit balls centered resp.~at $(R,0,0)$ and $(-R,0,0)$. 

Using now the Kirchhoff formula for solutions to the linear wave equation in $\R^{3+1}$, we obtain that
\begin{equation} \label{ntr1}
    u(t,x) = \frac{1}{4\pi t^2}\int_{\p B_{t,x}} (\underbrace{u_0(y)}_{(a)} + \underbrace{(y-x) \cdot \nabla u_0(y)}_{(b)}) \de S(y)
\end{equation}
Here, $\partial B_{t,x}$ is the boundary of the sphere of radus $t$ centered at the point $x \in \R^3$, whereas $\de S(y)$ denotes the induced area form on $\partial B_{t,x}$.

We then note that, if $R$ is chosen to be large enough, the term corresponding to $(a)$ in \eqref{ntr1} is, upon integration, positive and of size $\sim t^2$, regardless of the choice of $x$ and $t$. This is because the initial data for $u_0$ is 1 on the whole sphere $\partial B_{t,x}$, except at most two caps of unit size. On the other hand, the term corresponding to $(b)$ is, upon integration, at most of size $t$. This is because the expression $(y-x) \cdot \nabla u_0(y)$, restricted to $y \in \p B_{t,x}$, vanishes outside of a set of measure comparable to one (the union of supports of $\phi_0^{R}$ and $\phi_0^{-R}$). On the same set, on the other hand, the expression $(b)$ is of size at most $t$.

In conclusion, for large enough values of $R >0$, $u$ stays positive for all times. Thus the solution $\phi$ does not develop a singularity. This reasoning proves Theorem~\ref{thm:roughtwo} in the very specific case of equation \eqref{eq:nlw}, and it suggests that the result may hold true for a more general class of nonlinear wave equations satisfying the null condition~\eqref{eq:nullintro}.

\subsubsection{Main Theorem: second version}\label{sub:secondv}
We now state a more precise version of our global stability theorem. Let us consider an initial configuration of points $\{p_i\}_{i \in \{1, \ldots, N\}}$ along with the origin $O$ in $\R^3$. For technical convenience, we can assume that the minimum pairwise distance between the points is $2$ after possibly rescaling about the origin. We now scale the configuration by some $R \geq 1$ around the origin $O$, meaning that we consider the points $R p_i$. The smallest and largest pairwise distances between the points are then proportional to $R$; see Figure~\ref{fig:zero} to see how the configuration of our initial data looks. 

\begin{figure}[H]
\def\svgwidth{\linewidth}
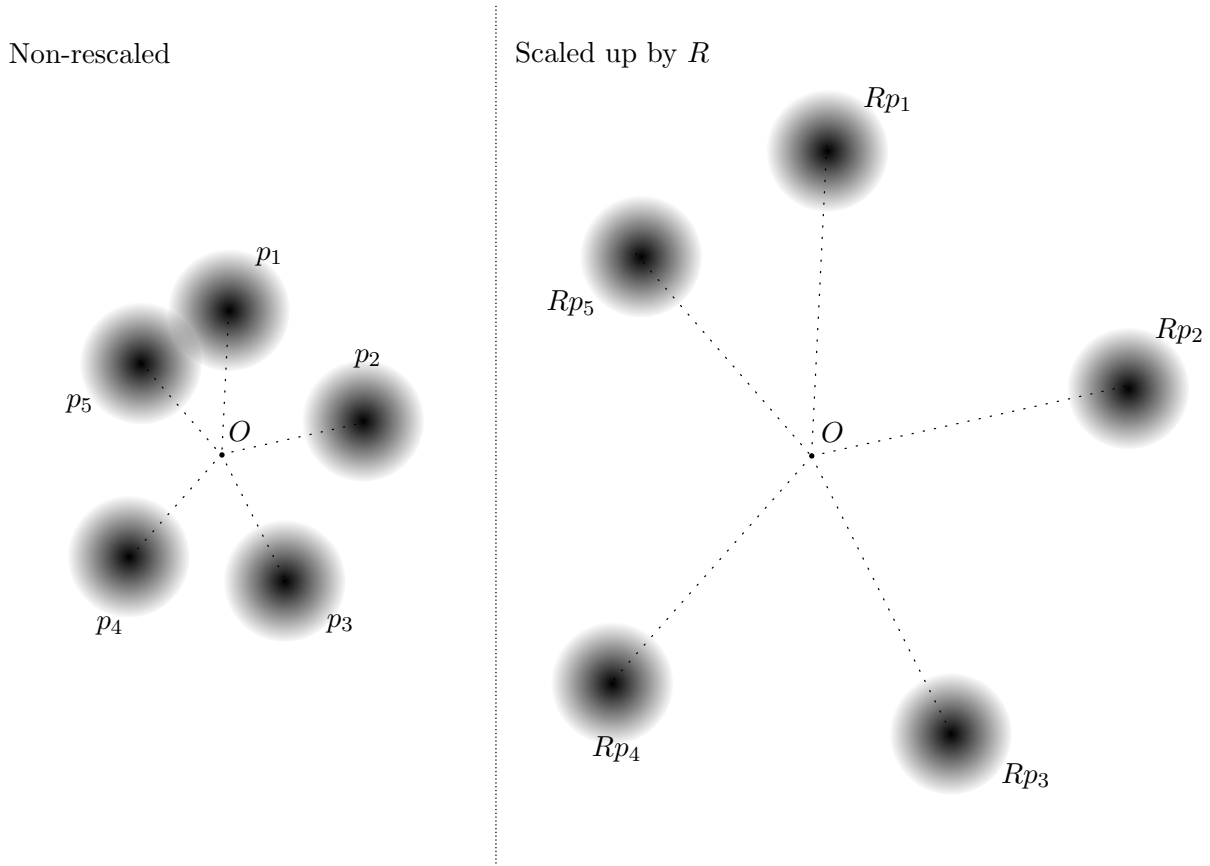
\caption{Initial data configuration and the parameter $R$.}
\label{fig:zero}
\end{figure}

We shall prove global stability of the trivial solution for general quasilinear systems of wave equations satisfying the classical null condition independently of the parameter $R$. In other words, the size of the initial data will be allowed to depend only on $N$, where $N$ is the number of points around which the data are localized, and $d_\Pi$, where $d_\Pi$ is the ratio of the largest and smallest pairwise distances between the $N$ points (see Section \ref{csystems}). It is in this sense that the smallness of the initial data does not depend on the scale of its configuration, which is measured by the parameter $R$. Summarizing, we have the following rough version of the global stability theorem (Theorem~\ref{thm:main}).

\begin{theorem}[Second rough version of Theorem~\ref{thm:main}]\label{thm:rough2}
For all sufficiently small data which are localized around $N$ points as in Figure~\ref{fig:zero}, the system of nonlinear wave equations \eqref{eq:nonlwaverough} admits global-in-time solutions. Moreover, the smallness of initial data data is measured in a norm which does not depend on the scaling parameter $R$, and thus depends only on $N$ and $d_\Pi$, which is the ratio of the largest and smallest pairwise distances between the $N$ points.
\end{theorem}

The main obstacle in proving this result is showing decay, which we are able to establish in this setting because waves originating from distant localized sources disperse before interacting. The main tool used to prove this theorem is a trilinear estimate measuring the nonlinear interaction between two waves when the nonlinearity satisfies the null condition. This will allow us to show improved estimates on terms measuring the interaction between the distinct waves. The proof of these trilinear estimates involves an analysis of the geometry of the interaction of two waves which originate from distant sources. We give a more detailed description of this procedure in Section \ref{subsub:quadraticimp} and in Section \ref{subsub:linearestint}, and the estimates themselves are proven in Section \ref{sub:improvedenergy}.

Another tool we shall need are modified Klainerman--Sobolev inequalities. These inequalities are designed to take advantage of the symmetry that exists when a pair of waves interact. This is described in more detail in Section \ref{subsub:rweightedintro}, and the proofs themselves can be found in Section \ref{sec:sobolevs}.

We believe that the dependence of $\eps$ on the ratio $d_\Pi$ of the largest and smallest pairwise distances between the points is purely technical. It arises from the fact that the pointwise estimates are far from sharp. We conjecture that the result is true with $\eps$ independent of the parameter $d_\Pi$, but the proof in this paper cannot directly be used to establish this result.\footnote{We have not tried to optimize the proof in terms of the dependence on $d_\Pi$. However, we note that the proof can be used to remove the dependence of $\eps$ on $d_\Pi$ when $R$ is sufficiently large as a function of $d_\Pi$. This is similar to the proof of Theorem \ref{thm:largedata}, which is carried out in Section \ref{sec:largedata}.}

\subsubsection{A class of large data}\label{sub:largedata}
As a corollary of the main theorem described above, we are able to prove global existence for a class of initial data whose $H^1$ norm is arbitrarily large.\footnote{The class of data considered here can be compared with the data considered in, say, \cite{yang2015} and \cite{Miao2017}. In both cases, the large energy of the data is concentrated in outgoing waves. For our class of data, energy is not required to be outgoing in the same sense, and the largeness comes from the fact that the data is not well localized. However, the data is still mostly outgoing in the sense that there cannot be much focusing.}

The data we consider shall be localized around $N$ points $\{ w_i \}$ with $i \in \{ 1, \ldots, N \}$. The energy can be made large because the pairwise distances between the $w_i$ can be made even larger. Thus, the data localized around $w_i$ and the data localized around $w_j$ with $i \ne j$ will not interact for a long time, allowing the dispersive effects of the wave equation to still dominate. The rough version of this large data result is as follows. For the precise statement, see Theorem~\ref{thm:largedata}.

\begin{theorem} [Rough version of Theorem \ref{thm:largedata}] \label{thm:largedatarough}
    Given any $L > 0$, there exist initial data with $H^1$ norm of size $L$ supported in $N(L)$ distinct balls of unit size for which the system of nonlinear wave equations~\eqref{eq:nonlwaverough} admits a global-in-time solution. 
\end{theorem}

We note, however, that the $H^k$ norm of the data in any of the $N(L)$ balls is still small (here, we assume that  $k \geq 19$).

\subsection{Structure of the paper} \label{structure}

We shall now describe the structure of the paper. In Section~\ref{sec:ideasofproof}, we shall describe in more detail the main difficulties, motivate the ideas of the proof, and give a detailed outline of the proof. The outline may be useful to consult while reading the bulk of the paper.

In Section~\ref{sec:setup}, we will introduce all the relevant classical definitions and all the notation. 

In Section~\ref{sec:statements}, we will give a rigorous statement of the main theorem of the paper (Theorem~\ref{thm:main}), and we will also give two auxiliary statements (Theorem~\ref{thm:linear} and Theorem~\ref{thm:nonlinear}). We will furthermore show how our main theorem follows from these two auxiliary theorems. After this, we will state the theorem on global existence for a class of data which is arbitrarily large in $H^1$ (Theorem \ref{thm:largedata}). The remainder of the paper will therefore be devoted to proving the auxiliary theorems and Theorem~\ref{thm:largedata}.

In Section~\ref{sec:technicaltools}, we will state and prove three statements that are fundamental for the proof of our result: Lemma~\ref{prop:decphii} concerning the improved $u$-decay of the solution to a quasilinear wave equation satisfying the null condition (with localized initial data), Lemma~\ref{lem:r1r2} concerning an important change of coordinates, and finally Lemma~\ref{lem:goodbad} which enables us to express null derivatives with respect to a certain light cone in terms of null derivatives with respect to another (translated) light cone. Lemmas \ref{lem:r1r2} and \ref{lem:goodbad} are used in the proof of the improved energy estimates and the associated trilinear estimates in Section \ref{sub:improvedenergy}.

Subsequently, in Section~\ref{sec:sobolevs}, we will prove particular types of Klainerman--Sobolev embeddings with $R$-weights, complementing the existing theory. These results are recorded in a separate section due to their crucial importance in the scheme of the paper.

Having proved these estimates, we will then proceed to estimate solutions to the linear inhomogeneous equations encoding the interaction of two waves, in Section~\ref{sec:linearest}. First, in Section~\ref{sub:constructinit}, we will decompose the initial data of our problem in $N$ compactly supported pieces, plus a remainder. Subsequently, in Section~\ref{sub:firstiterate}, we will deduce the equation satisfied by the ``first iterate'' (i.e.~the first-order approximation of the solution $\phi$ in terms of inverse $R$ weights), as well as the equation satisfied by the ``second iterate'' $\psi_{i j}$ in Section~\ref{sub:seconditerate}, which will be used crucially to close the estimates in the proof of Theorem~\ref{thm:nonlinear}. We will then proceed, in Section~\ref{sub:improvedenergy}, to show the improved bound for the energy of $\psi_{ij}$ along with the associated trilinear estimates: this is carried out in Proposition~\ref{prop:energy}. These estimates are the most important part of the following argument. We will finally deduce quantitative $L^\infty$ decay estimates for $\psi_{ij}$ in Section~\ref{sub:linfpsiij} using the $R$-weighted Klainerman--Sobolev inequalities proved in Section~\ref{sec:sobolevs}.

After that, we will turn to the proof of Theorem~\ref{thm:nonlinear} in Section~\ref{sec:mainproof}. This will conclude the proof of our main Theorem~\ref{thm:main}. Using Theorem \ref{thm:main}, we will then prove Theorem \ref{thm:largedata} in Section~\ref{sec:largedata}.

Finally, in the appendices, we will record several technical facts needed in the rest of the paper.

\subsection{Acknowledgements}
We would like to thank our advisors, Mihalis Dafermos and Sergiu Klainerman, for their support and comments on the manuscript. We are particularly indebted to them and to Jonathan Luk and Igor Rodnianski for many insightful discussions and useful suggestions when studying this problem. We are also very thankful to Yakov Shlapentokh-Rothman for making us aware of this problem.

\section{Motivation of the proof and overview of the main ideas}\label{sec:ideasofproof}
We shall now describe the difficulties in more detail, as well as the main points in the argument. The reader may want to return to this section for a detailed outline while reading the paper.

After a more precise description in Section~\ref{subsub:badrsobolevs} of why the classical theory does not apply, we will give an overview of the main ideas in the proof in Section \ref{subsub:mainideas}. Then, we will explain the intuition behind our $R$-weighted Sobolev embeddings in Section~\ref{subsub:rweightedintro}.  We will then proceed to outline how we can obtain improved $L^2$ estimates in the special case of an equation with cubic nonlinearities in  Section~\ref{subsub:cubic}. After that, we will outline how the same improvement can be obtained for a wave equation with quadratic nonlinearities in~Section~\ref{subsub:quadraticimp}. Subsequently, in Section~\ref{subsub:linearestint}, we will describe how some important geometric facts will be instrumental in showing the improved energy bounds and the associated trilinear estimates. Finally, in Section~\ref{subsub:closing}, we will describe how we close the argument.

\subsection{Description of the main difficulties}\label{subsub:badrsobolevs}

It is well known that, in order prove global stability for nonlinear wave equations, the two main ingredients are energy estimates and pointwise estimates (see~\cite{Klainerman1985}). More precisely, weighted energies allow us to prove pointwise decay estimates, which in turn enable us to show that the accumulation of the nonlinear effects remains negligible. We recall, for example, the Klainerman--Sobolev inequality (see~\cite{Klainerman1985}), which can be used to prove global stability for certain nonlinear wave equations:\footnote{More precisely, nonlinear wave equations with cubic nonlinearities in $\R^{3 + 1}$, and with quadratic nonlinearities in $\R^{d + 1}$, where $d \ge 4$.}
\begin{equation}\label{eq:ksobintro}
    \begin{aligned}
        |f| (t,r,\omega) \le \frac{C}{(1 + t + r) (1 + |t - r|)^{\frac 12}} \sum_{|\alpha| \le 2} \Vert \Gamma^\alpha f \Vert_{L^2 (\Sigma_t)}.
    \end{aligned}
\end{equation}
Here, $\alpha$ is a multiindex with $|\alpha|$ its length, and $\Gamma^\alpha$ corresponds to a product of Lorentz fields along with the scaling vector field. The function $f$ is any smooth function decaying sufficiently rapidly at infinity in $\Sigma_t$.\footnote{This is the set of all points $(t,x,y,z)$ in $\R^{3 + 1}$ with first coordinate equal to $t$. The data is posed on $\Sigma_0$.} In practice, if we take $f = \partial \phi$ where $\phi$ is the solution to some wave equation, the RHS of~\eqref{eq:ksobintro} is precisely an energy term. The terms on the RHS can then be estimated by energies, exactly because the Lorentz fields commute with $\Box$. Schematically, the $\partial_t$ energy estimate shows that
\begin{equation}
    \begin{aligned}
        \Vert \partial \Gamma^\alpha \phi \Vert_{L^2 (\Sigma_t)} \lesssim \Vert \partial \Gamma^\alpha \phi \Vert_{L^2 (\Sigma_0)},
    \end{aligned}
\end{equation}
where the RHS is determined by the initial data.

In our case, since we consider data localized around points $\{w_i\}_{i \in \{1, \ldots,N\}}$ whose pairwise distances are comparable to $R$, we note that the initial norm $\Vert \partial \Gamma^\alpha \phi \Vert_{L^2 (\Sigma_0)}$
will potentially be large as a function of $R$. Indeed, every Lorentz vector field having $w_i$ as its origin will produce an $R$ weight when applied to the piece of data localized at $w_j$, for $j \neq i$.

This would result in an estimate of the form
\begin{equation} \label{WastefulKS1}
    \begin{aligned}
        |\partial \phi| (t,r,\omega) \le {C R^2 \over (1 + t + r) (1 + |t - r|)^{{1 \over 2}}} \Vert \p \phi \Vert_{H^2 (\Sigma_t)},
    \end{aligned}
\end{equation}
The dependence on $R$ in this estimate is a serious obstruction to proving global stability independently of $R$.

\subsection{Overview of the main ideas and outline of the proof}\label{subsub:mainideas}

In order to overcome this obstruction, we shall need to apply the vector field method more carefully. We first note that the main contribution to the solution should come from a superposition of the contributions arising from each of the compactly supported pieces of data considered individually. Subtracting off these $N$ auxiliary solutions (which satisfy good asymptotics from existing theory, and which we can treat as known functions), we obtain a nonlinear equation with an inhomogeneity and vanishing initial data whose solutions have small energy in terms of the parameter $R$. This is carried out in Sections~\ref{sub:firstiterate} and~\ref{sub:seconditerate}.

More precisely, the strategy can be divided into three steps.
\begin{itemize}
\item[1.] First, we reduce to the case of $N$ compactly supported pieces of initial data localized around $N$ points, plus an additional error. This procedure is effective because the resulting error is small in terms of the parameter $R$. The relevant details are in Section~\ref{sub:constructinit} below.

\item[2.] As a second main step, we shall accomplish the task of showing that the main contribution to the solution comes from the superposition of the contributions arising from each individual compactly supported piece of data. This will be shown by proving that the resulting error is small in terms of $R$. This will be carried out in Section~\ref{sub:improvedenergy} below. 

\item[3.] Finally, we shall improve (in terms of  $R$) the pointwise estimates on the error arising from the first and second step. This will be done in Section~\ref{sub:linfpsiij}, using the $R$-weighted Sobolev embeddings of Section~\ref{sec:sobolevs}.
\end{itemize}

After finishing these three parts, we will be in a position to prove global stability. This is carried out in Section \ref{sec:mainproof}.

In the case of cubic and higher nonlinearities (see Section \ref{subsub:cubic}), the heuristic reasoning presented here (in particular, the second step in the previous list) can be made quantitative, and the resulting weights in $R$ are particularly favorable. 
Such an improvement in energy creates good powers of $R$ in the $L^2$ norms in the RHS of estimate \eqref{WastefulKS1}. We now note the two main observations which go into showing this energy improvement in $R$ for a cubic equation.
\begin{itemize}
\item The different solutions will not interact before time comparable to $R$.
\item The individual solutions satisfy good quantitative decay estimates.
\end{itemize}
The second of these points follows from the existing theory (see also the improvements discussed in~\ref{sec:onebump} below). The first of these points can be seen by looking at Figure~\ref{fig:first} below.

In order to close an argument for the cubic equation, we must then simply improve the $L^\infty$ estimates, which amounts to lowering the power of $R$ in the RHS of \eqref{WastefulKS1}.

Both of these aspects (the improvement in energy and the improvement in the Klainerman--Sobolev estimates) are more complicated for the quadratic equation (see Sections \ref{subsub:rweightedintro}, \ref{subsub:quadraticimp}, and \ref{subsub:linearestint}).

\subsubsection{\texorpdfstring{$R$}{R}-weighted vector fields and improved \texorpdfstring{$L^\infty$}{Linfinity} estimates}\label{subsub:rweightedintro} We first observe the following fact: the vector fields $\Gamma_R := {1 \over R} \Gamma$ do not introduce $R$ weights on the initial data (recall that the distance between the supports of each individual piece of data is proportional to $R$). Repeating the proof of the Klainerman--Sobolev inequality using the vector fields $\Gamma_R$ instead of the vector fields $\Gamma$ results in the following estimate:
\begin{equation} \label{BetterPointwise1}
    \begin{aligned}
        |f|(t,r,\omega) \le {C R^{{3 \over 2}} \over (1 + t + r) (1 + |t - r|)^{{1 \over 2}}} \sum_{|\alpha| \le 3} \Vert \Gamma_R^\alpha f \Vert_{L^2 (\Sigma_t)}.
    \end{aligned}
\end{equation}
Note that estimate~\eqref{BetterPointwise1} still depends on $R$. However, the constant on the RHS is already better in terms of $R$-weights than the corresponding one appearing in \eqref{WastefulKS1}. Estimate~\eqref{BetterPointwise1} is proven in Lemma~\ref{lem:bettersob2} below. Using a similar argument but not requiring decay in $u = t - r$ allows us to prove the following Klainerman--Sobolev inequality for $t \ge {R \over 10}$:
\begin{equation} \label{goodRpointwiseest1}
    \begin{aligned}
        |f| (t,r,\omega) \le {C R \over (1 + t + r)} \sum_{|\alpha| \le 3} \Vert \Gamma_R^\alpha f \Vert_{L^2 (\Sigma_t)}. 
    \end{aligned}
\end{equation}
A version of this is proven in Lemmas \ref{lem:bettersob1} and \ref{lem:bettersob2}. The estimate \eqref{goodRpointwiseest1} along with the improved energy estimates is enough to prove global stability for nonlinear equations with cubic and higher nonlinearities.

Concerning the quadratic case, such improvement in the Sobolev estimates is not enough, and we will elaborate on this point in the present and following Section~\ref{subsub:cubic}. The further complication is that we also need improved decay for good derivatives because we need to take advantage of the classical null condition (see Section \ref{subsub:linearestint} for a discussion on good derivatives).

To obtain better estimates, we realize that, restricting to $N = 2$ with both pieces of data supported in unit balls whose centers are on the $x$ axis, then those rotations and Lorentz boosts which involve only $y$, $z$, and $t$ do not introduce $R$-weights on the initial data. These are the following vector fields:
\begin{equation}\label{eq:goodvf}
    z\p_y - y \p_z, \qquad t \p_y + y \p_t, \qquad t \p_z + z\p_t.
\end{equation}
Thus, we can use such vector fields without dividing by $R$ in the RHS of the Klainerman--Sobolev estimates. This observation will allow us to get an estimate of the form
\begin{equation}\label{eq:sobolevtwo}
    \begin{aligned}
        |f (t,r,\omega) | \le {C R^{{1 \over 2}} \over (1 + t + r)} \Vert \overline{\Gamma}^\alpha f \Vert,
    \end{aligned}
\end{equation}
where now the term $\Vert \overline{\Gamma}^\alpha f \Vert$ schematically refers to $L^2$ norms we expect to correspond to $\partial_t$ energies after commuting by appropriately $R$ weighted Lorentz fields. By appropriately $R$ weighted, we mean the $\Gamma_R$ vector fields used in \eqref{BetterPointwise1} along with the vector fields that do not introduce bad weights given by \eqref{eq:goodvf}. The precise version of estimate~\eqref{eq:sobolevtwo} is proven in Lemma~\ref{lem:sobspheres} below.

This observation is useful even when $N > 2$. Indeed, for quadratic nonlinearities, all interactions occur between only two pieces of data up to an error that is better in $R$. This fact allows us to take advantage of \eqref{eq:sobolevtwo} even when the data are localized around more than two different points. This procedure will be carried out in Section \ref{sec:linearest}.

\subsubsection{The cubic case: subtracting off auxiliary solutions}\label{subsub:cubic}

In this section, we describe how we can improve the energy estimates (in terms of $R$-weights) for a cubic equation by subtracting off the main contribution to the solution arising from several localized pieces of initial data.

We take the cubic equation:
\begin{equation} \label{CubicEq1}
    \begin{aligned}
        \Box \phi = (\partial_t \phi)^3,
    \end{aligned}
\end{equation}
with initial data supported in two disjoint balls, one centered at $(-R,0,0)$, and the other centered at $(R,0,0)$. We should regard $R$ as being a very large number, and the initial data for $\phi$ as being small in some $H^k(\R^3)$ norm.

As was discussed previously in Section~\ref{subsub:mainideas}, we want to take advantage of the following two facts: the individual pieces of data give rise to solutions that disperse, and the two solutions do not interact for time comparable to $R$. We do so by subtracting off the functions $\phi_1$ and $\phi_2$, where $\phi_1$ is the solution to the equation~\eqref{CubicEq1} with initial data equal to the portion of the initial data for $\phi$ supported in the ball centered at $(-R,0,0)$. Similarly, $\phi_2$ is the solution to equation~\eqref{CubicEq1} with initial data equal to the portion of the initial data of $\phi$ that is supported in the ball centered at $(R,0,0)$. 

By our assumptions on the data and by classical theory (see \cite{Klainerman1985}), we know that $\phi_1$ and $\phi_2$ exist globally because their data are sufficiently small in $H^k(\R^3)$ (and we note that, here, the smallness does not depend on $R$, as $\phi_1$ and $\phi_2$ each only have localized data). Moreover, $\phi_1$ and $\phi_2$ decay at the same rates as solutions to the linear wave equation with compactly supported data. Therefore, quantitatively, we can assert that
\begin{equation} \label{CubicPointwise1}
    \begin{aligned}
    \Vert \partial \phi_1 \Vert_{L^\infty (\Sigma_t)} + \Vert \partial \phi_2 \Vert_{L^\infty (\Sigma_t)} \le \frac{C \eps}{1 + t}.
    \end{aligned}
\end{equation}
Furthermore, we have the following control on the energy of each individual piece:
\begin{equation} \label{CubicEnergyEst1}
    \begin{aligned}
        \Vert \partial \phi_1 \Vert_{L^2 (\Sigma_t)} + \Vert \partial \phi_2 \Vert_{L^2 (\Sigma_t)} \le C \eps.
    \end{aligned}
\end{equation}
We now subtract off these functions from the solution $\phi$, and we obtain an equation for $\psi := \phi - \phi_1 - \phi_2$. Schematically, the equation is as follows:
\begin{equation}\label{eq:cubicdiff}
    \begin{aligned}
        \Box \psi \approx (\partial_t \psi)^3 + (\partial_t \psi) ((\partial_t \phi_1)^2 + (\partial_t \phi_2)^2) + (\partial_t \phi_1) (\partial_t \phi_2) (\partial_t \phi_1 + \partial_t \phi_2).
    \end{aligned}
\end{equation}
This equation is inhomogeneous, and furthermore, the initial data for $\psi$ vanishes. The crucial observation here is that $(\partial_t \phi_1)^3$ and $(\partial_t \phi_2)^3$ do not appear on the RHS, and this is consistent with having subtracted off the main contribution to the solution arising from both $\phi_1$ and $\phi_2$. Indeed, because $\psi$ has vanishing initial data, we expect that only the inhomogeneities are responsible for adding energy into $\psi$.  
Moreover, note that the remaining terms are all supported in $t \ge R - 2$ (see again Figure~\ref{fig:first}), which encodes the fact that $\phi_1$ and $\phi_2$ only interact after time comparable to $R$.

Now, multiplying equation~\eqref{eq:cubicdiff} by $\p_t \psi$ and integrating by parts, we get the following basic energy estimate:
\begin{equation}
    \begin{aligned}
        \Vert \partial \psi \Vert_{L^2 (\Sigma_t)} \le 2 \int_{R - 2}^t \Vert \Box \psi \Vert_{L^2 (\Sigma_s)} \de s.
    \end{aligned}
\end{equation}
We now drop the terms containing $\psi$ (which we expect to be better behaved in terms of $R$) in the expression for $\Box \psi$ given by \eqref{eq:cubicdiff}, we are left with the interaction terms involving $\phi_1$ and $\phi_2$:
\begin{equation} \label{CubicInhomogeneity1}
    \begin{aligned}
        \int_{R - 2}^t \Vert \partial_t \phi_1 \partial_t \phi_2 (\partial_t \phi_1 + \partial_t \phi_2) \Vert_{L^2 (\Sigma_s)} \de s \le C \eps^3 \int_{R - 2}^t \frac{1}{s^2} \de s \le C \frac{\eps^3}{R}.
    \end{aligned}
\end{equation}
Here, we have used the pointwise estimates \eqref{CubicPointwise1} for two of the factors in the cubic expression, and we have used the energy estimate \eqref{CubicEnergyEst1} for the remaining factor. Thus, we expect to be able to propagate an estimate of the form
\begin{equation} \label{CubicBootstrap1}
    \begin{aligned}
        \Vert \partial \psi \Vert_{L^2 (\Sigma_t)} \le {C \eps^3 \over R}
    \end{aligned}
\end{equation}
for the nonlinear problem. This is indeed possible, and taking \eqref{CubicBootstrap1} as a bootstrap assumption along with using \eqref{goodRpointwiseest1} to control the nonlinearity allows us to close a bootstrap argument in the cubic case, yielding global stability of the trivial solution independently of the parameter $R$. We note that, in this procedure, we have estimated many derivatives of $\phi_1$ and $\phi_2$ in $L^\infty$. However, because we can consider $\phi_1$ and $\phi_2$ to be known functions with as much regularity as we want (as long as we only require boundedness in some $H^k$ space), this loss of derivatives does not constitute an obstruction.

\subsubsection{The quadratic case: trilinear estimates on null forms}\label{subsub:quadraticimp}

Returning to the quadratic case, we still expect that the solutions arising from each individual piece of initial data will start interacting only after time comparable to $R$. 
Before this interaction, we expect the solution to exist and to decay at rates that are known from the classical theory (see~\cite{Klainerman1986}).

We will decompose the data into $N$ compactly supported pieces which are localized near some $R p_i$, plus a remainder whose energy will be small in terms of $R$. The solution with data localized near each $R p_i$ will be denoted by $\phi_i$, and the solution with the remainder as initial data will be denoted by $\phi_0$. This process is formalized in Section~\ref{sub:constructinit} below.

For technical reasons, we cannot directly study the equation for the difference
$$\psi := \phi - \sum_{i = 0}^N \phi_i,$$
as was done in the cubic case. This is because the weights in $R$ prevent us from closing a bootstrap argument. The equation for $\psi$ is a nonlinear wave equation with vanishing initial data and inhomogeneous terms containing $\phi_i$. Because the initial data for $\psi$ vanishes, we again expect the behavior of $\psi$ to be dominated by the inhomogenous terms, just as in the cubic case discussed in Section \ref{subsub:cubic}. Thus, it is natural for us to study linear inhomogeneous equations such as
\begin{equation}\label{eq:psiijintro}
        \Box \psi_{ij} = 2 m (d \phi_i,d \phi_j).
\end{equation}
The RHS of equation~\eqref{eq:psiijintro} is precisely one of the inhomogeneous terms that appears in the RHS of the equation for $\psi$.
This term contains the information about the interaction between $\phi_i$ and $\phi_j$. We shall first prove estimates for $\psi_{ij}$. This involves getting estimates for certain trilinear terms.

The trilinear estimates arise from using $\partial_t \psi_{i j}$ as a multiplier in \eqref{eq:psiijintro}. Thus, the estimates involve controlling spacetime integrals like
\begin{equation}
    \begin{aligned}
        \int \int m(d \phi_i,d \phi_j) \partial_t \psi_{i j} \de x \de t,
    \end{aligned}
\end{equation}
and they show that the interaction between $\phi_i$ and $\phi_j$ is small. In general, $m$ can be any null form.

As a consequence of these trilinear estimates, we shall get improved energy estimates for $\psi_{i j}$ of the form
\begin{equation}
    \begin{aligned}
        \Vert \partial \psi_{i j} \Vert_{L^2 (\Sigma_t)} \le C {\eps^2 \over R},
    \end{aligned}
\end{equation}
which is analogous to the estimate \eqref{CubicBootstrap1} for the cubic equation.

After showing these improved estimates, we go one step further in decomposing $\phi$ by considering the equation for 
\begin{equation}\label{eq:capitalpsi}
\Psi:=\phi - \sum_{i=0}^N \phi_i - \sum_{i \ne j} \psi_{i j}.
\end{equation}
The resulting equation for $\Psi$ is a nonlinear equation with vanishing initial data and inhomogeneous terms on the RHS. Moreover, we expect $\Psi$ to have very good estimates in terms of $R$ weights. The improved estimates on $\psi_{i j}$ will allow us to control $\Psi$ in a sufficiently strong way.\footnote{We are not concerned about losing derivatives because we may take $\psi_{i j}$ to be known functions with as much regularity as we want, similar to $\phi_i$. Thus, we may estimate as many derivatives of $\psi_{i j}$ in $L^\infty$ as we want.}

We shall now describe in more detail the trilinear estimates we need for $\psi_{ij}$ and the strategy we follow to prove them. The procedure is formalized in Section~\ref{sub:improvedenergy}.

\subsubsection{Geometric tools for the trilinear estimates}\label{subsub:linearestint}
Before describing the trilinear estimates, we must pause to recall that solutions to nonlinear wave equations with localized initial data have ``good'' and ``bad'' derivatives in terms of decay. More precisely, we recall the null frame decomposition $\partial_v = \partial_t + \partial_r$, $\partial_u = \partial_t - \partial_r$, $e_A$, and $e_A$, $e_B$, where $e_A$ and $e_B$ are local orthonormal frames tangent to the spheres centered at $r = 0$. When the data are localized near $r = 0$ in $\Sigma_0$, we have that $\partial_v \phi$, $e_A \phi$, and $e_B \phi$ all decay faster in $r$ than $\partial_u \phi$ (see, for example, \cite{jonathannotes}). Thus, we call the derivatives $\partial_v$, $e_A$, and $e_B$ good derivatives, while we call the derivative $\partial_u$ the bad derivative. For simplicity, we shall use $\partial$ to represent an arbitrary derivative, while we shall use $\overline{\partial}$ to represent a good derivative.

Without loss of generality, we will take $i = 1$ and $j = 2$ in \eqref{eq:psiijintro}. Similar to the cubic case in Section~\ref{subsub:cubic}, we want to prove decay in $R$ of the energy of $\psi_{1 2}$. The $\partial_t$ energy estimate on \eqref{eq:psiijintro} is analogous to the schematic calculation in  \eqref{CubicInhomogeneity1}, as \eqref{eq:psiijintro} is driven by only inhomogeneities. However, in the quadratic case, it is not enough to use directly the $\partial_t$ energy estimate, H\"older's inequality, and the pointwise estimates for the known functions $\phi_1$ and $\phi_2$ as was done for the cubic case in estimate~\eqref{CubicInhomogeneity1}. 

This is the main difference between the quadratic case and the cubic case where the estimates \eqref{CubicPointwise1} sufficed to close a bootstrap argument. There, the inhomogeneity is cubic, and it is easy to gain ${1 \over R}$ (see equation~\eqref{CubicInhomogeneity1}). For a quadratic nonlinearity, the null condition must be used (indeed, we recall that general quadratic nonlinearities lead to blowup).

However, directly using the null condition is not sufficient because there is a region with interaction between $\partial_{u_1} \phi_1$ and $\partial_{u_2} \phi_2$. Here, $\partial_{u_i} = \partial_t - \partial_{r_i}$, and $r_i$ is the radial coordinate with $R p_i$ as its center. These are precisely the slowest decaying derivatives. Quadratic nonlinearities involving only terms that decay this slowly (specifically like ${1 \over t}$) are the obstruction to proving global existence when the null condition is not satisfied.

To see that this interaction must occur, we can consider the forward light cones centered at $R p_1$ and $R p_2$ (for simplicity, we can take $R p_1 = (-R,0,0)$ and $R p_2 = (R,0,0)$). Schematically, the interaction will then look as in Figure~\ref{fig:first}.

\begin{figure}
\def\svgwidth{\linewidth}
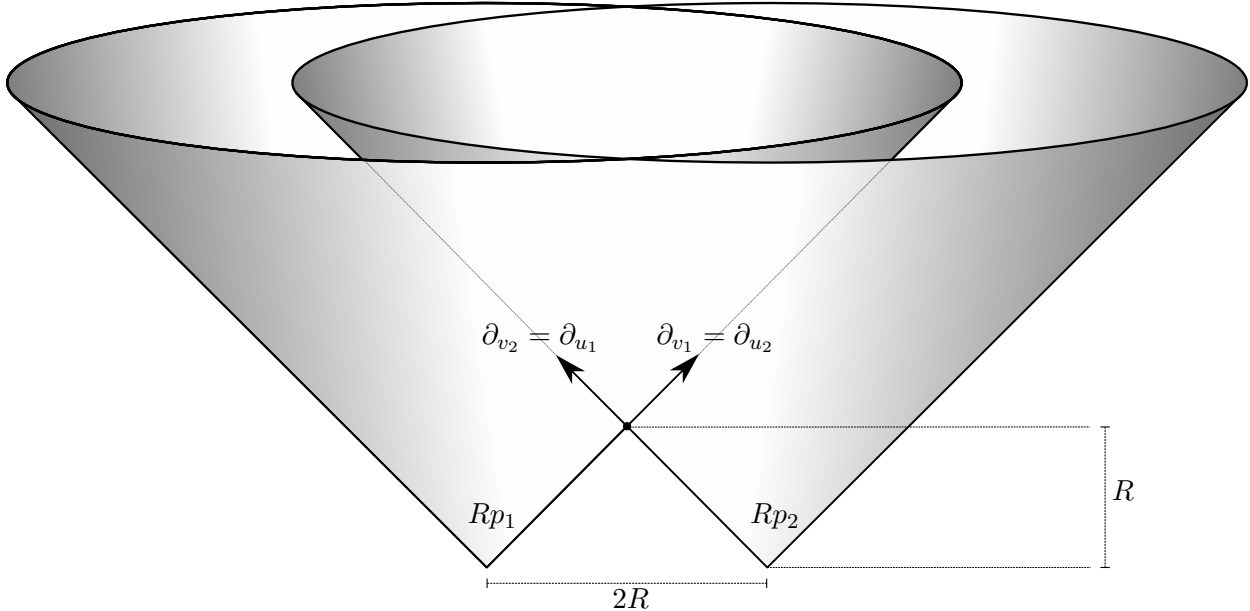
\caption{Interaction of two nonlinear waves.}
\label{fig:first}
\end{figure}

When the two light cones first intersect (which we can think of as being the first point of interaction of the distinct waves), a good derivative of $\phi_1$ is exactly the bad derivative of $\phi_2$. More precisely, we have that $\partial_{v_1} = \partial_{u_2}$ at this point (see Figure~\ref{fig:first}). Here, we are taking $\partial_{v_i} = \partial_t + \partial_{r_i}$ with $r_i$ as in the above. In terms of decay, this situation looks very similar to the nonlinearity $-(\partial_t \phi)^2$, because there is a region where the bad derivatives of $\phi_1$ and $\phi_2$ interact. Thus, in this region, the equation is similar to $\Box \phi = -(\partial_t \phi)^2$, which blows up in finite time even for arbitrarily small and compactly supported data (see \cite{John1981} and the discussion in \cite{speck2016}).

In the cubic case, we only required the observations given in the bullet points in Section~\ref{subsub:mainideas}. We shall now list the properties we must take advantage of in the quadratic case, which include three additional observations:
\begin{itemize}
\item The distinct solutions $\phi_1$ and $\phi_2$ effectively will not interact for time comparable to $R$, just as in the cubic case.
\item The distinct solutions $\phi_1$ and $\phi_2$ have sufficiently strong quantitative decay estimates, just as in the cubic case. However, we must now also use the fact that certain derivatives decay better after decomposing into a null frame. We will state and prove the necessary decay estimates for each $\phi_i$ in Lemma~\ref{prop:decphii} below.
\item The measure of the set of interaction between $\partial_{u_1} \phi_1$ and $\partial_{u_2} \phi_2$ is small in terms of $R$. This is encoded in the change of coordinates introduced in Lemma~\ref{lem:r1r2} below. Thus, using H\"older's inequality to control interactions between $\phi_1$ and $\phi_2$ is wasteful, as the two functions are supported on sets with small intersection.
\item The good derivatives of $\phi_1$ and $\phi_2$ become asymptotically aligned. This fact is formalized in Proposition~\ref{lem:goodbad} below.
\item The vector field $\partial_t$ can be written in terms of the good derivatives of $\phi_1$ and $\phi_2$ in the region in which the interaction between $\phi_1$ and $\phi_2$ is largest. Even though the coefficients of this decomposition degenerate as $t \rightarrow \infty$ as a consequence of the fact that the good derivatives of $\phi_1$ and $\phi_2$ become asymptotically aligned, it still gives an improvement. This is also done in Proposition~\ref{lem:goodbad} below.
\end{itemize}

Using these facts will allow us to get good energy estimates for $\psi_{ij}$ in terms of $R$ weights. These estimates and the associated trilinear estimates are established in Section \ref{sub:improvedenergy}. We can then combine these good energy estimates with the modified Klainerman--Sobolev estimates in Section \ref{sec:sobolevs} to get good pointwise estimates on $\psi_{i j}$. This is carried out in Section \ref{sub:linfpsiij}.

\subsubsection{Closing the argument} \label{subsub:closing}
The linear estimates for $\psi_{ij}$ described above are enough for us to prove global existence for the remainder $\Psi$ (see equation~\eqref{eq:capitalpsi}) using the vector field method as in, for example, \cite{Klainerman1985}, \cite{Klainerman1986}, \cite{LR2010}, and \cite{jonathannotes}.  
Moreover, $\Psi$ satisfies better energy bounds than $\psi_{ij}$ in terms of $R$. Indeed, schematically, the $\partial_t$ energy of $\Psi$ will be of size $\eps^3 R^{-\frac 32}$. This implies that the $\partial_t$ energy of $\Gamma_R^\alpha \Psi$ will be of size $\eps^3 R^{-{3 \over 2}}$, where the vector fields $\Gamma_R$ are as in Section \ref{subsub:rweightedintro}. Applying \eqref{goodRpointwiseest1} to $f = \partial \Gamma_R^\alpha \Psi$ and noting that the resulting norms on the right hand side of this estimate should be of size $\eps^3 R^{-{3 \over 2}}$ tells us that we should expect decay estimates like
\begin{equation} \label{goodRpointwiseest2}
    \begin{aligned}
        |\partial \Gamma_R^\alpha \Psi| (t,r,\omega) \le C \eps^3 {1 \over R^{{1 \over 2}}} {1 \over t + r}.
    \end{aligned}
\end{equation}
Moreover, using \eqref{eq:ksobintro}, we are able to obtain improved decay of the good derivatives like
\begin{equation} \label{goodtpointwiseest}
    \begin{aligned}
        |\overline{\partial} \Gamma_R^\alpha \Psi| (t,r,\omega) \le C \eps^3 {R^{{3 \over 2}} \over t^{{3 \over 2}}}.
    \end{aligned}
\end{equation}

These bounds are enough to prove global existence using a bootstrap argument for $\Psi$. We shall propagate the bootstrap assumption that the $\partial_t$ energy of $\Gamma_R^\alpha \Psi$ for all $|\alpha|$ sufficiently large is schematically of size $\eps^3 R^{-{3 \over 2}}$. We shall have to use the following observation. If we are interested in controlling nonlinear effects between time $R$ and $R^{20}$, we note that better pointwise $R$ decay is almost as good as better $t$ decay. This is because, evaluating the integrals,
\begin{equation*}
    \begin{aligned}
        \int_R^{R^{20}} {1 \over R^\delta t} \de t \le C {\log{(R)} \over R^\delta},
    \end{aligned}
\end{equation*}
while
\begin{equation*}
    \begin{aligned}
        \int_R^{R^{20}} {1 \over t^{1 + \delta}} \de t \le {C_\delta \over R^{\delta}}.
    \end{aligned}
\end{equation*}
These two expressions differ only by a factor of $\log{R}$, which we will be able to absorb using terms with extra negative powers of $R$. Thus, for this time scale, we can use the estimates \eqref{goodRpointwiseest2} which only have an improvement in $R$. 

Then, from time $R^{20}$ to $\infty$, we use the estimates \eqref{goodtpointwiseest}, which have worse $R$ weights. This can be done because the fact that we are integrating from $R^{20}$ will get rid of any bad $R$ weights. More precisely, for $\alpha \in \R$ we have that
\begin{equation}
    \begin{aligned}
        \int_{R^{20}}^\infty {R^\alpha \over t^{{3 \over 2}}} d t \le C R^{\alpha - 10}.
    \end{aligned}
\end{equation}
In practice, we will have $\alpha$ about $0$, so this estimate is strong enough to propagate an energy estimate of size $\eps^3 R^{-{3 \over 2}}$. This allows us to prove global existence for $\Psi$. After this, we have that $\Psi + \sum_{i \ne j} \psi_{i j} + \sum_i \phi_i$ is a solution to the original problem, as desired. The argument establishing the global existence and decay rates of $\Psi$ is carried out in Section \ref{sec:mainproof}.

\section{Setup}\label{sec:setup}

\subsection{Coordinate systems} \label{csystems}
We consider $(\R^{3 + 1},m)$ parametrized by the usual coordinates $(t,x,y,z)$, where $m$ is the Minkowski metric. We shall always use $\Sigma_s$ to denote the affine hyperplane with $t$-coordinate equal to $s$.

Let $\Pi := \{p_1, \ldots, p_N\}$ be a collection of points in $\R^3$ satisfying the following properties:
\begin{align*}
    \min_{\substack{i\neq j \\ i, j \in \{1, \ldots, N\}}} |p_i - p_j|\geq 2.
\end{align*}
We also define
\begin{equation}\label{eq:dpidef}
    \begin{aligned}
        d_\Pi := \Big(\max_{i,j} |p_i - p_j|\Big) \Big/ \Big(\min_{\substack{i\neq j \\ i, j \in \{1, \ldots, N\}}} |p_i - p_j|\Big).
    \end{aligned}
\end{equation}
Let $R \geq 1$ be a parameter. We will consider initial data centered around the points $R \cdot p_i$, on the initial surface $\Sigma_0$. Here, $|w|$ denotes the Euclidean length of $w \in \R^3$ as measured in $\R^3$. We recall that we prove gobal existence for initial data whose size (measured in a suitable higher-order Sobolev norm) depends only on $N$ and $d_\Pi$, and not on $R$.

For every $i \in \{0, 1, \ldots, N\}$, we consider polar coordinates adapted to the point $Rp_i$. Let us furthermore assume that $p_0 = 0$. The polar coordinates $(r_i, \theta_i, \varphi_i)$ satisfy the following relations, for all $w \in \R^3$:
\begin{equation}\label{eq:polarpi}
    w-Rp_i = r_i \left(
    \begin{array}{c}
    \cos \theta_i\\
    \cos\varphi_i \sin \theta_i\\
    \sin \varphi_i \sin \theta_i
    \end{array}
    \right).
\end{equation}
We adopt the following convention: 
\begin{equation}\label{eq:rthephi}
    r:= r_0, \quad  \theta := \theta_0, \quad \varphi := \varphi_0.
\end{equation} 
In particular, in coordinates $(r, \theta, \varphi)$, the set $\theta = 0$ corresponds to the $x$-axis (this is an abuse of notation as such coordinates break down at $\theta = 0$).

These polar coordinates induce the usual null coordinates with outgoing and incoming hypersurfaces.
\begin{definition}\label{def:riviui}
Consider $\R^{1+3}$ parametrized by the usual coordinate system $(t,x,y,z)$, and by the coordinates $(t,r_i, \theta_i, \varphi_i)$, where $(r_i, \theta_i, \varphi_i)$ have been introduced above. Then, the following defines a set of null coordinates for all $i \in \{0, \ldots, N\}$:
\begin{equation}
    u_i := t - r_i, \qquad v_i := t + r_i , \qquad \theta_i, \qquad \varphi_i.
\end{equation}
In the case $i=0$, these reduce to the usual null coordinates centered at the origin. Furthermore, in this case the set corresponding to $\theta =0$ in $(r, \theta, \phi)$ coordinates is the positive half of the $x$-axis.
\end{definition}

For every $i, j \in \{0,\ldots, N\}$, $i \neq j$, we know that $Rp_i \neq Rp_j$. Therefore, up to a translation and a rotation, we can always assume that the point $Rp_i$ has $(x,y,z)$-coordinates $(-R \frac 12 |p_i - p_j|,0,0)$, and that the point $Rp_j$ has $(x,y,z)$-coordinates $(R \frac 12 |p_i - p_j|,0,0)$. 

We then define cylindrical coordinates on $\R^3$ $(x, \rho, \varphi)$ adapted to the $x$-axis, satisfying the following relations:
\begin{equation}\label{eq:polarcoord}
    x, \qquad y = \rho \cos \varphi, \qquad z = \rho \sin \varphi. 
\end{equation}

\begin{remark}
    Note that the coordinate $\varphi$ introduced here coincides with the $\varphi$ introduced before in~\eqref{eq:rthephi}.
\end{remark}

We now take coordinates on all of $\R^{1 + 3}$ that are adapted to hypersurfaces which are hyperboloids in two directions and flat in the third direction. These coordinates will be denoted by $(\tau,\alpha,x,\varphi)$, and they satisfy the relations:
\begin{equation}\label{eq:xflathyp}
t = \tau \cosh{(\alpha)}, \qquad x = x, \qquad y = \tau \sinh{(\alpha)} \cos{(\varphi)}, \qquad z = \tau \sinh{(\alpha)} \sin{(\varphi)}. 
\end{equation}

\begin{remark}
    Again, the coordinate $\varphi$ defined here coincides with the $\varphi$ introduced in~\eqref{eq:rthephi}.
\end{remark}

\begin{definition}\label{def:hypdef}
For some $\bar \tau \geq 0$, we let $H_{\bar \tau}$ be the hypersurface which, in the coordinates defined by~\eqref{eq:xflathyp}, is defined as
\begin{equation}\label{eq:htdef}
	H_{\bar \tau} := \{\tau = \bar \tau \}.
\end{equation}
\end{definition}
\begin{remark}
Geometrically, these surfaces are just the unit hypersurface (i.e., the hypersurface defined by $\tau = 1$) scaled by a factor of $\tau$ with respect to the origin. They can also be seen to be the hypersurfaces found by intersecting the light cone $t = r$ with the hyperplane $x = \tau$ and translating the resulting two dimensional surface in the $x$-direction.
\end{remark}

\subsection{Vector fields and commutation}\label{sec:vectors}
We shall now introduce notation for the vector fields we shall use. This includes notation for the vector fields adapted to each point $p_i$, adapted to pairs of points $p_i$ and $p_j$, and the appropriate rescalings by the parameter $R$.

Let $i \neq j$, $i,j \in \{1, \ldots, N\}$. Let $O_{ij}$ be the unique isometry of $\R^3$ (composition of a translation and a rotation in $\R^3$) such that the following holds:
\begin{equation*}
    O_{ij}\Big(-  \frac{|p_i-p_j|}{2}, 0,0 \Big) =  p_i, \qquad  O_{ij}\Big(  \frac{|p_i-p_j|}{2}, 0,0 \Big) =  p_j,
\end{equation*}
Using this transformation, we can without loss of generality assume that $p_i =(-  \frac{|p_i-p_j|}{2}, 0,0 ) $, $p_j = (  \frac{|p_i-p_j|}{2}, 0,0)$.
We now let $\Gamma^{(h)}_{(w_1w_2)}$, $h \in \{i,j\}$ be the following vector fields:
\begin{equation}\label{eq:gammasdef}
\begin{aligned}
&\text{Rotation vector fields}
&\begin{cases}
&\Gamma^{(0)}_{(x w_1)} := x \p_{w_1} - w_1 \p_{x}, \quad \text{ if } w_1 \in \{y,z \},\\
&\Gamma^{(i)}_{(x w_1)} := \big(x+ R \frac{ |p_i - p_j|} 2\big) \p_{w_1} - w_1 \p_{x}, \quad \text{ if } w_1 \in \{y,z \},\\
&\Gamma^{(j)}_{(x w_1)} := \big(x-R \frac{ |p_i - p_j|} 2\big) \p_{w_1} - w_1 \p_{x}, \quad \text{ if } w_1 \in \{y,z \},\\
&\Gamma^{(i)}_{(yz)} = \Gamma^{(j)}_{(yz)} = \Gamma^{(0)}_{(yz)} := y \p_{z} - z \p_{y},\\
\end{cases}\\
&\text{Lorentz boosts}
&\begin{cases}
&\Gamma^{(0)}_{(tx)} := x\p_{t}+ t\p_{x},\\
&\Gamma^{(i)}_{(tx)} := \big(x+ R \frac{ |p_i - p_j|} 2\big)  \p_{t}+ t\p_{x},\\
&\Gamma^{(j)}_{(tx)} :=\big(x- R \frac{ |p_i - p_j|} 2\big)  \p_{t}+ t\p_{x}, \\
&\Gamma^{(0)}_{(tw_1)} = \Gamma^{(i)}_{(tw_1)}  = \Gamma^{(j)}_{(tw_1)} := w_1 \p_{t}+ t\p_{w_1}, \text{ if } w_1 \in \{y,z\} \\
\end{cases}\\
&\text{Scaling vector fields}
&\begin{cases}
&S^{(0)} := x\p_x + y \p_y + z\p_z + t \p_t,\\
&S^{(i)} := \big(x+ R \frac{ |p_i - p_j|} 2\big) \p_x + y \p_y + z\p_z + t \p_t,\\
&S^{(j)} := \big(x- R \frac{ |p_i - p_j|} 2\big) \p_x + y \p_y + z\p_z + t \p_t.\hspace{49pt}\\
\end{cases}
\end{aligned}
\end{equation}
We also rename the rotation vector fields as follows:
\begin{equation}\label{eq:omegadef}
\Omega^{(h)}_{(ab)}:= \Gamma^{(h)}_{(ab)}, \qquad \text{ for all } a,b \in \{x,y,z\},\quad h \in \{0,1,2\}.
\end{equation}
We furthermore define the set of all translations and spatial translations as follows:
\begin{equation}\label{eq:tidef}
\boldsymbol{T} := \{\p_x, \p_y, \p_z, \p_t \}, \qquad \boldsymbol{T}_s := \{\p_x, \p_y, \p_z\}.
\end{equation}
We also define the set
\begin{equation}\label{eq:kcenter}
    \boldsymbol{K} := \Ti \cup \{S^{(0)}\} \cup \bigcup_{\substack{w_1, w_2 \in \{t,x,y,z\}\\ w_1 \neq w_2}} \{\Gamma^{(0)}_{(w_1 w_2)}\} .
\end{equation}
Note that this is the usual set of all Killing vector fields of Minkowski space along with the scaling vector field.

Similarly, we have the set of all Killing vector fields based at $Rp_h$, for $h \in \{i,j\}$:
\begin{equation}\label{eq:kh}
    \boldsymbol{K}^{(h)} := \Ti \cup \{S^{(h)}\}\cup \bigcup_{w_1, w_2 \in \{x,y,z\}} \{\Gamma^{(h)}_{(w_1 w_2)}\}.
\end{equation}
Let us furthermore define the ``good'' vector fields
\begin{align}
	\boldsymbol{\Gamma} :=  \{\Gamma^{(0)}_{(yz)} , \Gamma^{(0)}_{(tz)}, \Gamma^{(0)}_{(ty)} \}.
\end{align}
This is the set of vector fields which do not introduce weights on initial data which is localized around the two points $Rp_i$ and $Rp_j$. Let us further define the $R$-weighted Lorentz fields adapted to either the piece of data localized around $Rp_i$ or $Rp_j$ as follows:
\begin{align}\label{eq:gammah}
	& \boldsymbol{\Gamma}^{(h)} := \boldsymbol{\Gamma} \cup \Ti \cup \bigcup_{w_1, w_2 \in \{t,x,y,z\}} \{R^{-1} \Gamma^{(h)}_{(w_1w_2)} \} \cup \{R^{-1} S^{(h)}\}, \quad \text{ for } h \in \{0,i,j\}.
\end{align}
Also, we define the good derivatives adapted to the $h$-th light cone:
\begin{equation}\label{eq:goodcone}
\boldsymbol{G}^{(h)} := \Big\{\p_{v_h}, \frac 1{1+r_h} \Gamma^{(h)}_{(xy)}, \frac 1{1+r_h} \Gamma^{(h)}_{(xz)}, \frac 1{1+r_h} \Gamma^{(h)}_{(yz)} \Big\}, \quad \text{ for } h \in \{0,1,2 \}.
\end{equation}
Here, we used the coordinate vector field $\p_{ v_h}$ induced by null coordinates $(u_h, v_h, \theta_h, \varphi_h)$ adapted to $Rp_h$, and defined as in equation~\eqref{eq:polarpi}. Note that $r_h = \frac 12 (v_h + u_h)$. We also note that the last three vector fields in~\eqref{eq:goodcone} are rotations, and their span is two-dimensional everywhere (they are tangent to the spheres of constant $r_h$-coordinate).

We then define the sets of $R$-normalized rotations with respect to the $h$-th light cone as follows:\vspace{-15pt}
\begin{equation}\label{eq:omegasetdef}
	\boldsymbol{\Omega}^{(h)} := \{\Omega^{(h)}_{(xy)}, \Omega^{(h)}_{(xz)}, \Omega^{(h)}_{(yz)} \} 	\qquad \boldsymbol{\Omega}_R^{(h)} := \Big\{\Omega_{(yz)},\frac{\Omega^{(h)}_{(xy)}}{R},\frac{\Omega^{(h)}_{(xz)}}{R}\Big\}, \quad \text{with}\quad h \in \{0,i,j\}.
\end{equation}

We define the set of all $R$-renormalized Minkowski Killing fields (plus scaling) based at the origin as follows:
\begin{equation}\label{eq:krcenter}
    \boldsymbol{K}_R := \Ti \cup \bigcup_{w_1, w_2 \in \{t,x,y,z\}} \{R^{-1} \Gamma^{(0)}_{(w_1w_2)} \} \cup \{R^{-1} S^{(0)}\}.
\end{equation}
And the similar set for the $h$-th light cone, with $h \in \{i,j\}$:
\begin{equation}\label{eq:krh}
    \boldsymbol{K}^{(h)}_R := \Ti \cup \bigcup_{w_1, w_2 \in \{t,x,y,z\}}\{R^{-1} \Gamma^{(h)}_{(w_1w_2)} \} \cup \{R^{-1} S^{(h)}\}.
\end{equation}

We now define a shorthand notation for ``good'' and ``bad'' derivatives.
\begin{definition}\label{def:shorthand}
	Let $\zeta \in \mathcal{C}^\infty(\R^4)$. We define
	\begin{equation}
	|\p \zeta| := \sum_{B \in \boldsymbol{T}} |B \zeta|, \qquad |\bar \p^{(h)} \zeta| := \sum_{G \in \boldsymbol{G}^{(h)}} |G \zeta|,
	\end{equation}
	so that the former notation indicates bad derivatives (all unit derivatives are included in this formula), and the latter comprises \emph{only} good derivatives adapted to the $h$-th light cone.
\end{definition}

We now define multi-indices.
\begin{definition}[Multi-index]\label{def:multiindex}
	Let $\BA$ be a set of vector fields, and $m \in \N$, $m \geq 0$. We let $I^{m}_{\BA}$ to be the set of ordered lists of $m$ elements of $\BA$. We furthermore let
	\begin{equation*}
		I^{\leq m}_{\BA} := \bigcup_{m_1 \in \{0, \ldots, m\}} I^{m_1}_{\BA}.
	\end{equation*}
\end{definition}

\begin{definition}\label{def:derivation}
	Let $m \in \N$, $m \geq 0$, and let $\BA$ be a set of vector fields. Let $f\in \mathcal{C}^\infty(\R^4)$, and let $I \in I^m_{\BA}$, such that
	\begin{equation*}
		I = (V_1, V_2, \ldots, V_m), \text{ with } V_i \in \BA \ \forall \ i \in \{1,\ldots, m\}.
	\end{equation*}
	We then define the derivative of $f$ by the multi-index $I$ as follows:
	\begin{equation}
	\der^I f := V_1 \, V_2 \, \cdots \, V_m f.
	\end{equation}
\end{definition}
\begin{definition}\label{def:indexsum}
	Let $m_1, m_2 \in \N$, and let $I \in I^{m_1}_{\BA}$, $J \in I^{m_2}_{\BA}$. We say that $K \in I^{m_1 + m_2}_{\BA}$ satisfies
	$$
	K = I + J,
	$$
	if $K$ can be decomposed into two disjoint ordered lists $K_1$ and $K_2$, such that $K = K_1 \cup K_2$ as sets (counting multiplicity) and, in addition, $I = K_1$, $J = K_2$, where the last two equations are understood in the sense of ordered lists, i.~e.~taking into account multiplicity and ordering.
\end{definition}

\begin{remark}
	Note that, with this definition, the sum of $I$ and $J$ is not unique.
\end{remark}
\begin{definition}
	We also define inclusion between multi-indices in the following way. We say that $I \subset J$ if $I \in I^{m_1}_{\BA}$, $J \in I^{m_2}_{\BA}$, with $m_1, m_2 \in \N$, $ 0  \leq m_1 \leq m_2$, and if $I$, as a list, is obtained from $J$ by removing $m_2 - m_1$ elements of $J$ (and preserving the ordering). 
\end{definition}

\begin{lemma}[Leibniz rule]
	Let $\BA$ be a set of vector fields, and let $f$, $g$ be smooth functions, $m \in \N$, $m \geq 0$. Let $K \in I^{m}_{\BA}$. We then have
	\begin{equation}
	\der^K (fg) = \sum_{I+J = K} \der^I f \der^J g,
	\end{equation}
	where the sum is over all $I \in I^{m_1}_{\BA}$, $J \in I^{m_2}_{\BA}$ such that $m_1, m_2 \in \N$, $m_1, m_2 \geq 0$,  $m_1 + m_2 = m$, and furthermore $I + J = K$ in the sense of Definition~\ref{def:indexsum}. 
\end{lemma}

\begin{definition}[Shorthand notation for iterated derivatives]
    Let $\eta \in \mathcal{C}^\infty(\R^4)$, and let $m, k$ be non-negative integers. We then define
    \begin{equation}
        |\partial^{k} f| : = \sum_{I \in I^{k}_{\boldsymbol{T}}} |\der^I f|, \qquad \text{where} \qquad \boldsymbol{T} := \{\p_x, \p_y, \p_z, \p_t\}.
    \end{equation}    
    Also,
    \begin{equation}
        |\partial^{\leq m} f| := \sum_{k=0}^m |\partial^k f|.
    \end{equation}
    Similarly, for spatial derivatives only, we define:
    \begin{equation}
        |\widehat{\partial}^{k} f| : = \sum_{I \in I^{k}_{\boldsymbol{T}_s}} |\der^I f|, \qquad \text{where} \qquad \boldsymbol{T}_s := \{\p_x, \p_y, \p_z\}.
    \end{equation}
\end{definition}

\subsection{Classical null forms}  \label{sec:null}
\begin{definition}\label{def:null}
	Let $k$ be a positive integer. We say that a $k$-linear form with constant coefficients $F: T\R^4 \times \ldots \times T\R^4 \to \R$ is a classical null form if, for every null vector $\xi$ with respect to the Minkowski metric, we have that 
	\begin{equation}\label{eq:nullcond}
	F(\xi, \ldots, \xi) = 0.
	\end{equation}
\end{definition}
We recall the following lemma about null forms and commutation.
\begin{lemma}[\cite{hormanderbook}, Lemma 6.6.5]\label{lem:nullhorm}
	Let
	\begin{equation*}
	    \Gamma \in \boldsymbol{K} \cup  \boldsymbol{K}^{(i)} \cup  \boldsymbol{K}^{(j)}.
	\end{equation*}
	Let also $F$ be a classical trilinear null form with components $F_{\alpha\beta\gamma}$, and $G$ be a classical bilinear null form with components $G_{\alpha\beta}$, both  understood as per Definition~\ref{def:null}. Then, we have, for functions $\eta, \zeta \in \mathcal{C}^\infty(\R^4)$,
	\begin{equation}\label{eq:commutfirst}
	\begin{aligned}
	&\Gamma (F_{\alpha\beta\gamma} \p^\alpha \eta \, \p^\beta \p^\gamma \zeta) =  F_{\alpha\beta\gamma} \p^\alpha \Gamma \eta \, \p^\beta \p^\gamma \zeta + F_{\alpha\beta\gamma} \p^\alpha \eta \, \p^\beta \p^\gamma \Gamma\zeta + F'_{\alpha\beta\gamma} \p^\alpha \eta \, \p^\beta \p^\gamma \zeta,\\
	&\Gamma (g_{\alpha\beta} \p^\alpha \eta \, \p^\beta \zeta) =  G_{\alpha\beta} \p^\alpha \Gamma \eta \, \p^\beta \zeta + G_{\alpha\beta} \p^\alpha \eta \, \p^\beta  \Gamma\zeta + G'_{\alpha\beta} \p^\alpha \eta \, \p^\beta \zeta.
	\end{aligned}
	\end{equation}
	where $F', G'$ are also classical null forms in the sense of Definition~\ref{def:null}.
\end{lemma}
\begin{remark}
In what follows, to simplify notation, we will denote a classical trilinear null form $F$ with components $F_{\alpha\beta\gamma}$ acting on two functions $\eta, \zeta$ by
$$
F(d \eta, d^2 \zeta) := F_{\alpha\beta\gamma} \p^\alpha \eta \p^\beta \p^\gamma \zeta
$$
($d^2$ indicates that the null form is acting on the hessian of $\zeta$).
Similarly, we will denote the action of a bilinear null form $G$ with components $G_{\alpha \beta}$ on two functions $\eta$, $\zeta$ by
$$
G(d \eta, d \zeta) := G_{\alpha\beta} \p^\alpha \eta \p^\beta  \zeta.
$$
\end{remark}
Iterating Lemma~\ref{lem:nullhorm}, we can prove the following
\begin{lemma}[Null forms and commutation]\label{lem:nullcomm}
	Let $K \in I^m_{\boldsymbol{K} \cup  \boldsymbol{K}^{(i)} \cup  \boldsymbol{K}^{(j)}}$, with $m \in \N$, $m \geq 0$. Let also $F$ be a trilinear null form, and $G$ be a bilinear null form.
	In these conditions, for all smooth functions $\eta, \zeta \in \mathcal{C}^\infty(\R^4)$, we have
	\begin{align}
	&\der^K F(d \eta, d^2 \zeta) = \sum_{I+J \subset K } F_{IJ}(d \, \der^I \eta, d^2 \, \der^J \zeta),\\
	&\der^K G(d \eta, d \zeta) = \sum_{I+J \subset K } G_{IJ}(d\, \der^I \eta, d \, \der^J \zeta).
	\end{align}
	Here, every $F_{IJ}$ is a trilinear null form as in Definition~\ref{def:null}, and every $G_{IJ}$ is a bilinear null form as in Definition~\ref{def:null}. The sum is taken over all $I, J$ such that $I \in I^{m_1}_{\boldsymbol{\Gamma}^{(i)}}$, $J \in I^{m_1}_{ \boldsymbol{\Gamma}^{(i)}}$,  $m_1 + m_2 \leq m$, $I + J\subset K$.
\end{lemma}
\begin{proof}[Proof of Lemma~\ref{lem:nullcomm}]
	The proof follows from Lemma~\ref{lem:nullhorm} and an induction argument.
\end{proof}
We now recall a lemma on the structure of null forms.
\begin{lemma}[Structure of null forms]\label{lem:nullstructpre}
	Let $F,G$ be resp.~a classical trilinear null form and a classical bilinear null form in the sense of Definition~\ref{def:null}. Recall the definition of the good derivatives adapted to the $h$-th light cone in equation~\eqref{eq:goodcone}. Then, there exists a positive constant $C >0$ such that the following holds. Let $\zeta, \eta \in \mathcal{C}^\infty(\R^4)$. We have the pointwise inequality
	\begin{align}
	&|G(\de \zeta, \de \eta)| \leq C (|\bar \p^{(h)} \zeta| |\p \eta| + |\p \zeta| |\bar \p^{(h)} \eta|),\\
	&|F(\de \zeta, \de^2 \eta)| \leq C (|\bar \p^{(h)} \zeta| |\p^2 \eta| + |\p \zeta| | \bar \p^{(h)} \p \eta|).
	\end{align}
	Recall the expressions $|\bar \p^{(h)} \zeta|$ and $|\p \zeta|$ from Definition~\ref{def:shorthand}, and that $h \in \{i,j\}$.
\end{lemma}

\begin{proof}
    The proof is straightforward writing the expressions for $F$ and $G$ in the null frame and making use of condition~\eqref{eq:nullcond} with $\xi = \p_{u_h}$.
\end{proof}

Combining Lemma~\ref{lem:nullcomm} and Lemma~\ref{lem:nullstructpre}, we obtain the following:
\begin{lemma}[Fundamental null form inequality]\label{lem:nullstruct}
	Let $F,G$ be resp.~a classical trilinear null form and a classical bilinear null form in the sense of Definition~\ref{def:null}. Let also $R \geq 1$.
	We define the set of vector fields
	\begin{equation*}
	    \boldsymbol{Z} := \boldsymbol{K} \cup \boldsymbol{K}_R \cup \bigcup_{h \in \{i,j\}}\big( \Ga^{(h)} \cup \boldsymbol{K}^{(h)} \cup \boldsymbol{K}^{(h)}_R\big).
	\end{equation*}
	Let $m \in \N$, $m \geq 0$.
	Let furthermore $I$ be a multi-index in $I^{\leq m}_{\boldsymbol{Z}}$.
	Then, there exists a positive constant $C >0$ such that the following holds. Let $\zeta, \eta \in \mathcal{C}^\infty(\R^4)$. We have the pointwise inequalities
	\begin{align}
	&|\der^I G(\de \zeta, \de \eta)| \leq C \sum_{H+K \subset I} (|\bar \p^{(i)} \der^H \zeta| |\p \der^K \eta| + |\p \der^H \zeta| |\bar \p^{(i)} \der^K \eta|),\\
	&|\der^I F(\de \zeta, \de^2 \eta)| \leq C \sum_{H+K \subset I} (|\bar \p^{(i)}\der^H \zeta| |\p^2\der^K\eta| + |\p \der^H \zeta| | \p \bar \p^{(i)} \der^K \eta|).
	\end{align}
	Here, we used the expressions $|\bar \p^{(i)} \zeta|$ and $|\p \zeta|$ from Definition~\ref{def:shorthand}.
\end{lemma}

\begin{proof}[Proof of Lemma~\ref{lem:nullstruct}]
    The proof for elements of $\boldsymbol{K}$ and $\boldsymbol{K}^{(h)}$ is evident. Regarding the remaining elements in the set $\boldsymbol{Z}$, we note that the condition $R \geq 1$ suffices to bound the terms $F'$ and $G'$ arising from the application of equation~\eqref{eq:commutfirst}, as these terms will acquire an $R^{-1}$ factor in the commutation.
\end{proof}

\section{Precise statement of the main theorems}\label{sec:statements}

\subsection{Main theorem} We proceed to state the main theorem of the present paper.

\begin{theorem}[Nonlinear wave equations with null condition and multi-localized initial data]\label{thm:main}
	Let $M$, $N$, $n_0$ be non-negative integers, such that $n_0 \geq 19$. Let $\{F^A_{BC}\}_{A,B,C = 1, \ldots, M}$ and $\{G^A_{BC}\}_{A,B,C = 1, \ldots, M}$ be resp.~a collection of trilinear and bilinear classical null forms, i.~e.~each of the $F^A_{BC}$ and $G^A_{BC}$ is as in Definition~\ref{def:null} and satisfies:
	$$
	F^A_{BC}: T\R^4 \times T\R^4 \times T\R^4 \to \R, \qquad 
	G^A_{BC}: T\R^4 \times T\R^4 \to \R.
	$$
	Let furthermore $\Pi := \{p_1, \ldots, p_N\}$ be a collection of $N$ points in $\R^3$, with $d_\Pi$ the ratio of the largest and smallest distances between them as defined in display~\eqref{eq:dpidef}.
	
	Then, there exists $\varepsilon_0 = \varepsilon_0(d_\Pi, \{F^A_{BC}\}_{A,B,C = 1, \ldots, M}, \{G^A_{BC}\}_{A,B,C = 1, \ldots, M},N) > 0$ such that the following holds true for all $R \geq 0$ and for all $0< \varepsilon < \varepsilon_0$.
	
	Consider a collection of functions 
	\begin{equation}
	(\bar \phi_{A,i}^{(0)},\bar \phi_{A,i}^{(1)})_{ A \in \{1, \ldots, M\}, i \in \{1, \ldots, N\}}
	\end{equation}
	which satisfy the following bounds, in the usual angular coordinates in $\R^3$:
	\begin{equation}\label{eq:boundsbumps}
	    |\widehat{\partial}^{k} \bar \phi_{A,i}^{(0)}|\leq \varepsilon (1+r)^{-k-2}, \qquad |\widehat{\partial}^{k} \bar \phi_{A,i}^{(1)}|\leq \varepsilon (1+r)^{-k-3},
	\end{equation}
	for all $0 \leq k \leq n_0$, and  $A \in \{1, \ldots, M\}$.
	Let us then construct initial data as follows. Let $w_i$ be the point $R \cdot p_i$. Let us furthermore define, for all $A \in \{1, \ldots, M\}$,
	\begin{equation}
	    \begin{aligned}
	    \phi_A^{(0)} := \sum_{i=1}^N \bar \phi_{A,i}^{(0)}(x-w_i),\qquad
	    \phi_A^{(1)} := \sum_{i=1}^N \bar \phi_{A,i}^{(1)}(x-w_i).
	    \end{aligned}
	\end{equation}
	Let us consider the initial value problem given by the following system of quasilinear wave equations with the specified data:
	\begin{equation}\label{eq:bumps}
	\begin{aligned}
	    &\Box \phi_A + \sum_{B, C= 1}^M F^{A}_{BC} (d \phi_B, d^2 \phi_C) = \sum_{B,C = 1}^M G^A_{BC} (d \phi_B, d \phi_C), \quad A = 1, \ldots, M,\\
	    & \phi_A|_{t=0} = \phi^{(0)}_A, \quad A = 1, \ldots, M,\\
	    & \p_t \phi_A|_{t=0} = \phi^{(1)}_A, \quad A = 1, \ldots, M.
	\end{aligned}
	\end{equation}
	Then, the initial value problem~\eqref{eq:bumps} admits a global-in-time solution $\phi_A$, which furthermore decays with quantitative rates. Thus, the trivial solution to~\eqref{eq:bumps} is asymptotically stable under this class of non-localized perturbations uniformly in the scale $R$.
\end{theorem}

\begin{remark}
    We note that the global solution $\phi_A$ constructed in the previous theorem moreover satisfies uniform (in $R$) energy estimates and quantitative decay estimates obtained by combining the inequalities in the statements of Lemma~\ref{prop:decphii}, Theorem~\ref{thm:linear} and Theorem~\ref{thm:nonlinear}.
\end{remark}

\begin{remark}
    We note that we require pointwise bounds on 19 derivatives of the initial data. The proof of Theorem~\ref{thm:main} could be optimized in terms of the number of derivatives, but we are not interested in such issues here. Concerning the calculation of the number of derivatives, see Remark~\ref{rmk:numberder}.
\end{remark}

From now on, to simplify notation, we will specialize our discussion to the case in which we have one single equation, instead of a system of equations, as the two proofs are the same, and no conceptual element is introduced in the proof for systems. We will therefore restrict our attention to a single equation of the type
\begin{equation} \label{squaseq}
    \Box \phi + F(d \phi, d^2 \phi) = G (d \phi, d \phi),
\end{equation}
where $F$ is a trilinear null form as in Definition~\ref{def:null} and $G$ is a bilinar null form as in Definition~\ref{def:null}. In the case of a single equation, $G$ is necessarily a constant multiple of the Minkowski metric $m$, although the theorem holds for more general $G$ (arising from systems of equations).

We now proceed to state the theorems used in the proof of the main Theorem~\ref{thm:main}.

\subsection{Statement of the auxiliary theorems}
\begin{theorem}[Interaction of two localized pieces of initial data]\label{thm:linear}
    There exists a constant $C >0 $ such that the following holds. Let  $0 < \varepsilon < \varepsilon_0$, where $\varepsilon_0$ is as in the statement of Theorem~\ref{thm:main}. Let $\psi_{ij}$ be a solution to the equation
    \begin{equation}\label{eq:linearizeddiff}
       \begin{aligned}
        &\Box \psi_{ij} + F (d \phi_i, d^2 \phi_j) =  G (d \phi_i, d \phi_j),\\
        & \psi_{ij}|_{t=0} = 0,\\
        & \p_t \psi_{ij}|_{t=0} = 0.
    \end{aligned}
    \end{equation}
    Here, $\phi_i$, $i,j \in \{1, \ldots, N\}$, are as constructed in Lemma~\ref{lem:decomposition} (note that in such lemma we require bounds on $n+7$ derivatives of initial data).
    
    For simplicity, let us suppose that $i =1$, $j =2$, the initial data for $\phi_1$ is centered at the point $w_1 = (-\frac 12 |p_1-p_2|R, 0,0)$, and the initial data for $\phi_2$ is centered at the point $w_2 = ( \frac 12 |p_1-p_2| R, 0,0)$ (we are assuming, without loss of generality, that $|p_i - p_j| \geq 2$).
    
    Then, for all $K_1 \in I^{\leq n-1}_{\Ga^{(h)} \cup \boldsymbol{K}_R}$ a multiindex of length at most $n-1$ composed of elements of $ \Ga^{(h)}\cup \boldsymbol{K}_R $, $h \in \{1,2\}$, for all $K_2 \in I^{\leq n-4}_{\Ga^{(h)}\cup \boldsymbol{K}_R}$, and for all $t \geq 0$, the following estimates hold for $\psi_{12}$:
    \begin{align}
        & \Vert\partial \der^{K_1} \psi_{12}\Vert_{L^2 (\Sigma_t)} \le {C \varepsilon^{2} \over R},\\
        & \Vert\partial \der^{K_1} \psi_{12}\Vert_{L^2 (H_t)} \le {C \varepsilon^{2} \over R},\\
        & \Vert (1+|u_h|)^{-\frac 12 - \frac \delta 2}\overline{\partial}^{(h)} \der^{K_1} \psi_{12} \Vert_{L^2 (\R^{3 + 1})} \le {C \varepsilon^{2} \over R},\\
        & \Vert \p \der^{K_2} \psi_{12} \Vert_{L^\infty (\Sigma_t)} \le {C \varepsilon^{2} \over \sqrt R t},\\
        & \Vert \partial \der^{K_2} \psi_{12} \Vert_{L^\infty (\Sigma_t)} \le {C \varepsilon^{2} R^2 \over t (1 + |u_i|)^{{1 \over 2}}}, \\
        & \Vert \overline{\partial}^{(h)} \der^{K_2} \psi_{12} \Vert_{L^\infty (\Sigma_t)} \le {C \varepsilon^{2} R^3 \over t^{{3 \over 2}}}.
    \end{align}
    Here, $h \in \{1,2\}$. The analogous estimates hold true for $\psi_{ij}$, with straightforward changes for the vector fields $\boldsymbol{\Gamma}^{(h)}$.
\end{theorem}

\begin{remark}
We note that the result is more generally true with $d(p_i,p_j)$ replacing $R$. We are using that all of the distances are comparable to $R$ up to a factor of $d_\Pi$ which we suppress because $\eps$ is allowed to depend on $d_\Pi$. In fact, the factor of $d_\Pi$ in these estimates would go in the denominator, which would in fact improve the estimates.
\end{remark}

\begin{remark}
    Note that the function $\psi_{1 2}$ is supported only at times $t \geq {R \over 2}$. Indeed, the equation has vanishing initial data, and we know that the inhomogeneous terms are $0$ before $t = {R \over 2}$ by comparing the domains of influence of $\phi_1$ and $\phi_2$. Furthermore, we have that the intersection of the supports of $\phi_1$ and $\phi_2$ is contained in the set
    $$
    \Big\{ t \geq \big( 1 - \frac 1 { |p_1 - p_2|}\big) r_1 \Big\} \cap \Big\{ t \geq \big( 1 - \frac 1 { |p_1 - p_2|}\big) r_2  \Big\} .
    $$
\end{remark}

\begin{theorem}[Interaction of a localized piece and the non-localized piece of initial data]\label{thm:linearnc}
    There exists a constant $C >0 $ such that the following holds. Let  $0 < \varepsilon < \varepsilon_0$, where $\varepsilon_0$ is as in the statement of Theorem~\ref{thm:main}. Let $\psi_{ij}$ be a solution to the equation
    \begin{equation}\label{eq:linearizeddiffnc}
       \begin{aligned}
        &\Box \psi_{ij} + F(d \phi_i, d^2 \phi_j) =  G (d \phi_i, d \phi_j),\\
        & \psi_{ij}|_{t=0} = 0,\\
        & \p_t \psi_{ij}|_{t=0} = 0.
    \end{aligned}
    \end{equation}
    Here, $\phi_i$, $i \in \{0, \ldots, N\}$, are as constructed in Lemma~\ref{lem:decomposition} (note again that in such lemma we require bounds on $n+7$ derivatives of initial data).
    
If either $i=0$ and $j \in \{1, \ldots, N\}$, or $j = 0$ and  $i \in \{1, \ldots, N\}$, we have the following estimates, valid for all $K_1 \in I^{\leq n-1}_{\boldsymbol{K}_R}$, $K_2\in I^{\leq n-4}_{\boldsymbol{K}_R}$,
    \begin{align}
    & \Vert\partial \der^{K_1} \psi_{ij}\Vert_{L^2 (\Sigma_t)} \le {C \varepsilon^{2} \over R^{\frac 32}},\\
    & \Vert\partial \der^{K_1} \psi_{ij}\Vert_{L^2 (H_t)} \le {C \varepsilon^{2} \over R^{\frac 32}},\\
    & \Vert (1+|u_h|)^{-\frac 12 - \delta}\overline{\partial}^{(h)} \der^{K_1} \psi_{ij} \Vert_{L^2 (\R^{3 + 1})} \le {C \varepsilon^{2} \over R^{\frac 32}},\\
    & \Vert \p \der^{K_2} \psi_{ij} \Vert_{L^\infty (\Sigma_t)} \le {C \varepsilon^{2} \over R^{\frac 32} (1+t)},\\
    & \Vert \partial \der^{K_2} \psi_{ij} \Vert_{L^\infty (\Sigma_t)} \le {C \varepsilon^{2}  \over R^{\frac 32} (1+v_i) (1 + |u_i|)^{{1 \over 2}}}, \\
    & \Vert \overline{\partial}^{(h)} \der^{K_2} \psi_{ij} \Vert_{L^\infty (\Sigma_t)} \le {C \varepsilon^{2}  \over R^{\frac 32} (1+t)^{{3 \over 2}}}.
    \end{align}
The estimates here are valid for all $h \in \{0, \ldots, N\}$.
\end{theorem}

\begin{theorem}[Existence of solutions to the nonlinear equation]\label{thm:nonlinear}
	There exists a constant $C > 0$ such that, for every $0 < \varepsilon < \varepsilon_0$, where $\varepsilon_0$ is as in the statement of Theorem~\ref{thm:main}, we have the following. We consider the initial value problem
	\begin{align}
            &\Box \Psi +F(d \Psi, d^2 \Psi) \nonumber \\
            &+ \sum_{\substack{i,j = 0,\ldots, N\\i\neq j}}(F(d \psi_{ij}, d^2 \Psi)+F(d \Psi, d^2 \psi_{ij})) + \sum_{\substack{g,h,i,j = 0,\ldots, N\\g \neq h, i\neq j}}F(d \psi_{gh}, d^2 \psi_{ij}) \nonumber \\  
            & +\sum_{i=0}^N\sum_{\substack{g,h = 0,\ldots, N\\g\neq h}} ( F(d \phi_i, d^2 \psi_{gh}) + F(d \psi_{gh}, d^2 \phi_i)) \nonumber \\  
            & +\sum_{i=0}^N ( F(d \phi_i, d^2 \Psi) + F(d \Psi, d^2 \phi_i)) \nonumber \\
             &\qquad = \sum_{\substack{i,j = 0,\ldots, N\\i\neq j}}(G(d \psi_{ij}, d \Psi)+G(d \Psi, d \psi_{ij}))+ \sum_{\substack{g,h,i,j = 0,\ldots, N\\g\neq h, i\neq j}}G(d \psi_{gh}, d \psi_{ij}) \label{eq:psistat} \\  
            & \qquad +\sum_{i=0}^N\sum_{\substack{g,h = 0,\ldots, N\\g\neq h}} ( G(d \phi_i, d \psi_{gh}) + G(d \psi_{gh}, d \phi_i)) \nonumber \\  
            & \qquad +\sum_{i=0}^N ( G(d \phi_i, d \Psi) + G(d \Psi, d \phi_i)),\\
            & \Psi|_{t=0} = 0, \nonumber \\
            & \p_t\Psi|_{t=0} = 0. \nonumber
	\end{align}
    Here, $\phi_i$ are as in the statement of Lemma~\ref{prop:decphii}, and $\psi_{ij}$ is as in the statement of Theorem~\ref{thm:linear} and Theorem~\ref{thm:linearnc}. Under these hypotheses, we have that equation~\eqref{eq:psistat} admits a global solution $\Psi$. Moreover, we have the following quantitative decay estimates, valid for all $I \in I^{\leq N_0}_{\boldsymbol{K}_R}$, with $N_0 \geq 7$, and $i \in \{0, \ldots, N\}$:
\begin{align*}
&\Vert \p \der^I \Psi \Vert_{L^2(\Sigma_t)} \leq \eps^{ 3 -\delta} R^{-\frac 32 + \delta} \quad  \text{ for } \quad t\geq 0,   \\
&\Vert (1+|u_i|)^{-\frac 12 - \frac \delta 2} \bar \p^{(i)} \der^I \Psi \Vert_{L^2([0,T]\times \R^3)} \leq \eps^{ 3 -\delta}  R^{-\frac 32 + \delta}  \quad  \text{ for } \quad t \geq 0, \quad i \in \{0, \ldots, N\},  \\
&|\p \der^J  \Psi(t, r, \theta,\varphi)| \leq \eps^{ 3 -\delta}  (1+t)^{-1} R^{-\frac 1 2 +\delta}  \quad  \text{ for } t \geq 0,    \\
&|\bar \p^{(i)} \der^J \Psi(t, r, \theta,\varphi)| \leq \eps^{ 3 -\delta}  (1+t)^{-\frac 32 } R^{\frac 32 + \delta}  \quad  \text{ for } \quad t \geq R^{20}, \quad i \in \{0, \ldots, N\}, \\
&|\p \der^J \Psi(t, r, \theta,\varphi)| \leq  \eps^{ 3 -\delta}  (1+v_i)^{-1}(1+|u_i|)^{-\frac 12} R^{\frac 12 + \delta}  \quad  \text{ for } \quad t \geq R^{20}, \quad i \in \{0, \ldots, N\}.
\end{align*}
\end{theorem}

\subsection{Proof of the main theorem given the auxiliary theorems}

\begin{proof}[Proof of Theorem~\ref{thm:main} given Theorems~\ref{thm:linear} and~\ref{thm:nonlinear}]
    We note that Lemma~\ref{prop:decphii} gives the existence of global-in-time solutions $\phi_i$ to the initial value problem~\eqref{eq:phii}. Moreover, we have that each $\psi_{i j}$ exists globally because it solves a linear wave equation. Then, by Theorem~\ref{thm:nonlinear}, we have the global-in-time existence of $\Psi$ satisfying the initial value problem~\eqref{eq:psistat}.  We then let
    \begin{equation*}
        \phi := \Psi + \sum_{i = 0}^N \phi_i + \sum_{\substack{i,j  = 0, \ldots, N\\i\neq j}} \psi_{ij},
    \end{equation*}
    and note that $\phi$ is a smooth global solution to the initial value problem:
    \begin{equation}
	\begin{aligned}
	    &\Box \phi + F (d \phi, d^2 \phi) =  G (d \phi, d \phi), \\
	    & \phi|_{t=0} = \phi^{(0)},\\
	    & \p_t \phi|_{t=0} = \phi^{(1)}.
	\end{aligned}
    \end{equation}
    The calculations showing this are carried out in more detail in Sections  \ref{sub:firstiterate} and \ref{sub:seconditerate}. Furthermore, one can clearly reconstruct the decay estimates for $\phi$ from the known decay estimates for $\Psi, \psi_{ij}, \phi_i$. This concludes the proof of the main theorem (Theorem~\ref{thm:main}).
\end{proof}

\subsection{A large data global existence result}
We are also able to prove a large data global existence result.

\begin{theorem}\label{thm:largedata}
Let $L \in \R$, $L \geq 0$, and $M \in \N$, $M \geq 0$ be given. There exist a real number $\eps_0$, an integer $N(L)$ and points $w_i \in \R^3$, $i \in \{1, \ldots, N(L)\}$, such that the following holds. There exists a collection
$$
(\phi_{A,i}^{(0)},\phi_{A,i}^{(1)})_{A \in \{1, \ldots, M\}, \, i \in \{1, \ldots, N(L)\}}
$$
of pairs of smooth vector valued functions on $\R^3$ satisfying the following properties:
\begin{align}
 &\text{supp}(\phi^{(0)}_{A,i}) \subset B(0,1), \qquad \text{supp}(\phi^{(1)}_{A,i}) \subset B(0,1),\\
 &\Vert \phi_{A,i}^{(0)}\Vert_{H^{N_1}(B(0,1))} \leq \eps_0, \qquad  \Vert \phi_{A,i}^{(1)}\Vert_{H^{N_1-1}(B(0,1))} \leq \eps_0,
 \end{align}
 for all $i \in \{1, \ldots, N(L)\}$ and all $A \in \{1, \ldots, M\}$, where $N_1 \geq 19$. Moreover,
under these conditions, the trivial solution to the initial value problem~\eqref{eq:bumps} is asymptotically stable under perturbations of the form 
$$
\phi_A^{(0)} = \sum_{i=1}^{N(L)} \phi_{A,i}^{(0)} (x - w_i), \qquad \phi_A^{(1)} = \sum_{i=1}^{N(L)} \phi_{A,i}^{(1)} (x - w_i).
$$
Moreover, we have that
$$
\Vert \phi_A^{(0)}  \Vert_{H^1(\R^3)}  \geq L, \qquad \Vert \phi_A^{(1)}  \Vert_{L^2(\R^3)}  \geq L,
$$
meaning that the initial data is allowed to have arbitrarily large $H^1$ norm.
\end{theorem}
The configuration of the data is further described in Section \ref{sec:largedata}. Moreover, we shall once again restrict ourselves to studying single equations of the form \eqref{squaseq} to simplify notation, although the same proof would establish the Theorem for quasilinear systems as in \eqref{eq:bumps}.

\section{Geometry of interacting wave fronts}\label{sec:technicaltools}

In this section, we prove three important technical tools which we will need in the rest of the paper. They are a statement concerning improved $u$-decay for solution to quasilinear wave equations satisfying the null condition and with localized initial data (Lemma~\ref{prop:decphii}), an important change of coordinates (Lemma~\ref{lem:r1r2}, which is used to formalize the fact that the measure of the interaction region of two nonlinear waves originating from sources located far away from each other is small), and finally a lemma concerning the asymptotic comparison of null derivatives with respect to two different light cones (Lemma~\ref{lem:goodbad}).

\subsection{Decay properties for localized initial data}\label{sec:onebump}
In this section, we prove the following result.

\begin{lemma}[Improved $u$-decay for solutions to quasilinear wave equations]\label{prop:decphii}
	Let $R>0$, $n \in \N$, $n \geq 0$. Let also $F,G$ be a trilinear resp.~bilinear classical null form as in Definition~\ref{def:null}. There exist a universal constant $C>0$ and a $\delta > 0$ such that the following holds. Let $\zeta$ be a smooth solution to the following initial value problem:
	\begin{equation}\label{eq:phii}
	\begin{aligned}
	&\Box \zeta + F(\de \zeta, \de^2 \zeta) = G(\de \zeta, \de \zeta),\\
	&\zeta|_{t = 0}  = \zeta^{(0)},\\
	&\p_t \zeta|_{t = 0} = \zeta^{(1)}.
	\end{aligned}
	\end{equation}
	We further suppose the following bounds on $\zeta^{(0)}$ and $\zeta^{(1)}$, for all $k$ non-negative integers, $k \leq n+7$:
	\begin{equation}\label{eq:decayinitial}
	    |\widehat{\partial}^{k} \zeta^{(0)}|\leq \varepsilon (1+r)^{-k-2}, \qquad |\widehat{\partial}^{k} \zeta^{(1)}|\leq \varepsilon (1+r)^{-k-3}.
	\end{equation}
	Here, given $f$ a smooth function on $\R^3$, recall that we defined $|\hat \partial^{k} f|$ as
	$$
	|\widehat{\partial}^{k} f| : = \sum_{I \in I^{k}_{\boldsymbol{T}_s}} |\der^I f|, \qquad \text{where} \qquad \boldsymbol{T}_s := \{\p_x, \p_y, \p_z\}.
	$$
	In these conditions, for every multi-index $I \in I^{\leq n}_{ \boldsymbol{K}}$, we have the following decay properties:
	\begin{equation}\label{eq:decphii}
	|\bar \p^{(0)} \der^{I} \zeta| \leq C \varepsilon \frac{1}{(1+r^2)(1+|u|)^\delta}, \quad |\p \der^{I} \zeta| \leq C \varepsilon \frac{1}{(1+v)(1+|u|)^{1+\delta}}.
	\end{equation}
	Here, we used the definition of good derivatives adapted to the light cone emanating from the origin in equation~\eqref{eq:goodcone}.
	
	Furthermore, we have the following uniform $L^2$ estimates, valid for all $I \in I^{\leq n+4}_{\boldsymbol{K}}$, and all $t \geq 0$:
	\begin{equation}\label{eq:l2phii}
	\Vert \p \der^I \zeta \Vert_{L^2(\Sigma_t)} \leq C \varepsilon.
	\end{equation}
\end{lemma}

\begin{proof}[Proof of Lemma~\ref{prop:decphii}] The proof follows straightforwardly from known theory, and we include it here for completeness. 
We will break up the proof in several {\bf Steps}. In {\bf Step 0}, we will recall the decay properties which follow from classical theory. Subsequently, in {\bf Step 1}, we will perform $r^p$ estimates on equation~\eqref{eq:phii} choosing $p=1+\delta$ and $p= \delta$. We will use the estimates from {\bf Step 0} to control the error terms in the RHS of such estimates. Having done that, in {\bf Step 2}, we will proceed to prove the first estimate in display~\eqref{eq:decphii}. This will involve a mean value theorem argument (which will give decay along a dyadic sequence) plus eliminating the restriction to the dyadic sequence. Finally, in {\bf Step 3}, we will commute the equation with the vector field $r\p_v$, in order to obtain the second claim in display~\eqref{eq:decphii}.

{\bf Step 0}. We recall that, from classical theory, the global solution $\zeta$ to the initial value problem~\eqref{eq:phii} satisfies the following estimates, for every multi-index $I \in I^{\leq n+4}_{\boldsymbol{K}}$:
	\begin{equation}\label{eq:decclassical}
	|\bar \p \der^{I} \zeta| \leq C \varepsilon \frac{1}{(1+v)^{\frac 32}}, \quad |\p \der^{I} \zeta| \leq C \varepsilon \frac{1}{(1+v)(1+|u|)^{\frac 12}}, \quad |\der^{I} \zeta| \leq C \varepsilon \frac{(1+|u|)^{\frac 12}}{(1+v)}.
	\end{equation}
We also have, for every multi-index $J \in I^{\leq n+3}_{\boldsymbol{K}}$:
	\begin{equation}\label{eq:decclassical2}
	|\bar \p \bar \p \der^{J} \zeta| \leq C \varepsilon \frac{1}{(1+v)^{\frac 52}}, \qquad | \p \bar \p \der^{J} \zeta| \leq C \varepsilon \frac{1}{(1+v)^{\frac 32}(1+|u|)}
	\end{equation}
Furthermore, for every multi-index $I \in I^{\leq n+6}_{\boldsymbol{K}}$, we have the characteristic energy bounds, valid for all $u_1 \in \R$:
\begin{equation}\label{eq:nullenergy}
    \int_{\{u = u_1\} \cap \{t \geq 0\}} |\bar \p \der^I \zeta|^2 r^2 \de v \de \omega \leq C \eps^2.
\end{equation}

{\bf Step 1}. We now note that $r \zeta$ satisfies the following equation:
\begin{equation}\label{eq:zetar}
    \p_u \p_v (r \zeta) -\slashed{\Delta} (r \zeta) = r F(d \zeta, d^2 \zeta) - r G(d\zeta, d\zeta).
\end{equation}
After commutation with $\der^I$, with $I \in I^{\leq n+1}_{\boldsymbol{K}}$, we have:
\begin{equation}\label{eq:zetarcomm}
    \p_u \p_v (r \der^I \zeta) -\slashed{\Delta} (r \der^I \zeta) = \sum_{H + K \subset I} \Big( r F_{HK}(d \der^H \zeta, d^2 \der^K \zeta) - r G(d \der^H \zeta, d \der^K \zeta) \Big),
\end{equation}
where $F_{HK}$, $G_{HK}$ are a collection of trilinear (resp.~bilinear) classical null forms, in the sense of Definition~\ref{def:null}.

We let let $\chi$ to be a smooth cutoff function such that 
$$
\chi(v) =\frac{1}{1+av^2},
$$
for $v \geq 0$ (we think of $a$ as being a positive small parameter). We then multiply equation~\eqref{eq:zetar} by $\chi(v) r^p \p_v (r\der^I \zeta)$ to get
\begin{equation*}
\begin{aligned}
    &\frac 12\p_u (\chi(v) r^{p} (\p_v (r\der^I \zeta))^2)+\frac p 2 \chi(v) r^{p-1} (\p_v (r\der^I \zeta))^2 + \frac 12 \p_v  (\chi(v) r^{p} |\slashed \nabla (r\der^I \zeta)|^2 )- \chi(v)\frac{p-2}2 r^{p-1}|\slashed \nabla (r\der^I \zeta)|^2\\
    &\qquad - \frac{\chi'(v)}{2} r^{p}|\slashed \nabla (r\der^I \zeta)|^2 \equiv_{\mathbb{S}^2} \sum_{H + K \subset I}(r F_{HK}(d \der^H \zeta, d^2 \der^K \zeta) - r G_{HK}(d \der^H\zeta, d\der^K\zeta))\chi(v) r^p \p_v (r\der^I \zeta).
\end{aligned}
\end{equation*}
Here, the symbol $\equiv_{\mathbb{S}^2}$ denotes that equality is achieved upon going to polar coordinates $(t, r, \omega)$ and integrating over $\omega \in \mathbb{S}^2$. We now integrate the previous display on the region $\mathcal{R}_{0}^{u_2} := \{(u,v): u \leq u_1, \frac 12 (u+v) \geq 0\}$ with respect to the form $\de u \de v$. We obtain, discarding the positive boundary terms, bounding the initial data term, and assuming that $p = 1+ \delta$ (this gets rid of the boundary term at $r=0$):
\begin{equation*}
\begin{aligned}
    &\int_{\{u = u_1\}\cap \{t \geq 0\}} \frac \chi 2 r^{1+\delta} (\p_v(r\der^I \zeta))^2 \de u \de \omega + \int_{\mathcal{R}_{0}^{u_1}} \frac {1+\delta} 2 \chi\, r^{\delta}(\p_v (r\der^I \zeta))^2 \de u \de v \de \omega \\
    &+\int_{\mathcal{R}_{0}^{u_1}} \frac {1-\delta} 2 \chi\, r^{\delta}|\slashed{\nabla} (r\der^I \zeta)|^2 \de u \de v \de \omega\\ 
    &\qquad \leq C \eps^2 + \sum_{H + K \subset I}\Big| \int_{\mathcal{R}_{0}^{u_1}} ( F_{HK}(d \der^H \zeta, d^2 \der^K \zeta) - G_{HK}(d \der^H\zeta, d \der^K\zeta))\, \chi \, \p_v \der^I \zeta \, r^{2+\delta} \de u \de v \de \omega\Big|.
\end{aligned}
\end{equation*}
Let us now sum the previous inequality over all $I \in I^{\leq n+4}_{\boldsymbol{K}}$ and take the limit as $a \to 0$. We get, letting $\eta_{I} := r \der^I \zeta$,
\begin{equation}\label{eq:rp1pdeltasum}
\begin{aligned}
    &\sum_{I \in I^{\leq n+4}_{\boldsymbol{K}}}\int_{\{u = u_1\}\cap \{t \geq 0\}} \frac 1 2 r^{1+\delta} (\p_v \eta_{I})^2 \de u \de \omega + \sum_{I \in I^{\leq n+4}_{\boldsymbol{K}}}\int_{\mathcal{R}_{0}^{u_1}} \frac {1+\delta} 2 \, r^{\delta}(\p_v \eta_{I})^2 \de u \de v \de \omega \\
    &+\sum_{I \in I^{\leq n+4}_{\boldsymbol{K}}}\int_{\mathcal{R}_{0}^{u_1}} \frac {1-\delta} 2 \, r^{\delta}|\slashed{\nabla} \eta_{I}|^2 \de u \de v \de \omega\\ 
    &\qquad \leq C \eps^2 + \lim_{a\to 0}\sum_{I \in I^{\leq n+4}_{\boldsymbol{K}}}\sum_{H + K \subset I}\Big| \int_{\mathcal{R}_{0}^{u_1}} ( F_{HK}(d \der^H \zeta, d^2 \der^K \zeta) - G_{HK}(d \der^H\zeta, d \der^K\zeta))\, \chi \, \p_v \eta_I \, r^{2+\delta} \de u \de v \de \omega\Big|.
\end{aligned}
\end{equation}
Let us now bound the terms on the RHS of equation~\eqref{eq:rp1pdeltasum}.
Let us start with the terms containing $G_{HK}$:
\begin{equation*}
    \begin{aligned}
        &\int_{\mathcal{R}_{0}^{u_1}} \big| G_{HK}(d \der^H\zeta, d  \der^K \zeta)\big|r^{2+\delta} |\p_v \eta_I| \de u \de v \de \omega\\
        &\quad \leq C \int_{\mathcal{R}_{0}^{u_1}} r^{1+\delta}|\bar \p^{(0)} (r \der^H \zeta)| \, |\p \der^K \zeta|\, |\p_v \eta_I| \de u \de v \de \omega+ C  \int_{\mathcal{R}_{0}^{u_1}} r^{1+\delta}|\der^H \zeta|\,| \p \der^K \zeta| \, |\p_v \eta_I| \de u \de v \de  \omega\\
        &\quad \leq C \eps \int_{\mathcal{R}_{0}^{u_1}} r^{\delta}|\bar \p^{(0)} \eta_H|\, |\p_v \eta_I| \de u \de v \de \omega
        \\
        & \qquad + C \Big( \int_{\mathcal{R}_{0}^{u_1}} r^{\delta}|\p_v \eta_I|^2 \de u \de v \de  \omega \Big)^{\frac 12}\Big( \int_{\mathcal{R}_{0}^{u_1}} r^{2+\delta}|\der^H \zeta|^2 \, | \p \der^K \zeta|^2\,  \de u \de v \de  \omega \Big)^{\frac 12}
    \end{aligned}
\end{equation*}
The first term in the last display can be clearly absorbed in the LHS of~\eqref{eq:rp1pdeltasum}, upon an application of the Cauchy--Schwarz inequality. Another application of the Cauchy--Schwarz inequality reduces estimating the second term in the last display to the following expression:
\begin{equation*}
    \begin{aligned}
        & \int_{\mathcal{R}_{0}^{u_1}} r^{2+\delta}|\der^H \zeta|^2 \, | \p \der^K \zeta|^2\,  \de u \de v \de  \omega  \leq C \eps^2 \int_{\mathcal{R}_{0}^{u_1}} r^{\delta}|\der^H \zeta|^2 \,  \de u \de v \de  \omega = C \eps^2 \int_{\mathcal{R}_{0}^{u_1}} r^{\delta-2}|r\der^H \zeta|^2 \,  \de u \de v \de  \omega\\
         & \quad \leq C \varepsilon^2 \int_{\mathcal{R}_{0}^{u_1}} r^{\delta}|\p_v (r \der^H\eta)|^2 \,  \de u \de v \de  \omega = C \varepsilon^2 \int_{\mathcal{R}_{0}^{u_1}} r^{\delta}|\p_v \eta_H|^2 \,  \de u \de v \de  \omega,
    \end{aligned}
\end{equation*}
which can be again absorbed in the LHS of~\eqref{eq:rp1pdeltasum}. Note that we used a Hardy-type inequality to obtain this bound.

\begin{remark}
    Note that we have to estimate at most $n+5$ derivatives of $\zeta$ in $L^\infty$. This is why we require $n+7$ derivatives of initial data, and then note that the estimates arising from classical theory (the estimates is~\eqref{eq:decclassical}) ``lose'' two derivatives, which means that we control $n+5$ derivatives in $L^\infty$ of $\zeta$.
\end{remark}

We now turn to estimating the term
\begin{equation}
    \Big| \int_{\mathcal{R}_{0}^{u_1}}  F_{HK}(d \der^H\zeta, d^2\der^K \zeta)\, \chi \, \p_v \eta_I \, r^{2+\delta} \de u \de v \de \omega\Big|.
\end{equation}
We note that, if $|K| < n+4$, we can estimate this term in exactly the same way we estimated the terms in $G_{HK}$. This is because we are allowed to absorb the high-order derivative terms in the LHS. On the other hand, if $H = 0$ (the empty multi-index), $K = I$, and $|I|=n+4$, a different argument is needed. We now focus on that case. Let $\tilde F := F_{0I}$. First, we estimate
\begin{equation*}
\begin{aligned}
    &\Big| \int_{\mathcal{R}_{0}^{u_1}}  \tilde F(d \zeta, d^2 \der^I \zeta)\, \chi \, \p_v \eta_I \, r^{2+\delta} \de u \de v \de \omega\Big| \\
    &\qquad \leq \Big| \int_{\mathcal{R}_{0}^{u_1}}  \tilde F(d \zeta, d^2 (r\der^I \zeta))\, \chi \, \p_v \eta_I \, r^{1+\delta} \de u \de v \de \omega\Big| \\
    &\qquad + 2 \Big| \int_{\mathcal{R}_{0}^{u_1}}  \tilde F^{\alpha\beta\gamma} \p_\alpha \zeta \, \p_\beta r \, \p_\gamma \der^I \zeta\, \chi \, \p_v \eta_I \, r^{1+\delta} \de u \de v \de \omega\Big|.
\end{aligned}
\end{equation*}
We then note that the following identity holds true:
\begin{equation*}
\begin{aligned}
&\p_v \eta_I \  \tilde F^{\alpha\beta\gamma} \ \p_\alpha \zeta  \ \p_\beta \p_\gamma  \eta_I\\
&= \p_\beta \big(\tilde F^{\alpha\beta\gamma} \  \p_\alpha \zeta \ \p_v \eta_I  \  \p_\gamma  \eta_I\big) -  \tilde F^{\alpha\beta\gamma} \  \p_\beta \p_\alpha \zeta \ \p_v \eta_I  \  \p_\gamma \eta_I\\
&\quad -\frac 12 \p_v \big(\p_\alpha \zeta \ \tilde F^{\alpha \beta\gamma} \ \p_\beta  \eta_I \p_\gamma \eta_I\big) + \frac 12  \p_v \p_\alpha \zeta\ \tilde F^{\alpha \beta\gamma} \ \p_\beta  \eta_I \p_\gamma \eta_I.
\end{aligned}
\end{equation*}

In view of this, we obtain, integrating by parts the first and last term arising from the previous display, since $|\chi'(v)| = \frac{2av}{(1+av^2)^2} \leq \frac{2av}{2av^2} \leq \frac 1 r,$
\begin{equation}\label{eq:masterf}
\begin{aligned}
    &\Big|\int_{\mathcal{R}_{0}^{u_1}}  \tilde F(d \zeta, d^2 \eta_I)\, \chi \, \p_v \eta_I \, r^{1+\delta} \de u \de v \de \omega \Big|  \\
    &\quad \leq \Big|\int_{\{u=u_1\}\cap \{t \geq 0\}} \tilde F^{\alpha u \gamma} \, \p_\alpha \zeta \, \p_v \eta_I \, \p_\gamma \eta_I \, \chi  \, r^{1+\delta}\de v \de \omega\Big| \\
    &\qquad +5\Big| \int_{\mathcal{R}_{0}^{u_1}} \tilde F^{\alpha v\gamma} \  \p_\alpha \zeta \ \p_v \eta_I  \  \p_\gamma \eta_I \, r^{\delta} \,  \de u \de v \de \omega\Big|\\
    &\qquad +5\Big| \int_{\mathcal{R}_{0}^{u_1}} \tilde F^{\alpha u\gamma} \  \p_\alpha \zeta \ \p_v \eta_I  \  \p_\gamma \eta_I \, r^{\delta} \,  \de u \de v \de \omega\Big|\\
    &\qquad +\Big|\int_{\mathcal{R}_{0}^{u_1}} \tilde F^{\alpha\beta\gamma} \  \p_\beta \p_\alpha \zeta \ \p_v \eta_I \ \p_\gamma \eta_I r^{1+\delta} \chi \,  \de u \de v \de \omega \Big|\\
    &\qquad + \frac 32 \Big|  \int_{\mathcal{R}_{0}^{u_1}} \tilde F^{\alpha \beta\gamma}  \p_\alpha \zeta \ \p_\beta  \eta_I \  \p_\gamma \eta_I \, r^\delta \, \de u \de v\Big| \\
    &\qquad + \frac 12 \Big| \int_{\mathcal{R}_{0}^{u_1}} \tilde F^{\alpha \beta\gamma} \ \p_v \p_\alpha \zeta \ \p_\beta  \eta_I \p_\gamma \eta_I \, r^{1+\delta} \, \chi \, \de u \de v \de \omega\Big| + C \eps^2.
\end{aligned}
\end{equation}
We now estimate the terms:
\begin{equation*}
    \begin{aligned}
        &\Big|\int_{\{u=u_1\}\cap \{t \geq 0\}} \tilde F^{\alpha u \gamma} \, \p_\alpha \zeta \, \p_v \eta_I \, \p_\gamma \eta_I \, \chi  \, r^{1+\delta}\de v \de \omega\Big|\\
        &\quad \leq \frac 1 {20} \int_{\{u=u_1\}\cap \{t \geq 0\}} (\p_v \eta_I)^2 r^{1+\delta}\de v \de \omega + C\int_{\{u=u_1\}\cap \{t \geq 0\}} |\bar \p \eta_I|^2|\p \zeta|^2  r^{1+\delta}\de v \de \omega\\
        &\qquad + C\int_{\{u=u_1\}\cap \{t \geq 0\}} | \p \eta_I|^2|\bar \p \zeta|^2  r^{1+\delta}\de v \de \omega.
    \end{aligned}
\end{equation*}
Now, 
\begin{equation*}
\begin{aligned}
    &\int_{\{u=u_1\}\cap \{t \geq 0\}} |\bar \p \eta_I|^2|\p \zeta|^2  r^{1+\delta}\de v \de \omega \\
    &\qquad \leq C\eps^2 \int_{\{u=u_1\}\cap \{t \geq 0\}} | \p_v \eta_I|^2\de v \de \omega +C\int_{\{u=u_1\}\cap \{t \geq 0\}} |\slashed \nabla \der^I \zeta|^2|\p \zeta|^2  r^{3+\delta}\de v \de \omega\\
    &\qquad \leq C\eps^2  \int_{\{u=u_1\}\cap \{t \geq 0\}} r^{1+\delta}| \p_v \eta_I|^2\de v \de \omega +C \varepsilon^4.
\end{aligned}
\end{equation*}
Note that in the previous inequality we used bound~\eqref{eq:nullenergy} from the classical theory to bound the angular terms, and in addition we estimated the first term with a higher power of $r$ (the error in this procedure is just at a finite $r$-value, and as such we can control it again by the characteristic energy).

Furthermore,
\begin{align*}
        &\int_{\{u=u_1\}\cap \{t \geq 0\}} | \p \eta_I|^2|\bar \p \zeta|^2  r^{1+\delta}\de v \de \omega\\
        &\quad \leq C\int_{\{u=u_1\}\cap \{t \geq 0\}} |\bar \p \eta_I|^2|\bar \p \zeta|^2  r^{1+\delta}\de v \de \omega  +C\int_{\{u=u_1\}\cap \{t \geq 0\}} | \p_u \eta_I|^2|\bar \p \zeta|^2  r^{1+\delta}\de v \de \omega\\
        &\quad \leq  C\eps^2 \int_{\{u=u_1\}\cap \{t \geq 0\}} | \p_v \eta_I|^2\de v \de \omega +C \varepsilon^4+C\eps^2\int_{\{u=u_1\}\cap \{t \geq 0\}} |\bar \p \zeta|^2 r^\delta \de v \de \omega \\
        &\qquad +C\int_{\{u=u_1\}\cap \{t \geq 0\}} | \p\der^I \zeta|^2|\bar \p \zeta|^2  r^{3+\delta}\de v \de \omega\\
        &\quad \leq  C\eps^2 \int_{\{u=u_1\}\cap \{t \geq 0\}} r^{1+\delta}| \p_v \eta_I|^2\de v \de \omega +C\eps^2\int_{\{u=u_1\}\cap \{t \geq 0\}} |\p_v \eta_I|^2 r^\delta \de v \de \omega + C \eps^4\\
        &\quad \leq  C\eps^2 \int_{\{u=u_1\}\cap \{t \geq 0\}} r^{1+\delta}| \p_v \eta_I|^2\de v \de \omega + C \eps^4.
\end{align*}
Here, we used the estimates in {\bf Step 0}. This concludes the inequality for the boundary term at $u = u_1$.

We then proceed to estimate the second and third terms in display~\eqref{eq:masterf}. We obtain
\begin{equation}\label{eq:firsttermf}
\begin{aligned}
    &\Big| \int_{\mathcal{R}_{0}^{u_1}} \tilde F^{\alpha v\gamma} \  \p_\alpha \zeta \ \p_v \eta_I  \  \p_\gamma \eta_I \, r^{\delta} \, \de u \de v \de \omega\Big|+\Big| \int_{\mathcal{R}_{0}^{u_1}} \tilde F^{\alpha u\gamma} \  \p_\alpha \zeta \ \p_v \eta_I  \  \p_\gamma \eta_I \, r^{\delta} \, \de u \de v \de \omega\Big|\\
    &\quad \leq \frac 1{20} \int_{\mathcal{R}_{0}^{u_1}}  (\p_v \eta_I)^2    r^{\delta}  \de u \de v \de \omega +C \int_{\mathcal{R}_{0}^{u_1}}  |\p \eta_I|^2 | \p \zeta|^2  r^{\delta}  \de u \de v \de \omega\\
    &\quad \leq \frac 1{20} \int_{\mathcal{R}_{0}^{u_1}}  (\p_v \eta_I)^2    r^{\delta}  \de u \de v \de \omega  +C \int_{\mathcal{R}_{0}^{u_1}}  |\p \der^I \zeta|^2|\p \zeta|^2  r^{2+\delta}  \de u \de v \de \omega \\
    &\qquad +C \int_{\mathcal{R}_{0}^{u_1}}|\der^I \zeta|^2  |\p \zeta|^2  r^{\delta}  \de u \de v \de \omega\\
    & \quad \leq \frac 1{20} \int_{\mathcal{R}_{0}^{u_1}}  (\p_v \eta_I)^2    r^{\delta}  \de u \de v \de \omega +C \eps^4 \int_{\mathcal{R}_{0}^{u_1}}\frac{1}{(1+|u|)^2(1+v)^4} r^{2+\delta}  \de u \de v \de \omega\\
    &\quad +C \eps^4 \int_{\mathcal{R}_{0}^{u_1}}\frac{(1+|u|)}{(1+v)^2} (1+v)^{-2} (1+|u|)^{-1} r^{\delta}  \de u \de v \de \omega\\
    & \quad \leq \frac 1{20} \int_{\mathcal{R}_{0}^{u_1}}  (\p_v \eta_I)^2    r^{\delta}  \de u \de v \de \omega  +C \eps^4.
\end{aligned}
\end{equation}
Here, again, we used the pointwise estimates in {\bf Step 0}, see display~\eqref{eq:decclassical}.

Similarly, for the fourth term,
\begin{align}
    &\Big|\int_{\mathcal{R}_{0}^{u_1}} \tilde F^{\alpha\beta\gamma} \  \p_\beta \p_\alpha \zeta \ \p_v \eta_I \ \p_\gamma \eta_I r^{1+\delta} \chi \de u \de v \de \omega\Big| \nonumber\\
    &\quad \leq \frac 1{20} \int_{\mathcal{R}_{0}^{u_1}}  (\p_v \eta_I)^2    r^{\delta}  \de u \de v \de \omega +C \int_{\mathcal{R}_{0}^{u_1}}  |\p \bar \p \zeta|^2 | \p \eta_I|^2  r^{2+\delta}  \de u \de v \de \omega \nonumber\\
    &\quad +C \int_{\mathcal{R}_{0}^{u_1}}  |\p \p \zeta|^2 | \bar \p \eta_I|^2  r^{2+\delta}  \de u \de v \de \omega \nonumber\\
    &\quad \leq \frac 1{10} \int_{\mathcal{R}_{0}^{u_1}}  ((\p_v \eta_I)^2 +|\slashed{\nabla}\eta_I|^2)   r^{\delta}  \de u \de v \de \omega +C \int_{\mathcal{R}_{0}^{u_1}}  | \p_u \eta_I|^2 | \p \bar \p \zeta|^2  r^{2+\delta}  \de u \de v \de \omega \nonumber\\ 
    &\quad \leq \frac 1{10} \int_{\mathcal{R}_{0}^{u_1}}  ((\p_v \eta_I)^2 +|\slashed{\nabla}\eta_I|^2)   r^{\delta}  \de u \de v \de \omega + C \int_{\mathcal{R}_{0}^{u_1}}  | \der^I \zeta|^2 | \p\bar \p \zeta|^2  r^{2+\delta}  \de u \de v \de \omega \nonumber\\
    &\quad + C \int_{\mathcal{R}_{0}^{u_1}}  | \p \der^I \zeta|^2 |\p \bar \p \zeta|^2  r^{4+\delta}  \de u \de v \de \omega  \nonumber\\
    &\quad \leq \frac 1{10} \int_{\mathcal{R}_{0}^{u_1}}  ((\p_v \eta_I)^2 +|\slashed{\nabla}\eta_I|^2)   r^{\delta}  \de u \de v \de \omega + C \eps^4 \int_{\mathcal{R}_{0}^{u_1}} \frac{1+|u|}{(1+v)^2} \frac{1}{(1+|u|)^2(1+v)^{3}} r^{2+\delta} \de u \de v \de \omega \nonumber\\
    &\qquad + C \int_{\mathcal{R}_{0}^{u_1}}  | \p \der^I \zeta|^2 |\bar \p( r\p \zeta)|^2  r^{2+\delta}  \de u \de v \de \omega + C \int_{\mathcal{R}_{0}^{u_1}} | \p \der^I \zeta|^2 | \p \zeta|^2 r^{2+\delta}  \de u \de v \de \omega \nonumber\\
    &\quad \leq \frac 1{10} \int_{\mathcal{R}_{0}^{u_1}}  ((\p_v \eta_I)^2 +|\slashed{\nabla}\eta_I|^2)   r^{\delta}  \de u \de v \de \omega + C \eps^2 \int_{\mathcal{R}_{0}^{u_1}} |\bar \p (r\p \zeta)|^2  r^{\delta}  \de u \de v \de \omega+ C \eps^4 \nonumber \\
    &\quad \leq \frac 1 9 \sum_{I \in I^{\leq n+4}_{\boldsymbol{K}}}\int_{\mathcal{R}_{0}^{u_1}}  ((\p_v \eta_I)^2 +|\slashed{\nabla}\eta_I|^2)  r^{\delta}  \de u \de v \de \omega+ C \eps^4 \label{eq:secondtermf} 
\end{align}

Now, for the fifth term in the RHS of display~\eqref{eq:masterf}, we have, by analogous estimates as those in display~\eqref{eq:firsttermf},
\begin{equation}
    \begin{aligned}
        \Big|  \int_{\mathcal{R}_{0}^{u_1}} \tilde F^{\alpha \beta\gamma}  \p_\alpha \zeta \ \p_\beta  \eta_I \  \p_\gamma \eta_I \, r^\delta \, \de u \de v\de \omega\Big| \leq \frac 1{20} \int_{\mathcal{R}_{0}^{u_1}}  (\p_v \eta_I)^2    r^{\delta}  \de u \de v \de \omega  +C \eps^4. 
    \end{aligned}
\end{equation}
Finally, for the sixth term in the RHS of display~\eqref{eq:firsttermf},
\begin{equation*}
    \begin{aligned}
        &\Big| \int_{\mathcal{R}_{0}^{u_1}} \tilde F^{\alpha \beta\gamma} \ \p_v \p_\alpha \zeta \ \p_\beta  \eta_I \p_\gamma \eta_I \, r^{1+\delta} \, \chi \, \de u \de v \de \omega\Big| \\ &\quad \leq  C \int_{\mathcal{R}_{0}^{u_1}}  |\bar \p \bar \p \zeta | \ |\p  \eta_I|^2  \, r^{1+\delta} \, \de u \de v \de \omega + C  \int_{\mathcal{R}_{0}^{u_1}} |\p_v \p \zeta| \ |\bar \p  \eta_I| |\p \eta_I| \, r^{1+\delta} \, \de u \de v \de \omega.
    \end{aligned}
\end{equation*}
The second term in the last display is estimated exactly as in display~\eqref{eq:secondtermf}. Regarding the first term, we have, using estimate~\eqref{eq:secondtermf} in {\bf Step 0},
\begin{equation*}
    \begin{aligned}
        &\int_{\mathcal{R}_{0}^{u_1}}  |\bar \p \bar \p \zeta | \ |\p  \eta_I|^2  \, r^{1+\delta} \, \de u \de v \de \omega \\
        & \quad \leq C  \int_{\mathcal{R}_{0}^{u_1}} |\bar \p \bar \p \zeta | \ |\p \der^I \zeta|^2  \, r^{3+\delta} \, \de u \de v \de \omega  +  C  \int_{\mathcal{R}_{0}^{u_1}} |\bar \p \bar \p \zeta | \ |\der^I \zeta|^2  \, r^{1+\delta} \, \de u \de v \de\omega\\
        &\quad \leq C \eps^4  \int_{\mathcal{R}_{0}^{u_1}} (1+v)^{-\frac 52} (1+v)^{-2}(1+|u|)^{-1} r^{3+\delta} \, \de u \de v \de \omega \\
        & \qquad +  C \eps^4  \int_{\mathcal{R}_{0}^{u_1}} (1+v)^{-\frac 52} (1+v)^{-2}(1+|u|)  \, r^{1+\delta} \, \de u \de v  \de\omega \leq C \eps^4.
    \end{aligned}
\end{equation*}

All in all, we have, combining the estimates for the terms on the RHS of~\eqref{eq:masterf}, for all $u_1 \in \R$, and absorbing terms in the LHS accordingly,
\begin{align*}
    &\sum_{I \in I^{\leq n+4}_{\boldsymbol{K}}}\int_{\{u = u_1\}\cap \{t \geq 0\}} \frac 1 2 r^{1+\delta} (\p_v \eta_{I})^2 \de u \de \omega + \sum_{I \in I^{\leq n+4}_{\boldsymbol{K}}}\int_{\mathcal{R}_{0}^{u_1}} \frac {1+\delta} 2 \, r^{\delta}(\p_v \eta_{I})^2 \de u \de v \de \omega \\
    &\qquad +\sum_{I \in I^{\leq n+4}_{\boldsymbol{K}}}\int_{\mathcal{R}_{0}^{u_1}} \frac {1-\delta} 2 \, r^{\delta}|\slashed{\nabla} \eta_{I}|^2 \de u \de v \de \omega \leq C \varepsilon^2.
\end{align*}

By a completely analogous reasoning, we have, setting $p = \delta$, and integrating now on the region $\mathfrak{D}_{u_1}^{u_2} := \{(u,v): u_1 \leq u \leq u_2, u+v \geq 0 \}$, the corresponding estimate:
\begin{align}
    &\sum_{I \in I^{\leq n+4}_{\boldsymbol{K}}}\int_{\{u = u_2\}\cap \{t \geq 0\}} \frac 12 r^{\delta} (\p_v \eta_I)^2 \de u \de \omega + \sum_{I \in I^{\leq n+4}_{\boldsymbol{K}}}\int_{\mathfrak{D}_{u_1}^{u_2}} \frac {\delta} 2 r^{\delta-1}(\p_v \eta_I)^2 \de u \de v \de \omega \nonumber\\
    &+\sum_{I \in I^{\leq n+4}_{\boldsymbol{K}}}\int_{\mathfrak{D}_{u_1}^{u_2}} \frac {2-\delta} 2 r^{\delta-1}|\slashed{\nabla} \eta_I|^2 \de u \de v \de \omega \nonumber\\ 
    &\qquad \leq \sum_{I \in I^{\leq n+4}_{\boldsymbol{K}}} \int_{\{u = u_1\}\cap \{t \geq 0\}} \frac 12 r^{\delta} (\p_v \eta_I)^2 \de u \de \omega \label{eq:rpdelta}\\
    &\qquad +\sum_{I \in I^{\leq n+4}_{\boldsymbol{K}}} \sum_{H+K \subset I} \int_{\mathfrak{D}_{u_1}^{u_2}} \big|r F_{HK}(d \der^H \zeta, d^2 \der^K  \zeta) - r G_{HK}(d \der^H\zeta, d\der^K\zeta)\big|r^{\delta} |\p_v \eta_I| \de u \de v \de \omega \nonumber\\
    &\qquad + C\sum_{I \in I^{\leq n+4}_{\boldsymbol{K}}} \int_{\{t=0\} \cap \{u_1 \leq u \leq u_2 \}} r^{\delta} ((\p_v \eta_I)^2 + |\slashed \nabla \eta_I|^2) \de r \de \omega. \nonumber
\end{align}
We deal with the error terms on the RHS of the previous display exactly as in the case $p = 1+\delta$, obtaining the inequality:
\begin{equation}\label{eq:rpdeltafin}
\begin{aligned}
    &\sum_{I \in I^{\leq n+4}_{\boldsymbol{K}}} \int_{\{u = u_2\}\cap \{t \geq 0\}} \frac 12 r^{\delta} (\p_v \eta_I)^2 \de u \de \omega + \sum_{I \in I^{\leq n+4}_{\boldsymbol{K}}}\int_{\mathfrak{D}_{u_1}^{u_2}} \frac {\delta} 2 r^{\delta-1}(\p_v \eta_I)^2 \de u \de v \de \omega\\
    &\qquad +\sum_{I \in I^{\leq n+4}_{\boldsymbol{K}}}\int_{\mathfrak{D}_{u_1}^{u_2}} \frac {2-\delta} 2 r^{\delta-1}|\slashed{\nabla} \eta_I|^2 \de u \de v \de \omega\\ 
    &\qquad \leq \sum_{I \in I^{\leq n+4}_{\boldsymbol{K}}}\int_{\{u = u_1\}\cap \{t \geq 0\}} \frac 12 r^{\delta} (\p_v \eta_I)^2 \de u \de \omega + C\sum_{I \in I^{\leq n+4}_{\boldsymbol{K}}}\int_{\{t=0\} \cap \{u_1 \leq -r \leq u_2 \}} r^{\delta} ((\p_v \eta)^2 + |\slashed \nabla \eta_I|^2) \de r \de \omega\\
    &\qquad + C \eps^2(1+|u_1|)^{-1}.
\end{aligned}
\end{equation}
Note that these estimates are valid also after commuting with the Laplacian on the conformal sphere $\mathbb{S}^2$:
\begin{equation*}
    \slashed \Delta_{\mathbb{S}^2} f :=  \frac 1 {\sin \theta} \p_\theta (\sin \theta \, \p_\theta f ) + \p^2_\varphi f.
\end{equation*}

{\bf Step 2}. From {\bf Step 1}, commuting with the Laplacian on the conformal sphere $\mathbb{S}^2$, we have the following estimates, for $0 \leq u_1 \leq u_2$:

\begin{align}
    &\sum_{I \in I^{\leq n+2}_{\boldsymbol{K}}} \int_{\{u = u_1\}\cap \{t \geq 0\}} \frac 12 r^{1+\delta} ((\p_v \eta_I)^2 + (\p_v \slashed \Delta_{\mathbb{S}^2}\eta_I)^2)\de u \de \omega  \leq C \varepsilon^2, \label{eq:fordyadic1}\\
    &\sum_{I \in I^{\leq n+2}_{\boldsymbol{K}}} \int_{\mathcal{R}_{0}^{\infty}} r^{\delta}((\p_v \eta_I)^2 +(\p_v \slashed \Delta_{\mathbb{S}^2}\eta_I)^2 +|\slashed{\nabla} \eta_I|^2+|\slashed{\nabla} \slashed{\Delta}_{\mathbb{S}^2}\eta_I|^2) \de u \de v \de \omega \leq C \varepsilon^2,\label{eq:fordyadic2}\\
    &\sum_{I \in I^{\leq n+2}_{\boldsymbol{K}}} \int_{\{u = u_2\}} r^{\delta} ((\p_v \eta_I)^2 + (\p_v \slashed \Delta_{\mathbb{S}^2} \eta_I)^2) \de u \de \omega \nonumber \\
    &\quad \leq C \sum_{I \in I^{\leq n+2}_{\boldsymbol{K}}}\int_{\{u = u_1\}} r^{\delta} ((\p_v \eta_I)^2 + (\p_v \slashed \Delta_{\mathbb{S}^2} \eta_I)^2) \de u \de \omega + C \eps^2(1+|\bar u|)^{-1}. \label{eq:fordyadic3}
\end{align}
Using display~\eqref{eq:fordyadic2}, we have, along a dyadic sequence $(u_n)_{n \in \N}$, such that $2^n \leq u_n \leq 2^{n+1}$,
\begin{equation}\label{eq:ondyadic}
    \sum_{I \in I^{\leq n+2}_{\boldsymbol{K}}} \int_{\{u = u_n\}} r^{\delta} ((\p_v \eta_I)^2 + (\p_v \slashed \Delta_{\mathbb{S}^2} \eta_I)^2) \de u \de \omega \leq C \eps^2 (1+|u_n|)^{-1}.
\end{equation}
With the aid of~\eqref{eq:fordyadic3} we now remove the restriction to the dyadic sequence: for all $\bar u \geq 0$, we have:
\begin{equation}\label{eq:allu}
    \sum_{I \in I^{\leq n+2}_{\boldsymbol{K}}} \int_{\{u = \bar u\}} r^{\delta} ((\p_v \eta_I)^2 + (\p_v \slashed \Delta_{\mathbb{S}^2} \eta_I)^2) \de u \de \omega \leq C \eps^2 (1+|\bar u|)^{-1}.
\end{equation}
Now, interpolating (using H\"older's inequality) between estimates~\eqref{eq:fordyadic1} and~\eqref{eq:allu}, we have, for some $\delta'$ such that $0 < \delta' < \delta$,
\begin{equation}\label{eq:deltaprime}
    \begin{aligned}
        \sum_{I \in I^{\leq n+2}_{\boldsymbol{K}}}\int_{\{u = \bar u\}} r^{1+\delta'} ((\p_v \eta_I)^2 + (\p_v \slashed \Delta_{\mathbb{S}^2} \eta_I)^2) \de u \de \omega \leq C \eps^2 (1+|\bar u|)^{-\delta'}.
    \end{aligned}
\end{equation}

We now note the following basic estimate, which follows from the Sobolev embedding on the conformal sphere $\mathbb{S}^2$ plus H\"older's inequality:
\begin{equation}\label{eq:sobolevrp}
    \begin{aligned}
        &|\eta_I(u,v,\omega)| \leq C \int_{\mathbb{S}^2} |\eta_I(u,v,\omega)| + |\slashed \Delta_{\mathbb{S}^2} \eta_I(u,v,\omega)| \de \omega \\
        &\qquad \leq C \Big(\int_{u}^v \int_{\mathbb{S}^2} r^{1+ \delta'}(|\p_v \eta_I(u,\bar v,\omega)|^2 + |\p_v \slashed \Delta_{\mathbb{S}^2}  \eta_I(u,\bar v,\omega)|^2)\de \bar v \de \omega \Big)^{\frac 12}.
    \end{aligned}
\end{equation}
(note that there is no boundary term at $r=0$ because $\eta = r \zeta$ and $\zeta$ is regular at $r=0$.

Combining now~\eqref{eq:deltaprime} with~\eqref{eq:sobolevrp}, we obtain
\begin{equation}\label{eq:uncommuted}
    \sum_{I \in I^{\leq n+2}_{\boldsymbol{K}}}|\eta_I(u,v,\omega)| \leq C \eps (1+|u|)^{-\delta'}.
\end{equation}
Therefore, we obtain:
\begin{equation}\label{eq:forvariousdec}
    \sum_{I \in I^{\leq n+2}_{\boldsymbol{K}}}|r \der^I \zeta(u,v,\omega)| \leq C \eps (1+|u|)^{-\delta'}.
\end{equation}
It then follows immediately that:
\begin{equation}\label{eq:impdeccone}
   |\p \der^I\zeta(u,v,\omega)| \leq C \eps (1+|u|)^{-\delta'-1}r^{-1}.
\end{equation}
for all $I \in I^{\leq n+1}_{\boldsymbol{K}}$.

It is evident that, defining $r'$ to be the radial distance from a point $p$ such that $|p|=1$, we can go into coordinates $(t, r', \omega')$, which then induce null coordinates $(u', v', \omega')$ in the usual way. It is then clear that we can repeat the above reasoning in {\bf Step 1} and in the current {\bf Step}, to obtain:
\begin{equation}\label{eq:rprime}
   |\p \der^I \zeta(u',v',\omega')| \leq C \eps (1+|u'|)^{-\delta'-1}(r')^{-1}.
\end{equation}
Combining this with estimate~\eqref{eq:impdeccone}, we then have
\begin{equation}\label{eq:impdec}
   |\p \der^I \zeta(u,v,\omega)| \leq C \eps (1+|u|)^{-\delta'-1}(1+r)^{-1}.
\end{equation}
Again, this holds for all $I \in I^{\leq n+1}_{\boldsymbol{K}}$.
Now, in the region where $u \leq cr$, we have that $r \geq c_1 v$, for some $c_1 > 0$, implying the first inequality in display~\eqref{eq:decphii} restricted to this spacetime region.

In order to complete {\bf Step 2} and the proof of the first estimate in~\eqref{eq:decphii}, we are then left with showing the estimate
\begin{equation}\label{eq:claiminterior}
    |\p \der^I \zeta(u,v,\omega)| \leq C \eps (1+|u|)^{-\delta''-2}
\end{equation}
in the region where $u \geq cr$, for some $c > 0$ and $\delta'' > 0$. Note that, in this spacetime region, there is a constant $c_2 > 0$ such that $c_2 t \leq u \leq t$.

We deduce immediately from equation~\eqref{eq:rprime} the following, which holds for all $I \in I^{\leq n+1}_{\boldsymbol{K}}$:
\begin{equation}\label{eq:forspaceint}
   |\p_r \der^I \zeta(t,r,\omega)| \leq C \eps (1+|u|)^{-\delta'-1}(1+r)^{-1} \leq C \eps (1+t)^{-\delta'-1}(1+r)^{-1}.
\end{equation}
We now proceed to integrate equation~\eqref{eq:forspaceint} radially, to get, for all $I \in I^{\leq n+1}_{\boldsymbol{K}}$:
\begin{equation}\label{eq:radialint}
    |\der^I \zeta(t,r,\omega)| \leq \int^{ r_0}_{r} |\p_r \zeta(t, \bar r, \omega)| \de \bar r + |\zeta(t, r_0, \omega)|.
\end{equation}
Here, we have chosen the number $r_0$ such that $r_0 \geq r$, and furthermore $t-r_0 = cr_0$. Now, using estimate~\eqref{eq:uncommuted} at the point $(t,r_0,\omega)$, we have
\begin{equation}\label{eq:forintconcl}
    |\der^I \zeta(t, r_0, \omega)| \leq C \eps (1+t)^{-1-\delta'}.
\end{equation}
Furthermore, we have 
\begin{equation}\label{eq:evalint}
\int^{r_0}_{r} |\p_r \der^I \zeta(t, \bar r, \omega)| \de \bar r \leq C \eps \log(1+r_0)(1+t)^{-1-\delta'} \leq C \eps (1+t)^{-1-\delta''}.
\end{equation}
Here, $\delta''$ satisfies $0 < \delta'' < \delta'$.

Combining now displays~\eqref{eq:radialint},~\eqref{eq:forspaceint}, and~\eqref{eq:evalint} we have, in the region where $u \geq cr$,
\begin{equation*}
    \sum_{I \in I^{\leq n+1}_{\boldsymbol{K}}}| \der^I \zeta(u,v,\omega)| \leq C \eps (1+|u|)^{-1-\delta''}.
\end{equation*}
This immediately yields
claim~\eqref{eq:claiminterior}, which concludes {\bf Step 2}.

{\bf Step 3}. The improved decay for the ``good'' derivatives follows again from  equation~\eqref{eq:forvariousdec} and the following inequality:
\begin{equation*}
    |\bar \p^{(0)} f| \leq C \frac{1}{1+v} \sum_{H\in \boldsymbol{K}} |\der^H f|.
\end{equation*}
This shows the first inequality in display~\eqref{eq:decphii}.
\end{proof}

\subsection{An important change of coordinates} We shall now define a coordinate system which will be useful when computing the interaction between waves originating from different points in space, and we prove a change of variables formula for this coordinate system. We suppose without loss of generality here that $i = 1$, $j = 2$, and that 
$$
p_1 = (-1,0,0), \qquad 
p_2 = (1,0,0),
$$
so that we focus on interaction of waves emanating from resp.~the points $(-R,0,0)$ and $(R,0,0)$. This Lemma is used in proving the improved energy estimates in Section \ref{sub:improvedenergy}.

\begin{lemma}\label{lem:r1r2}
	On $\R^4$, parametrized by cartesian coordinates $(t,x,y,z)$ we consider the coordinate system $(t,r_1, r_2, \varphi)$ defined by the following relations:
	\begin{equation}
	\begin{aligned}
	&t = t,\\
	&(x+R)^2 + \rho^2 = r_1^2,\\
	&(x-R)^2 + \rho^2 = r_2^2,\\
	&\varphi = \arctan(z/y).
	\end{aligned}
	\end{equation}
	Here, $\rho^2 := y^2 + z^2$. We also consider null coordinates defined by:
	\begin{equation}
	u_h := t-r_h, \quad v_h := t+r_h, \qquad \text{where } h \in \{1,2\}.
	\end{equation}
	Then, for every smooth and integrable function $f:\R^4 \to \R$, the following change of variable formulas hold:
	\begin{equation}\label{eq:changevars}
		\begin{aligned}
			&\int_{\R^4} f(t,x,y,z) \de t \de x \de y \de z = \int_{\R \times \mathcal{R} \times [0,2\pi)}f(t,r_1, r_2,\varphi)\frac {r_1 r_2} {2R}  \de t \de r_1 \de r_2 \de \varphi,\\
			&\int_{\R^4} f(t,x,y,z) \de t \de x \de y \de z = \int_{\mathcal{S}\times [0,2\pi)}f(u_1, v_1, u_2, \varphi)\frac {r_1 r_2} {4R}  \de u_1 \de v_1 \de u_2 \de \varphi,
		\end{aligned}
	\end{equation}
	where $\mathcal{R}$ is the set of those values of $r_1$ and $r_2$ which satisfy:
	$$
	r_1 + r_2 \ge 2 R, \qquad |r_1 - r_2| \leq 2R.
	$$
	Furthermore, $\mathcal{S}$ is the set of those values of $u_1, v_1$ and $u_2$ such that
	$$
	v_1 - u_2 \geq 2R, \qquad |u_2 - u_1| \leq 2R.
	$$
	Finally, the following formulas for integrals on null cones hold true:
\begin{equation}\label{eq:nullconescc}
\begin{aligned}
	&\int_{C^{(1)}_{\bar u_1}} f \ \de S_{C^{(1)}_{u_1}}= \int_{\mathcal{S}_{\bar u_1} \times [0,2\pi)} f(v_1, u_2, \varphi) \frac{r_1 r_2}{4R}\de v_1 \de u_2 \de \varphi,\\
	&\int_{C^{(2)}_{\bar u_2}} f \ \de S_{C^{(2)}_{u_2}}= \int_{\mathcal{S}_{\bar u_2} \times [0,2\pi)} f(v_2, u_1, \varphi) \frac{r_1 r_2}{4R}\de v_2 \de u_1 \de \varphi.
\end{aligned}
\end{equation}
Here, $C^{(1)}_{\bar u_1}$ is the cone $\{u_1 = \bar u_1\} \cap \{t \geq 0\}$, and similarly $C^{(2)}_{\bar u_2}$ is the cone $\{u_2 = \bar u_2\} \cap \{t \geq 0\}$. Furthermore, $\de S_{C^{(h)}_{u_h}}$, $h =1, 2$ is the volume form induced on cones of constant $u_h$ coordinate, with $h =1, 2$.
Moreover, $\mathcal{S}_{\bar u_1}$ is the set of those values of $v_1$, $u_2$ such that 
$$
v_1 - u_2 \geq 2R, \qquad v_1 \geq - \bar u_1, \qquad |u_2 - \bar u_1| \leq 2R,
$$
and, similarly, $\mathcal{S}_{\bar u_2}$ is the set of those values of $v_2$, $u_1$ such that
$$
v_2 - u_1 \geq 2R, \qquad v_2 \geq - \bar u_2, \qquad |u_1 - \bar u_2| \leq 2R.
$$
\end{lemma}
\begin{remark}
    We note that $r_1$ is the Euclidean distance in $\R^3$ to the point $Rp_1 = (-R, 0,0)$, and similarly $r_2$ is the Euclidean distance in $\R^3$ to the point $Rp_2 = (R,0,0)$.
\end{remark}
\begin{proof}[Proof of Lemma~\ref{lem:r1r2}]
Let $\rho^2 := y^2 + z^2$. On $\R^3$, we consider the following coordinates $(r_1, r_2, \varphi)$:
\begin{equation}
\begin{aligned}
	&r_1^2 = (x+R)^2 + \rho^2,\\
	&r_2^2 = (x-R)^2 + \rho^2,\\
	&\varphi = \arctan(z/y).
\end{aligned}
\end{equation}
so that the relations hold:
\begin{equation}
\begin{aligned}
	&x = \frac{r_1^2 - r_2^2}{4R},\qquad y = \rho \cos \varphi, \qquad z= \rho \sin \varphi.
\end{aligned}
\end{equation}
where $\rho = \sqrt{\frac 12 \Big ( r_1^2 + r_2^2\Big) - R^2 -x^2}$.
Let us now compute the Jacobian determinant of such change of variables. We have
\begin{equation}
	J_1:= |\det( J_{(r_1, r_2, \varphi)}(x,y,z))| = 
	\left| 
	\begin{array}{ccc}
	\frac{r_1}{2R} & -\frac{r_2}{2R} & 0\\
	\partial_{r_1} \rho \cos \varphi & \partial_{r_2} \rho \cos \varphi & - \rho \sin \varphi\\
	\partial_{r_1} \rho \sin \varphi & \partial_{r_2} \rho \sin \varphi &  \rho \cos \varphi
	\end{array}
	\right|
\end{equation}
Expanding with respect to the first row,
\begin{equation*}
	J_1 =\Big| \frac{r_1}{4R} \partial_{r_2} \rho^2 + \frac{r_2}{4R} \partial_{r_1} \rho^2 \Big| = \frac 1 {4R} \Big|2 r_1 r_2 \Big| = \frac{r_1 r_2}{2R},
\end{equation*}
since we have
\begin{equation*}
	\begin{aligned}
		\partial_{r_1}\rho^2 = r_1 - \frac{r_1(r_1^2-r_2^2)}{4R^2}, \qquad \partial_{r_2}\rho^2 = r_2 + \frac{r_2(r_1^2-r_2^2)}{4R^2}
	\end{aligned}
\end{equation*}
Therefore, we get the following expression for the corresponding volume forms:
\begin{equation}
	\de x \wedge \de y \wedge \de z = \frac{r_1 r_2}{2R} \de r_1 \wedge \de r_2 \wedge \de \varphi.
\end{equation}
Let us now consider the product manifold $\R \times \R^3$, and let the first variable to be time $(t)$.

Now, we shall find the range of admissible $r_1$ and $r_2$. It is clear that $r_1 \ge 0$ and $r_2 \ge 0$. Moreover, because the distance between $R p_1$ and $R p_2$ is $2 R$ and $r_1$ is the distance from $R p_1$ while $r_2$ is the distance from $R p_2$, we have that $r_1 + r_2 \ge 2 R$ from the triangle inequality. Similarly, it follows that $|r_1 - r_2| \leq 2R$ from the triangle inequality, as desired (note that these inequalities alone ensure that both $r_1$ and $r_2$ are non-negative).

We now define $(u_1, v_1, u_2, \varphi)$ as follows:
\begin{equation*}
	u_1 := t-r_1, \quad v_1 := t+r_1, \quad u_2 := t-r_2.
\end{equation*}
This implies
\begin{equation*}
	t = \frac 12 (u_1 + v_1), \quad r_1 = \frac 12 (v_1 - u_1), \quad r_2 = \frac 12 (u_1 + v_1) - u_2.
\end{equation*}
The Jacobian determinant is now
\begin{equation*}
 J_2 := |\det (J_{(u_1, v_1, u_2, \varphi)}(t, r_1, r_2, \varphi))| = \left| 
 \begin{array}{cccc}
 	\frac 12 & \frac 12 & 0 & 0\\
 	\frac 12 & - \frac 12 & 0 & 0\\
 	\frac 12 & \frac 12 & - 1 & 0\\
 	0 & 0 & 0 & 1
 \end{array}
 \right| = \frac 12.
\end{equation*}
All in all, we obtain that the following formulas hold. We then have, if $f$ is a smooth function,
\begin{equation*}
	\begin{aligned}
	&\int_{\R^4} f(t,x,y,z) \de t \de x \de y \de z = \int_{\R \times \mathcal{R} \times [0,2\pi)}f(t,r_1, r_2,\varphi)\frac {r_1 r_2} {2R}  \de t \de r_1 \de r_2 \de \varphi\\
	&\int_{\R^4} f(t,x,y,z) \de t \de x \de y \de z = \int_{\mathcal{S} \times [0,2\pi)}f(u_1, v_1, u_2, \varphi)\frac {r_1 r_2} {4R}  \de u_1 \de v_1 \de u_2 \de \varphi
	\end{aligned}
\end{equation*}
where $\mathcal{R}$ is the subset of the values of $r_1$ and $r_2$ such that
$$
r_1 + r_2 \geq 2R, \qquad |r_1 - r_2 | \leq 2R.
$$
Translating these bounds into $(u_1, v_1, u_2, \varphi)$ coordinates we obtain that the second integral in the display above is over the set $\mathcal{S}\times [0, 2\pi)$, where $\mathcal{S}$ is the set of those values of $u_1, v_1, u_2$ which satisfy:
$$
	v_1 - u_2 \geq 2R, \qquad |u_2 - u_1| \leq 2R.
$$
This concludes the proof of the change of variables in display~\eqref{eq:changevars}.

The proof of~\eqref{eq:nullconescc} follows in a straightforward way.
\end{proof}

\subsection{Asymptotic comparison of derivatives intrinsic to two distinct light cones}

Here, we wish to formalize the fact that, as time increases, ``good'' derivatives for both cones become aligned on the interaction set, thus giving rise to improved estimates. This Lemma is used in the proof of the improved energy estimates in Section \ref{sub:improvedenergy}.

\begin{lemma}\label{lem:goodbad}
    Let $c \in (0,1/10)$, let $\eta \in \mathcal{C}^\infty(\R^4)$, and recall the coordinates $r_1$ and $r_2$. There exists a constant $C > 0$ such that the following pointwise inequalities hold true:
    \begin{equation}\label{eq:goodbad}
        |\bar{\p}^{(1)} \eta| \leq C \frac R {r_1} |\p \eta| + C |\bar{\p}^{(2)}\eta|, \qquad |\bar{\p}^{(2)} \eta| \leq C \frac R {r_2} |\p \eta| + C |\bar{\p}^{(1)}\eta|.
    \end{equation}
    Moreover, let the spacetime region $\mathcal{I}_{12}$ be defined as the region where both $|u_1| \leq cR $ and $|u_2| \leq cR$. Restricting to the region $\mathcal{I}_{12}$, we have the following estimate, valid for every smooth function $\eta$:
    \begin{equation}\label{eq:igortrick}
    |\p \eta| \leq C \frac{t}{R} (|\overline \p^{(1)} \eta| + |\overline \p^{(2)} \eta|).
\end{equation}
\end{lemma}

\begin{proof}[Proof of Lemma~\ref{lem:goodbad}]
Let us focus on proving the first inequality in display~\eqref{eq:goodbad}, the second being analogous. Without loss of generality, let us furthermore assume that $r_1 \geq c R$, as the claim is clear in case $r_1\leq cR$.

We will derive an expression for good derivatives adapted to one of the cones in terms of the other. We have the following relation:
\begin{equation*}
	r_1^2 - r_2^2 = 4xR,
\end{equation*}
hence, taking the gradient of both sides, we have
\begin{equation*}
	r_1 \partial_{r_1} - r_2 \partial_{r_2} = 2R\partial_{x}.
\end{equation*}
Here, $\p_{r_h}$ is the coordinate vector field induced by $(t, r_h, \theta_h, \varphi_h)$ (defined in display~\eqref{eq:polarpi}), and $h \in \{1,2\}$. Recalling that $\partial_{v_h} = \partial_t + \partial_{r_h}$, $h = 1,2$:
\begin{equation*}
	\begin{aligned}
		&r_1 \partial_{v_1} = 2R \partial_x + r_2 \partial_{v_2} - (r_2-r_1)\partial_t  = 2 R \partial_x + r_2 \partial_{v_2} + \frac{4xR}{r_1 + r_2} \partial_t.
	\end{aligned}
\end{equation*}
We then have, from the triangle inequality, $|r_1 - r_2| \leq 2R$, which implies $\frac{r_2}{r_1} \leq 1 + \frac{2R}{r_1}$, which, by the fact that we have $r_1 \geq c R$, implies $\frac{r_2}{r_1} \leq 1 + \frac 2 c$, so that $$\frac{r_2}{r_1} \leq C,$$
for some positive constant $C$.
Similarly, we have
\begin{equation*}
	\Big| \frac{4xR}{r_1 + r_2} \Big | \leq C R,
\end{equation*}
due to the fact that $r_1 + r_2 \geq |x|$.
All in all, we obtain the following, if $\eta \in \mathcal{C}^\infty(\R^4)$:
\begin{equation}\label{eq:vchange}
|\partial_{v_1}\eta| \lesssim \frac R {r_1} |\partial \eta| + |\bar \partial^{(2)}\eta|.
\end{equation}
Similarly, for rotation vector fields,
\begin{equation*}
\Omega^{(1)}_{(xy)} =( x + R)\partial_y - y \partial_x = \Omega^{(2)}_{(xy)} + 2 R \partial_y.
\end{equation*}
Hence, 
\begin{equation}
\frac 1 {r_1} \Omega^{(1)}_{(xy)} =\frac{( x + R)}{r_1}\partial_y - y \partial_x = {1 \over r_1} \Omega^{(2)}_{(xy)} + {2 R \over r_1} \partial_x,
\end{equation}
and we easily conclude the proof of the first inequality in display~\eqref{eq:goodbad}. The proof of the second inequality is identical.

We now turn to the proof of bound~\eqref{eq:igortrick}. Let us first restrict to the case $t \geq 2 (1-c)^{-1}R$.
Recall the definition of the region $\mathcal{I}_{12} := \{|u_1| \leq cR \} \cap \{|u_2| \leq cR \}$. We note that
\begin{equation*}
    \p_y = \frac 1 {2R}(\Omega^{(1)}_{(xy)} - \Omega^{(2)}_{(xy)} ), \qquad \p_z = \frac 1 {2R}(\Omega^{(1)}_{(xz)} - \Omega^{(2)}_{(xz)} ).
\end{equation*}
From these, for any smooth function $\eta$, we obtain, in the region $\mathcal{I}_{12}$,
\begin{equation}\label{eq:yzbounds}
    |\p_z \eta| + |\p_y \eta| \leq C \frac{t}{R} (|\overline \p^{(1)} \eta| + |\overline \p^{(2)} \eta|).
\end{equation}
This is because we are assuming $R \leq \frac 1 2 (1-c) t$, which, together with the fact that $r_1, r_2 \leq cR + t$ implies $r_1, r_2 \leq Ct$, for some positive constant $C$.
Now, we also have that
\begin{equation*}
    y\p_x = - \Omega^{(2)}_{(xy)} +(x-R)\p_y, \qquad z\p_x = - \Omega^{(2)}_{(xz)} +(x-R)\p_z.
\end{equation*}
This implies:
\begin{equation*}
    (y^2+z^2)\p_x = - y\Omega^{(2)}_{xy} +(x-R)\p_y -z \Omega^{(2)}_{xz} +(x-R)\p_z.
\end{equation*}
We now note that
\begin{equation*}
    4xR = r_1^2 - r_2^2 \leq (t+cR)^2 - (t-cR)^2 = 4 c t R,
\end{equation*}
(and similarly for $-x$)
hence we have that $|x| \leq ct$. We also have that, as $r_1^2 +r_2^2 = 2x^2 + 2y^2 +2z^2 +2R^2,$ and $r_1, r_2 \geq t - cR$,
\begin{equation}\label{eq:yzineq}
y^2 + z^2 \geq (t-cR)^2 - R^2 - c^2 t^2= (1-c^2)t^2 -2cRt +(c^2-1)R^2.
\end{equation}
Now, one can verify that, for $t \geq 2(1-c)^{-1}R$, $$\frac 12 (1-c^2)t^2 -2cRt +(c^2-1)R^2\geq 0,$$ which implies, by~\eqref{eq:yzineq}, that $y^2 + z^2 \geq \frac 12 (1-c^2) t^2$.
Together with the fact that $|x|, |y|, |z|, r_1, r_2 \leq t+cR \leq Ct$ (as $t \geq 2(1-c)^{1}R$), we can conclude that, when restricting to the case $t \geq 2 (1-c)^{-1}R$,
\begin{equation*}
    |\p_x \eta| \leq C \frac{t}{R} (|\overline \p^{(1)} \eta| + |\overline \p^{(2)} \eta|).
\end{equation*}
Now, we consider the identity, which follows from the definition of $r_1$ and $r_2$,
\begin{equation*}
r_1^2 +r_2^2 = 2(x^2 +y^2 +z^2 +R^2).    
\end{equation*}
Taking the gradient of this expression, and adding $(r_1+r_2)\p_t$ on both sides, we obtain
\begin{equation*}
    \p_t = \frac 1 {r_1+r_2}(r_1 \p_{v_1} + r_2 \p_{v_2} -2 (x\p_x + y\p_y+ z\p_z)).
\end{equation*}
We conclude using the previous bounds obtained on $\p_x$, $\p_y$ and $\p_z$:
\begin{equation*}
    |\p_t \eta| \leq C \frac{t}{R} (|\overline \p^{(1)} \eta| + |\overline \p^{(2)} \eta|),
\end{equation*}
always restricting to the case $t \geq 2 (1-c)^{-1}R$. This proves claim~\eqref{eq:igortrick}, when restricting to the case $t \geq 2 (1-c)^{-1}R$

Let us now turn to the case $t \leq 2 (1-c)^{-1}R$. The bounds~\eqref{eq:yzbounds} for $\p_y$ and $\p_z$ are still valid (note that every point in the region $\mathcal{I}_{12}$ satisfies $t \geq c'R$, for some positive constant $c'$). Moreover, we have, as before,
\begin{align*}
&\p_t = \frac 1 {r_1+r_2}(r_1 \p_{v_1} + r_2 \p_{v_2} -2 (x\p_x + y\p_y+ z\p_z)),\\
&\p_x = \frac 1 {2R} (r_1 \p_{v_1} -r_2\p_{v_2} +(r_2-r_1)\p_t).
\end{align*}
Substituting, we have
\begin{equation}\label{eq:prexder}
\big(1 + \frac{x(r_2-r_1)}{R(r_1+r_2)} \big)\p_t = \frac 1 {r_1+r_2}(r_1 \p_{v_1} + r_2 \p_{v_2} -2 ( y\p_y+ z\p_z))-\frac x {(r_1 +r_2)R} (r_1 \p_{v_1} -r_2\p_{v_2}).
\end{equation}
We then note that, since $r_1^2 - r_2^2 = 4xR$,
$$
1+\frac{x(r_2-r_1)}{R(r_1+r_2)} = 1  -\frac{4x^2}{(r_1+r_2)^2}\geq 1 - \frac{2x^2}{x^2 + R^2}.
$$
Now, if we can prove that that $|x| \leq \frac 12 R$, $1 - \frac{2x^2}{x^2 + R^2}$ would be bounded below by $1/2$, and we would use equation~\eqref{eq:prexder} to conclude. Now, we have
$$
|x| = \frac 1 {4R}|r_1^2 - r_2^2| \leq ct \leq 2 c (1-c)^{-1}R < \frac 12 R,
$$
since $c < \frac 1 {10}$.
Hence equation~\eqref{eq:prexder} implies:
$$
|\p_t \eta| \leq C  (|\overline \p^{(1)} \eta| + |\overline \p^{(2)} \eta|),
$$
which is the claim for the $\p_t$ derivative restricted to the region $t \leq 2(1-c)^{-1}R$. We finally use the relation
$$
\p_x = \frac 1 {2R} (r_1 \p_{v_1} -r_2\p_{v_2} +(r_2-r_1)\p_t)
$$
to deduce the claim for the $\p_x$ derivative.
This concludes the proof of the lemma.
\end{proof}

\section{\texorpdfstring{$R$}{R}-weighted Klainerman--Sobolev inequalities}\label{sec:sobolevs}

We require two different sets of global Sobolev inequalities depending on whether we are seeking to obtain estimates on the linear equations for $\psi_{ij}$ or the nonlinear equation. In the estimates for $\psi_{ij}$, because the main interactions come from the two functions $\phi_i$ and $\phi_j$, there is a rotation that does not introduce $R$ weights for $\psi_{ij}$. Indeed, the centers of the supports of $\phi_i$ and $\phi_j$ lie on a line, and rotations about this line do not introduce $R$-weights in the initial data for $\phi_i$ and $\phi_j$. Similarly, the Lorentz boosts tangent to certain two-dimensional hyperboloids which are translation invariant along this line do not produce $R$-weights on the initial data either. On the other hand, in the estimates for the nonlinear problem, all Killing vector fields introduce $R$-weights in the initial data. Therefore, we need to use the usual Killing vector fields divided by the parameter $R$ (see also the discussion in Section~\ref{subsub:rweightedintro}).

\subsection{Sobolev inequalities for the linear equations}

We list here the modified Klainerman--Sobolev inequalities we need to use in this setting. We begin with the estimates that are used for $\psi_{ij}$. Without loss of generality (by changing coordinates), we can assume that the line connecting the centers of the supports of $\phi_i$ and $\phi_j$ is just the $x$ axis, as was done in the above, so that the two pieces of initial data are localized around the point $Rp_1 = (-R, 0,0)$ and $Rp_2 = (R, 0,0)$. These estimates are used in Section~\ref{sub:linfpsiij}.
\begin{lemma}\label{lem:sobspheres}
	Let $f \in C^\infty (\R^3)$. Consider polar coordinates $(r, \theta, \varphi)$ adapted to the $x$-axis, so that $r > 0$, $\theta \in [0,\pi)$, $\varphi \in [0, 2 \pi)$, and so that $\p_\varphi = y \p_z - z\p_y$. There exists a a positive constant $C$ such that the following holds. We have that, for $\theta \in \big[\pi/8, 7 \pi/ 8\big]$,
	\begin{equation}\label{eq:presobspheres}
	\begin{aligned}
	|f(r, \theta, \varphi)|^2 \leq C \frac R {r^2}  \sum_{I \in I^{\leq 3}_{\Ga^{(0)}}} \Vert \der^I f \Vert_{L^2(\R^3)}^2.
	\end{aligned}
	\end{equation}
	Here, recall the definition of the vector fields $\Ga^{(0)}$ in display~\eqref{eq:gammah}.
	
	This implies that, for $h \in \{0, i, j\}$, we have that, restricting to the region where
	$\rho = \sqrt{y^2 + z^2} \geq \frac 1 {10} t $,
	\begin{equation}\label{eq:sobspheres}
	\begin{aligned}
	|f(t,x,y,z)|^2 \leq C \frac R {t^2}  \sum_{I \in I^{\leq 3}_{\Ga^{(h)}}} \Vert \der^I f \Vert_{L^2(\Sigma_t)}^2.
	\end{aligned}
	\end{equation}
\end{lemma}

\begin{proof}[Proof of Lemma~\ref{lem:sobspheres}]
	Let us consider the sphere $\mathbb{S}^2$ with coordinates $(\theta, \varphi)$, so that $\theta = 0$ corresponds to a point lying on the positive $x$-axis.
	Let $\chi(\theta)$ be a smooth cutoff function such that
	$$
	\chi(\theta) = \left\{\begin{array}{l}
	    1 \qquad \text{if} \quad \theta \in  [\pi/8, 7 \pi/ 8],  \\
	    0 \qquad \text{if} \quad \theta \in  [0, \pi/16) \cup (15 \pi /16, \pi].
	\end{array} \right.
	$$
	We shall use a localized Sobolev embedding on the unit sphere. We have that $\partial_\varphi = \Omega_{(yz)} =  y \p_z - z\p_y$. With $\partial_\theta^R := R^{-1} \partial_\theta$, we now estimate
	\begin{equation*}
	\begin{aligned} 
		&|\chi(\theta) f(r, \theta, \varphi)|^2 \leq 2\int_0^\theta |f \partial_\theta \chi| |\chi f| + \chi |\partial_\theta f| |\chi f|\de \vartheta \\ &\quad \leq  2 \frac R r \int_0^\theta \left ( \left |f {1 \over R} \partial_\theta \chi \right | \left |\chi f \right | + \left | \chi {1 \over R} \partial_\theta f \right | \left | \chi f \right | \right ) r \de \vartheta \\
		&\quad \le 2 \frac R r \int_0^\theta \left ((f \partial_\theta^R \chi)^2 + 2 (\chi f)^2 + (\chi \partial_\theta^R f)^2 \right ) r \de \vartheta.
	\end{aligned}
	\end{equation*}
	Multiplying and dividing by $r$ and integrating over $\mathbb{S}^1$ in $\varphi$, we have that
	\begin{align}\label{eq:sphere1}
	\int_{\mathbb{S}^1} |\chi(\theta) f(r,\theta,\varphi)|^2 d \varphi \le C \frac{R}{r^2} \int_0^\theta \int_{\mathbb{S}^1} \left ((f \partial_\theta^R \chi)^2 + 2 (\chi f)^2 + (\chi \partial_\theta^R f)^2 \right ) r^2 \de \varphi \, \de \vartheta
	\end{align}
	Now, taking $\partial_\varphi$, we have that
	\begin{equation}\label{eq:sphere2}
	\begin{aligned}
	&\int_{\mathbb{S}^1} |\partial_\varphi (\chi(\theta) f(r, \theta,\varphi))|^2 \de \varphi = \int_{\mathbb{S}^1} |\chi(\theta) \partial_\varphi f(r,\theta,\varphi)|^2 \de \varphi \\ 
	&\quad \le C  {R \over r^2} \int_0^\theta \int_{\mathbb{S}^1} \left ((\partial_\varphi f \partial_\theta^R \chi)^2 + 2 (\chi \partial_\varphi f)^2 + (\chi \partial_\theta^R \partial_\varphi f)^2 \right ) r^2 \de \varphi \, \de \vartheta.
	\end{aligned}
	\end{equation}
	Now, we recall the Sobolev inequality on $\mathbb{S}^1$ given by
	\begin{equation*}
	\Vert f \Vert_{L^\infty(\mathbb{S}^1)}^2 \le C\big( \Vert f \Vert_{L^2(\mathbb{S}^1)}^2 + \Vert \partial_\varphi f \Vert_{L^2(\mathbb{S}^1)}^2\big).
	\end{equation*}
	Treating $\chi f$ as a function on $\mathbb{S}^1$ in $\theta$ and $\varphi$ and $r$ fixed along with using both~\eqref{eq:sphere1} and~\eqref{eq:sphere2} gives us that
	\begin{align}
	\Vert \chi f \Vert_{L^\infty}^2 \le C {R \over r^2} \left (\Vert f \Vert_{L^2(S_r)}^2 + \Vert \partial_\theta^R f \Vert_{L^2(S_r)}^2 + \Vert \partial_\varphi f \Vert_{L^2(S_r)}^2 + \Vert \partial_\theta^R \partial_\varphi f \Vert_{L^2(S_r)}^2 \right )
	\end{align}
	with $C$ depending on $\chi$. We note that we have used the fact that $0 < c \le \sin{(\theta)}$ on the support of $\chi$, as $\chi(\theta) = 0$ for $\theta \le {\pi \over 16}$ and for $\theta \ge {15 \pi \over 16}$. Here, $S_r$ is the sphere of radius $r$ centered at the origin in $\R^3$.
	
	We now have that
	\begin{equation*}
	\partial_\theta^R = a_1(x,y,z) \Omega_{(xy)} + a_2(x,y,z) \Omega_{(xz)} + a_3 (x,y,z) \Omega_{(yz)},
	\end{equation*}
	with $a_1$, $a_2$, and $a_3$ smooth functions on $\R^3$ which are all pointwise controlled by $\frac C R$. Thus, using H\"older's inequality, we have that
	\begin{equation*}
	\Vert \partial_\theta^R f \Vert_{L^2(S_r)}^2 \le C \left (\Vert \Omega^R_{(xy)} f \Vert_{L^2 (S_r)}^2 + \Vert \Omega^R_{(xz)} f \Vert_{L^2 (S_r)}^2 + \Vert \Omega_{(y z)} f \Vert_{L^2 (S_r)}^2 \right ).
	\end{equation*}
	for any $f \in \mathcal{C}^\infty(S_r)$. Here, we denoted $\Omega^R_{(xy)} := R^{-1}\Omega_{(xy)}$, and $\Omega^R_{(xz)} := R^{-1} \Omega_{(xz)}$.
	The claim is then easily obtained by the trace lemma in the $r$-direction (Lemma~\ref{lem:radialtrace}), noting that, for a smooth function $f$, we have the pointwise inequality $|\p_r f| \leq C |\p f|$.
\end{proof}

Now, we must also get estimates using the $x$ translation invariant hyperboloids.
Recall the hyperboloidal coordinates $(\tau,\alpha,x,\varphi)$ introduced in display~\eqref{eq:xflathyp}. Recall furthermore that we defined $\rho := \sqrt{y^2 +z^2}$.
Then, we have the following lemma, which is used in Section~\ref{sub:linfpsiij}.

\begin{lemma}\label{lem:sobhyp}
    There is a positive constant $C$ such that the following inequality holds, for all $(t,x,y,z)$ such that $t\geq 1$ and $\rho = \sqrt{y^2+z^2 }\leq \frac{1}{10} t$:
    \begin{equation}
        |f(t,x,y,z)| \leq  \frac C t  \sum_{I \in I^{\leq 3}_{\Ga^{(h)}}}\Vert  \der^I f\Vert_{L^2(H_\tau \cap \{\rho \leq \frac 1 {10} t\})}.
    \end{equation}
    Here, $h \in \{0,i,j\}$, and $\tau$ satisfies $\tau = \sqrt{t^2 - \rho^2}$.
\end{lemma}

\begin{proof}[Proof of Lemma~\ref{lem:sobhyp}]
Consider coordinates $(\tau, \bar y, \bar z)$ such that
\begin{equation*}
    t = \tau \sqrt{1 + \bar y^2 + \bar z^2}, \qquad y = \tau \bar y, \qquad z = \tau \bar z.
\end{equation*}
We then note that the coordinate vector fields $\p_{\bar y}$ and $\p_{\bar z}$ are parallel to the Lorentz boosts $\Gamma_{(ty)}$ and $\Gamma_{(tz)}$:
\begin{equation}\label{eq:gammalocalcoord}
    \p_{\bar y} = \frac \tau t (y \p_t + t \p_y), \qquad \p_{\bar z} = \frac \tau t (z \p_t + t \p_z).
\end{equation}
We then consider the function
$$
\tilde f(\bar y, \bar z) := f(\tau \sqrt{1+\bar y^2 + \bar z^2},  \tau \bar y, \tau \bar z).
$$
We then use the following version of the Sobolev embedding, valid for all $(\bar y, \bar z)$ in the ball $\tilde B$, defined as the set where $\bar y^2 + \bar z^2 \leq 1/99$:
\begin{equation}\label{eq:forhyps}
|\tilde f(\bar y, \bar z)|^2 \leq C\big( \Vert \tilde f\Vert^2_{L^2(\tilde B)} + \Vert  \p^2_{\bar y}\tilde f\Vert^2_{L^2(\tilde B)} + \Vert  \p^2_{\bar z}  \tilde f\Vert^2_{L^2(\tilde B)}\big).
\end{equation}
The claim then follows by changing variables in the integrals appearing in display~\eqref{eq:forhyps}, using the expression~\eqref{eq:gammalocalcoord}, and applying the trace lemma in the $x$-direction (Lemma~\ref{lem:tracehyp}). Finally, the restriction $\bar y^2 + \bar z^2 \leq 1/99$ translates to $t^2 \leq 100/99\,  \tau^2$, which implies $\rho \leq \frac{1}{10} t$.
\end{proof}

\subsection{Sobolev inequalities for the nonlinear equation}
Finally, we turn to the $L^\infty$ estimates needed to close the bootstrap argument for the nonlinear equation in the proof of Theorem~\ref{thm:nonlinear}. In this case, we note that all the Lorentz vector fields will introduce $R$-weights on initial data. For this purpose, we are going to be using the vector fields $\boldsymbol{K}_R$, as defined in display~\eqref{eq:krcenter}. All of these vector fields have $\frac 1 R$ weights. However, following the discussion above in Section~\ref{subsub:rweightedintro}, it is too wasteful to naively use the classical Klainerman--Sobolev inequality introducing these weights. Therefore, we shall now reprove the Klainerman--Sobolev inequality in this modified setting, being careful to keep track of the $R$ weights. The first result we prove is suitable to gain additional weights in a region of large $r$-coordinate. This estimate is used in Section~\ref{sub:linfpsiij}.

\begin{lemma}\label{lem:bettersob1}
Let $f:\R^{1+3} \to \R$ be a smooth function. Then, we have the estimate
\begin{equation}\label{eq:krsobolev}
    \begin{aligned}
        |f(t,r,\theta, \varphi)|^2 \le C { R^2 \over {(1+r)^2}} \sum_{I \in I^{\leq 3}_{\boldsymbol{K}_R}}\Vert  \der^I f \Vert^2_{L^2(\Sigma_t)}.
    \end{aligned}
\end{equation}
Here, we used the definition of the set $\boldsymbol{K}_R$ in display~\eqref{eq:krcenter}.
\end{lemma}

\begin{proof}[Sketch of proof of Lemma~\ref{lem:bettersob1}]
  We note that this can be proven in exactly the same way as Lemma~\ref{lem:sobspheres} above by appropriately replacing every occurrence of $\p_\varphi$ with $\p_\varphi/R$ (we note that in this case, we also think of $\p_\varphi$ as introducing a bad $R$-weight), upon dividing all inequalities by a factor of $R$. When we take the square root, the final inequality we obtain is thus worse by a factor of $R$.
\end{proof}

For $r$ very small, we need the following estimate, whose proof follows the proof of the Klainerman--Sobolev inequality in the analogous region (see Section 9 of~\cite{jonathannotes} and also \cite{sogge}). This estimate is used in Section~\ref{sec:mainproof}, when we prove the main theorem.

\begin{lemma}\label{lem:bettersob2}
Let $f:\R^{1+3} \to \R$ be a smooth function. We fix a smooth, positive, even function $\chi : \R \rightarrow \R$ with $\chi = 1$ for $|x| \le {4 \over 5}$ and with $\chi = 0$ for $|x| \ge {9 \over 10}$. Then, there exists a constant $C >0 $ such that the following estimate holds:
\begin{equation}\label{eq:sob32}
    \begin{aligned}
        \left |\chi \left ({r \over t} \right ) f(t,r,\theta, \varphi) \right | \le {C R^{{3 \over 2}} \over t^{{3 \over 2}}} \sum_{I \in I^{\leq 3}_{\boldsymbol{K}_R}}\Vert  \der^I f \Vert_{L^2(\Sigma_t)}.
    \end{aligned}
\end{equation}
Moreover, as a result of this, in the region where $t \ge c R$ where $c$ is some constant, we have that
\begin{equation}\label{eq:sob11}
    \begin{aligned}
        \left |\chi \left ({r \over t} \right ) f(t,r,\theta, \varphi) \right | \le {C R \over t} \sum_{I \in I^{\leq 3}_{\boldsymbol{K}_R}}\Vert  \der^I f \Vert_{L^2(\Sigma_t)}. 
    \end{aligned}
\end{equation}
\end{lemma}
\bp
We recall the following Sobolev inequality on $\R^3$, for smooth and compactly supported functions $g$:
\begin{equation*}
    \Vert g\Vert_{L^\infty(\R^3)}\leq C \Vert \p^2 g \Vert^{\frac 34}_{L^2(\R^3)}  \Vert g \Vert^{\frac 14}_{L^2(\R^3)}.
\end{equation*}
Now, by Lemma~9.6 in the lecture notes~\cite{jonathannotes}, we have that
\begin{equation*}
    |\p^2 g| \leq \frac C {|t-r|^2} \sum_{I \in I^{\leq 2}_{\boldsymbol{K}_R}} |\der^I g|.
\end{equation*}
We then combine the previous two displays, choosing $g(t,r,\theta,\varphi) = \chi\big( \frac r t \big) f(t,r,\theta,\varphi)$. We then apply the chain rule, and use the fact that in the region considered we have $u 
\geq ct$ for some positive constant $c$. Upon an application of H\"older's inequality, we conclude.
\ep

We finally recall the classical estimate by Klainerman, appropriately modified to be used with $R$-weighted vector fields:
\begin{lemma}[Classical Klainerman--Sobolev inequality with $R$-weights]\label{lem:classicalks}
There exists a constant $C > 0$ such that, for all $f$ smooth functions on $\R^{1+3}$, parametrized by coordinates $(t, r, \theta, \varphi)$, and for all $R \geq 1$, we have the following inequality:
\begin{equation}\label{eq:ks1}
    (1+r+t)(1+|t-r|)^{\frac 12} |f(t,r,\theta, \varphi)|\leq C R^{2}\sum_{I \in \boldsymbol{K}^{\leq 2}_R}\Vert \der^I f \Vert_{L^2(\Sigma_t)}.
\end{equation}
\end{lemma}

\begin{proof}[Proof of Lemma~\ref{lem:classicalks}]
The result without $R$-weights is classical, a proof can be found for example in~\cite{jonathannotes} and also in \cite{sogge}:
\begin{equation*}
    (1+r+t)(1+|t-r|)^{\frac 12} |f(t,r,\theta, \varphi)|\leq C \sum_{I \in \boldsymbol{K}^{\leq 2}}\Vert \der^I f \Vert_{\Sigma_t}.
\end{equation*}
Recall now that every element in $\boldsymbol{K}$ can be written as $R$ multiplying an element in $\boldsymbol{K}_R$. This concludes the proof, since $R \geq 1$.
\end{proof}

\section{Main Estimates}\label{sec:linearest}

\subsection{Initial data decomposition}\label{sub:constructinit}
In this section, we decompose the initial data in Theorem~\ref{thm:main} in a suitable manner, by introducing cutoff functions such that the diameter of the support of every individual cutoff is comparable to the parameter $R$.

Let us first suppose, without loss of generality, that the configuration of points $\Pi$ introduced in the statement of Theorem~\ref{thm:main} is such that the following condition holds:
\begin{equation}\label{eq:sepcond}
\min_{\substack{i,j \in \{1, \ldots, N\}\\i\neq j}} |p_i - p_j| \geq 2.
\end{equation}
Let us start from a smooth cutoff function $\chi_0: \R \to \R$, such that $|\chi_0(x)|\leq 1$, and such that
\begin{equation*}
    \chi_0(x) := \left\{ 
        \begin{array}{c}
            1 \quad \text{for}\quad |x|\leq  1 / 4, \\
            0 \quad \text{for}\quad |x| \geq 1 / 2.
        \end{array}
    \right.
\end{equation*}
We then recall the initial data as defined in the statement of Theorem~\ref{thm:main}, and we set, recalling that $w_i = R \cdot p_i$, and that here we consider $x \in \R^3$,
\begin{equation}\label{eq:defcutoffbumps}
    \tilde \phi^{(0)}_{i} := \chi_0\left(\frac{x-w_i}{R}\right) \phi^{(0)}_{i}(x-w_i), \qquad \tilde \phi^{(1)}_{i} := \chi_0\left(\frac{x-w_i}{R}\right) \phi^{(1)}_{i}(x-w_i).
\end{equation}
Recall that here $i \in \{1, \ldots, N\}$. Condition~\eqref{eq:sepcond} now ensures that, for $i \neq j$, $i,j \in \{1, \ldots, N\}$, the support of $\tilde \phi^{(0)}_{i}$ is disjoint from the support of $\tilde \phi^{(0)}_{j}$, and similarly the support of $\tilde \phi^{(1)}_{i}$ is disjoint from that of $\tilde \phi^{(1)}_{j}$.

This decomposition corresponds to localizing each of the pieces of data around the point $w_i$, where it is centered. We still have to consider the remainder, which we define as
\begin{equation}\label{eq:defcutoffrest}
    \tilde \phi^{(0)}_{0} := \phi^{(0)} - \sum_{i=1}^N \tilde \phi^{(0)}_{i}, \qquad \tilde \phi^{(1)}_{0} := \phi^{(1)} - \sum_{i=1}^N \tilde \phi^{(1)}_{i}.
\end{equation}

We now note that the following lemma holds true.

\begin{lemma}\label{lem:decomposition}
    There exists a universal constant $C>0$, and a constant $C(N, d_\Pi)$ depending on $N$ and $d_\Pi$ (which has been defined in equation~\eqref{eq:dpidef}), such that the following holds. Let $\tilde \phi^{(0)}_{i}$, $\tilde \phi^{(1)}_{i}$, $i \in \{0, \ldots, N\}$ be constructed from data coming from the statement of Theorem~\ref{thm:main}, according to formulas~\eqref{eq:defcutoffbumps}~and~\eqref{eq:defcutoffrest}. In particular, we are assuming that each of the $\phi^{(0)}_{i}$ and $\phi^{(1)}_{i}$ satisfies bounds~\eqref{eq:boundsbumps}. Then, the following bounds hold for $i \in \{0,\ldots,N\}$ and for $k$ non-negative integer, $k \leq n+7$:
	\begin{align}\label{eq:boundsit}
	    &|\widehat{\partial}^{k} \tilde{\phi}^{(0)}_{i}|\leq \frac{C \varepsilon} {(1+r_i)^{k+2}} \quad &\text{for } k \geq 0, \qquad &|\widehat{\partial}^{k-1} \tilde \phi^{(1)}_{i}|\leq \frac{ C \varepsilon}{ (1+r_i)^{k+2}}, \quad &\text{for }k \geq 1,\\
	    &|\widehat{\partial}^{k} \tilde{\phi}^{(0)}_{0}|\leq \frac{C(N, d_\Pi) \varepsilon}{ (1+r)^{k+2}} \quad &\text{for } k \geq 0, \qquad &|\widehat{\partial}^{k-1} \tilde \phi^{(1)}_{0}|\leq \frac{C(N, d_\Pi)\varepsilon}{ (1+r)^{k+2}}, \quad &\text{for }k \geq 1,\label{eq:boundsphizero}\\
	    &|\widehat{\partial}^{k} \tilde{\phi}^{(0)}_{0}|\leq \frac{C(N, d_\Pi) \varepsilon} {R^{k+2}} \quad &\text{for } k \geq 0, \qquad &|\widehat{\partial}^{k-1} \tilde \phi^{(1)}_{0}|\leq \frac{C(N, d_\Pi) \varepsilon} {R^{k+2}}, \quad &\text{for }k \geq 1. \label{eq:boundzeroloc}
	\end{align}
	Here, for $i \geq 1$, the coordinate $r_i$ is defined as the radial distance to the point where the $i$-th piece of data is localized, i.~e.~the point $w_i = R \cdot p_i$:
	$$
	r_i = |w_i -  x|, \quad \text{for} \quad  x \in \R^3.
	$$
	In particular, upon possibly restricting $\varepsilon_0$ to a smaller value depending on $N$ and $d_\Pi$, all the following initial value problems admit a global-in-time solution:
	\begin{equation}\label{eq:bumps2}
    \begin{aligned}
    &\Box \phi_{i} + F (d \phi_i, d^2 \phi_i) =  G (d \phi_i, d \phi_i),\\
    & \phi_i|_{t=0} = \tilde \phi^{(0)}_{i},\\
    & \p_t \phi_i|_{t=0} = \tilde \phi^{(1)}_{i},
    \end{aligned}
\end{equation}
for $i \in \{0, \ldots, N\}$.

Furthermore, every such solution $\phi_i$ for $i \in \{1, \ldots, N\}$ satisfies the conclusions from Lemma~\ref{prop:decphii}: we have the following bounds, for $i \in \{1, \ldots, N\}$, and for all multi-indices $I \in I^{\leq n}_{ \Ga^{(i)}}$:
\begin{equation}\label{eq:asymptotici}
|\bar \p^{(i)} \der^{I} \phi_i| \leq C(N, d_\Pi) \varepsilon \frac{1}{(1+r_i^2)(1+|u_i|)^\delta}, \quad |\p \der^{I} \phi_i| \leq C(N, d_\Pi) \varepsilon \frac{1}{(1+v_i)(1+|u_i|)^{1+\delta}}.
\end{equation}
In addition, we have the following uniform $L^2$ estimates, valid for all $i \in \{1, \ldots, N\}$ and all $I \in I^{\leq n+6}_{ \Ga^{(i)}}$:
\begin{equation}\label{eq:decl2phii}
\Vert \p \der^I \phi_i \Vert_{L^2(\Sigma_t)} \leq C(N, d_\Pi) \varepsilon.
\end{equation}
Finally, we have the improved bounds for $\phi_0$, valid for all $K_1 \in I^{\leq n+6}_{\boldsymbol{K}^{(i)}}$, and all $K_2 \in I^{\leq n}_{\boldsymbol{K}^{(i)}}$:
\begin{align}
&\Vert \p \der^{K_1} \phi_0 \Vert_{L^2(\Sigma_t)} \leq C(N,d_\Pi) \eps R^{-\frac 32},\label{eq:l2impphi0} \\
&|\bar \p^{(i)} \der^{K_2} \phi_0| \leq C(N,d_\Pi) \varepsilon \frac{R^{- \frac 32}}{(1+r_i^2)(1+|u_i|)^\delta},\\
& |\p \der^{K_2} \phi_0| \leq C(N,d_\Pi) \varepsilon \frac{R^{-\frac 32}}{(1+v_i)(1+|u_i|)^{1+\delta}}.\label{eq:linfimpphi0}
\end{align}
\end{lemma}

\begin{proof}[Proof of Lemma~\ref{lem:decomposition}]
Let us first focus on the case $i \geq 1$. We have that, for a multi-index $I \in I^{k}_{\boldsymbol{T}_s}$, with $k \leq n+7$,
\begin{equation*}
\begin{aligned}
    &|\der^I \tilde \phi^{(0)}_{i}| \leq \sum_{J+K \subset I} \left|\der^J \chi_0\left(\frac{x-w_i}{R}\right)\right| \cdot  \left|\der^K \phi_{i}^{(0)}\left(x-w_i\right)\right|\\
    & \qquad \leq C \varepsilon \frac 1 {R^{|J|}} \frac 1 {r_i^{|K|+2}} \chi^{(|J|)}_0\left(\frac{x-w_i}{R}\right) \leq C \frac{\varepsilon}{(1+r_i)^{k+2}}.
\end{aligned}
\end{equation*}
This proves the claim~\eqref{eq:boundsit} for $\phi^{(0)}_{i}$, with $i\in \{1, \ldots, N\}$. The claim for $\phi^{(1)}_{i}$, $i \geq 1$ is analogous. The global existence statement and decay for $\phi_i$ with $i \neq 0$ (inequalities~\eqref{eq:asymptotici} and~\eqref{eq:decl2phii}) then follow readily from Lemma~\ref{prop:decphii}. The bounds for $\phi_i^{(1)}$ with $i \in \{ 1, \dots, N \}$ follow analogously.

We now need to show the improved estimates for $\phi_0$. We have the expression
\begin{equation*}
    \tilde\phi^{(0)}_{0} = \sum_{i=1}^N \Big(1- \chi_0\Big(\frac{x-w_i}{R}\Big) \Big) \phi^{(0)}_{i}.
\end{equation*}
Let us now focus only on one term in the sum, again with $I \in I^k_{\boldsymbol{T}_s}$:
\begin{equation*}
\Big| \der^I\Big(\Big(1- \chi_0\Big(\frac{x-w_i}{R}\Big) \Big) \phi^{(0)}_{i}\Big)\Big|.
\end{equation*}
Let us now set $L = 4|w_i| = 4R|p_i|$. If we restrict to the region $|x| \leq L$, we have the inequality:
\begin{equation*}
\begin{aligned}
    &\Big| \der^I \Big(\Big(1- \chi_0\Big(\frac{x-w_i}{R}\Big) \Big) \phi_{i}^{(0)}\Big)\Big|\\
    &\leq \sum_{J + K \subset I} \Big| \der^J \Big(1- \chi_0\Big(\frac{x-w_i}{R}\Big) \Big) \der^K\phi_{i}^{(0)}| \leq C(d_\Pi) \frac \varepsilon {R^{k+2}} \leq C(d_\Pi) \frac \varepsilon {(1+r)^{k+2}},
\end{aligned}
\end{equation*}
since in this region $r$ is at most $4R|p_i|$.

If instead we restrict to the region where $ r\geq 4 R|p_i|$, we note that $\chi_0\Big(\frac{x-w_i}{R}\Big)$ is identically $0$ in this region. Moreover, we have $r_i \geq |r - R|p_i|| \geq 3 R |p_i|$ Therefore, we have that $|r| \geq |r_i - R |p_i|| = r_i - R|p_i|\geq \frac 2 3 r_i$, and the claim follows readily:
\begin{equation*}
\Big| \der^I \Big(\Big(1- \chi_0\Big(\frac{x-w_i}{R}\Big) \Big) \phi_{i}^{(0)}\Big)\Big| =  \Big| \Big(1- \chi_0\Big(\frac{x-w_i}{R}\Big) \Big)\der^I\phi_{i}^{(0)}| \leq C(d_\Pi) \frac \varepsilon {(1+r_i)^{k+2}} \leq C(d_\Pi) \frac \varepsilon {(1+r)^{k+2}}.
\end{equation*}
Finally, summing all the contributions from the different $\phi_{i}^{(0)}$'s we conclude the proof of inequality~\eqref{eq:boundsphizero}. The bounds for $\tilde\phi^{(1)}_0$ follow analogously.

Note now that, upon following the same reasoning as above choosing $w_i$ as the center of our coordinate system, we have the bounds: 
\begin{equation*}
 |\widehat{\partial}^{k} \tilde{\phi}^{(0)}_{0}|\leq C(N, d_\Pi) \varepsilon (1+r_i)^{-k-2}, \qquad |\widehat{\partial}^{k-1} \tilde \phi^{(1)}_{0}|\leq C(N, d_\Pi) \varepsilon (1+r_i)^{-k-2},
\end{equation*}
valid for all $k \leq n+7$. We also note that $\Big(1- \chi_0\Big(\frac{x-w_i}{R}\Big) \Big) \phi_{i}^{(0)}$ and $\Big(1- \chi_0\Big(\frac{x-w_i}{R}\Big) \Big) \phi_{i}^{(1)}$ are supported outside of a ball of radius $R/4$ centered at $w_i$, which implies claim~\eqref{eq:boundzeroloc}.

In particular, this implies that for all $i \in \{1, \ldots, N\}$ and for all $K_1 \in I^{\leq n+2}_{\boldsymbol{K}^{(i)}}$, we have
$$
\Vert \p \der^{K_1} \phi_0 \Vert_{L^2(\Sigma_0)}\leq C(N,d_\Pi) \varepsilon R^{-\frac 32}.
$$
Note that some care is needed to derive this estimate as weighted vector fields adapted to the $i$-th piece of data may hit the $j$-th piece of data.

By Lemma~\ref{prop:decphii}, for every $i \in \{0, \ldots, N\}$, we then have the improved bounds for $\phi_0$, valid for all $K_1 \in I^{\leq n+6}_{\boldsymbol{K}^{(i)}}$, $K_2 \in I^{\leq n}_{\boldsymbol{K}^{(i)}}$:
\begin{equation}
\begin{aligned}
&\Vert \p \der^{K_1} \phi_0 \Vert_{L^2(\Sigma_t)} \leq C(N,d_\Pi) \eps R^{-\frac 32}, \\
&|\bar \p^{(i)} \der^{K_2} \phi_0| \leq C(N,d_\Pi) \varepsilon \frac{R^{- \frac 32}}{(1+r_i^2)(1+|u_i|)^\delta}, \quad |\p \der^{K_2} \phi_0| \leq C(N,d_\Pi) \varepsilon \frac{R^{-\frac 32}}{(1+v_i)(1+|u_i|)^{1+\delta}}.
\end{aligned}
\end{equation}
\end{proof}

\subsection{Derivation of the equation for the first iterate}\label{sub:firstiterate}

In this section, we derive the system satisfied by the difference
\begin{equation}\label{eq:psidefinition}
    \psi: = \phi - \sum_{i = 0}^N \phi_i,
\end{equation}
where each of the $\phi_i$ is the global solution to the following initial value problem:
\begin{equation}\label{eq:onebump}
\begin{aligned}
    &\Box \phi_i + F (d \phi_i, d^2 \phi_i) =  G (d \phi_i, d \phi_i), \\
    & \phi_i|_{t=0} = \tilde \phi^{(0)}_{i}, \\
    & \p_t \phi_i|_{t=0} = \tilde \phi^{(1)}_{i}.
\end{aligned}
\end{equation}
\begin{remark}
Note that the initial data for these $N+1$ auxiliary problems is $(\tilde \phi^{(0)}_{i}, \tilde \phi^{(1)}_{i})$ and it was constructed in Section~\ref{sub:constructinit}. It is not to be confused with the data in the statement of Theorem~\ref{thm:main}.
\end{remark}
We have the following lemma:
\begin{lemma}\label{lem:firstit}
    Let $N$ be a non-negative integer. Let furthermore $(\tilde \phi^{(0)}_{i}, \tilde \phi^{(1)}_{i})$  be a collection of smooth functions in $\R^3$, for $i \in \{0, \ldots, N\}$, as constructed in formulas~\eqref{eq:defcutoffbumps} and~\eqref{eq:defcutoffrest}. Recall that, by construction,
    $$
    \phi^{(0)}:= \sum_{i=0}^N \tilde \phi^{(0)}_{i}, \qquad \phi^{(1)} := \sum_{i=0}^N  \tilde \phi^{(1)}_{i}.
    $$
    Suppose that $\phi$ is the smooth solution to the following initial value problem:
	\begin{equation}\label{eq:bumpsfit}
    \begin{aligned}
    &\Box \phi + F (d \phi, d^2 \phi) = G (d \phi, d \phi),\\
    & \phi|_{t=0} =  \phi^{(0)},\\
    & \p_t \phi|_{t=0} =  \phi^{(1)}.
    \end{aligned}
\end{equation}
Suppose that $\phi_i$, $i \in \{0,\ldots,N\}$ is the smooth solution to the initial value problem~\eqref{eq:onebump}. Then, letting
\begin{equation}
    \psi := \phi - \sum_{i=0}^N \phi_i,
\end{equation}
we have that $\psi$ satisfies the following system:
\begin{equation}\label{eq:fordiff}
   \begin{aligned}
    &\Box \psi +F(d \psi, d^2 \psi)+ \sum_{\substack{i,j = 0, \ldots, N\\i \neq j}} F (d \phi_i, d^2 \phi_j) +\sum_{i=0}^N ( F(d \phi_i, d^2 \psi) + F(d \psi, d^2 \phi_i)) \\
    & \qquad = G(d \psi, d \psi) +\sum_{\substack{i,j = 0, \ldots, N\\i \neq j}} G (d \phi_i, d \phi_j) + \sum_{i=0}^N ( G(d \phi_i, d \psi) + G(d \psi, d \phi_i)),\\
    & \psi|_{t=0} = 0,\\
    & \p_t \psi|_{t=0} = 0.
\end{aligned}
\end{equation}
\end{lemma}

\begin{proof}[Proof of Lemma~\ref{lem:firstit}]
    The proof follows from a straightforward calculation. For simplicity, let's assume $F = 0$ (the proof in the other case being totally analogous). We start by calculating:
    \begin{equation*}
    \begin{aligned}
        &G(d \psi, d \psi) = G(d \phi -  \sum_{i=0}^N d \phi_i, d \phi -  \sum_{i=0}^N  d\phi_i) =  \\
        &\quad G(d \phi, d\phi) - \sum_{i=0}^N G(d \psi, d \phi_i) - \sum_{i=0}^N G(d \phi_i, d \psi) - \sum_{\substack{ i,j =0,\ldots,N\\ i\neq j}} G(d \phi_i, d \phi_j) - \sum_{i=0}^N G(d \phi_i, d \phi_i)\\
        & \quad = \Box \left (\phi - \sum_{i=0}^N \phi_i \right ) -  \sum_{i=0}^N G(d \psi, d \phi_i) - \sum_{i=0}^N G(d \phi_i, d \psi) - \sum_{\substack{ i,j =1,\ldots,N\\ i\neq j}} G(d \phi_i, d \phi_j).
    \end{aligned}
    \end{equation*}
From the fact that $\psi = \phi - \sum_{i=0}^N \phi_i$ the claim then follows readily. It is also evident that $\psi|_{t = 0}= 0$, and that $\p_t \psi|_{t = 0}= 0$.
\end{proof}

\subsection{Derivation of the equation for the second iterate}\label{sub:seconditerate}

In this section, we derive the system satisfied by the difference
\begin{equation}\label{eq:Psidefinition}
    \Psi : = \psi - \sum_{\substack{i,j = 0\\i\neq j}}^N \psi_{ij},
\end{equation}
where each of the $\psi_{ij}$'s is the global solution to the following initial value problem:
\begin{equation}\label{eq:pairs}
\begin{aligned}
    &\Box \psi_{ij} +  F(d \phi_i, d^2 \phi_j) + F(d \phi_j,d^2 \phi_i) = 2 G (d \phi_i, d \phi_j),\\
    & \psi_{ij}|_{t=0} = 0, \\
    & \p_t \psi_{ij}|_{t=0} = 0,
\end{aligned}
\end{equation}
valid for all $i,j \in \{0, \ldots, N\}$. 

\begin{remark}
    Note that the equation satisfied by $\psi_{ij}$ is formed taking equation~\eqref{eq:fordiff} and considering only the inhomogeneous contributions from $\phi_i$ and $\phi_j$.
\end{remark}

We have the following lemma.

\begin{lemma}\label{lem:secondit}
Let $N$ be a positive integer and let $\psi$, $\phi_i$, $i \in \{0,\ldots, N\}$ be as in the statement of Lemma~\ref{lem:firstit}. Define furthermore $\psi_{ij}$ as in equation~\eqref{eq:pairs}, and let $\Psi$ be as in equation~\eqref{eq:Psidefinition}. Under these conditions, $\Psi$ satisfies the following initial value problem:
\begin{align}
    &\Box \Psi +F(d \Psi, d^2 \Psi) \nonumber\\
    &+ \sum_{\substack{i,j = 0,\ldots, N\\i\neq j}}(F(d \psi_{ij}, d^2 \Psi)+F(d \Psi, d^2 \psi_{ij})) + \sum_{\substack{g,h,i,j = 0,\ldots, N\\g \neq h, i\neq j}}F(d \psi_{gh}, d^2 \psi_{ij})\nonumber\\  
    & +\sum_{i=0}^N\sum_{\substack{g,h = 0,\ldots, N\\g\neq h}} ( F(d \phi_i, d^2 \psi_{gh}) + F(d \psi_{gh}, d^2 \phi_i))\nonumber\\  
    & +\sum_{i=0}^N ( F(d \phi_i, d^2 \Psi) + F(d \Psi, d^2 \phi_i)) \nonumber\\
     &\qquad = \sum_{\substack{i,j = 0,\ldots, N\\i\neq j}}(G(d \psi_{ij}, d \Psi)+G(d \Psi, d \psi_{ij})) + \sum_{\substack{g,h,i,j = 0,\ldots, N\\g\neq h, i\neq j}}G(d \psi_{gh}, d \psi_{ij})\label{eq:secondit}\\  
    & \qquad +\sum_{i=0}^N\sum_{\substack{g,h = 0,\ldots, N\\g\neq h}} ( G(d \phi_i, d \psi_{gh}) + G(d \psi_{gh}, d \phi_i))\nonumber\\  
    & \qquad +\sum_{i=0}^N ( G(d \phi_i, d \Psi) + G(d \Psi, d \phi_i)),\nonumber\\
    & \Psi|_{t=0} = 0,\nonumber\\
    & \p_t\Psi|_{t=0} = 0.\nonumber
\end{align}
\end{lemma}

\begin{proof}[Proof of Lemma~\ref{lem:secondit}]
We start from equation~\eqref{eq:fordiff}:
\begin{equation}
   \begin{aligned}
    &\Box \psi +F(d \psi, d^2 \psi)+ \sum_{\substack{i,j = 0, \ldots, N\\i \neq j}} F (d \phi_i, d^2 \phi_j) \\  
    &\qquad +\sum_{i=0}^N ( F(d \phi_i, d^2 \psi) + F(d \psi, d^2 \phi_i))  \\
    & \qquad = G(d \psi, d \psi) +\sum_{\substack{i,j = 0, \ldots, N\\i \neq j}} G (d \phi_i, d \phi_j) \\ &\qquad + \sum_{i=0}^N ( G(d \phi_i, d \psi) + G(d \psi, d \phi_i)).
\end{aligned}
\end{equation}
This implies:
\begin{equation}
   \begin{aligned}
    &\Box \Big( \psi - \sum_{\substack{i,j = 0,\ldots, N\\i\neq j}} \psi_{ij} \Big)+F(d \psi, d^2 \psi) \\
    &\qquad +\sum_{i=0}^N ( F(d \phi_i, d^2 \psi) + F(d \psi, d^2 \phi_i)) \\
    & \qquad = G(d \psi, d \psi) \\ &\qquad + \sum_{i=0}^N ( G(d \phi_i, d \psi) + G(d \psi, d \phi_i)).
\end{aligned}
\end{equation}
The conclusion follows readily using equation~\eqref{eq:Psidefinition}.
\end{proof}

\subsection{The improved energy estimates and the trilinear estimates}\label{sub:improvedenergy}

We shall now prove the trilinear estimates described in Section~\ref{subsub:quadraticimp}. More precisely, we shall obtain improved energy estimates on solutions to the first iterates $\psi_{ij}$, which require us to prove these trilinear estimates. Indeed, we must control a trilinear spacetime integral in order to control the $\partial_t$ energy of the functions $\psi_{i j}$ in \eqref{eq:pairs}. These trilinear estimates will control the bilinear interaction between the solutions $\phi_i$ and $\phi_j$. The trilinear estimates which result are explicitly stated in Proposition~\ref{prop:trilinear}.

The functions $\psi_{ij}$ satisfy linear equations with fixed inhomogeneities, so we already know that solutions exist globally. The improvements introduced by these estimates will be strong enough to prove global existence for the nonlinear equation arising from the second iterate.

In proving our estimates, we shall consider two different cases. The first case involves $\psi_{i j}$ with $i \ne 0$ and $j \ne 0$ (this will be the content of Proposition~\ref{prop:energy}). The second case involves $\psi_{ij}$ with either $i = 0$ or $j = 0$ (this will be the content of Lemma~\ref{lem:phi0j}). We have the following lemma:
\begin{proposition}\label{prop:energy}
Let $i \ne 0$ and $j \ne 0$, and let $\psi_{ij}$ be a solution to the initial value problem~\eqref{eq:pairs}. For simplicity, let us suppose that the initial data for $\phi_i$ is centered at the point $w_i = (-R, 0,0)$, and the initial data for $\phi_j$ is centered at the point $w_j = (R, 0,0)$ (we are assuming, without loss of generality, that $|p_i - p_j| = 2$). Then, we have that, for all multi-indices $I \in I^{\leq n-1}_{\boldsymbol{\Gamma}^{(h)} }$, $h \in \{1,2\}$, the following inequality holds true:
\begin{equation}\label{eq:enpsiij}
    \begin{aligned}
    \sup_{t\geq 0} \Vert \partial \der^I \psi_{i j} \Vert_{L^2 (\Sigma_t)} \le C {\eps^2 \over R},
    \end{aligned}
\end{equation}
where the constant $C$ can depend on the distance between $p_i$ and $p_j$, which we recall are the points where the initial data for $\phi_i$ and $\phi_j$, respectively, are centered when $R = 1$.
\end{proposition}

We note that the following trilinear estimates are established as a result of the proof of Proposition~\ref{prop:energy}.
\begin{proposition}\label{prop:trilinear}
    Let $\eta_h$, with $h \in \{1,2\}$, be solutions to the following initial value problem:
    \begin{equation}\label{eq:etatrilinear}
    \begin{aligned}
	    &\Box \eta_h = 0,\\
	    & \big(\eta_h|_{t=0},\ \p_t \eta_h|_{t=0}\big)  = (\eta_h^{(0)},\eta_h^{(1)}).
    \end{aligned}
    \end{equation}
    Let us furthermore suppose that the support of the initial data for $\eta_1$ is localized around the point $(-R,0,0)$ and that the support of initial data for $\eta_2$ is localized around the point $(R,0,0)$:
    \begin{equation}
        \text{supp }\big(\eta_1^{(0)}\big), \ \text{supp }\big(\eta_1^{(1)}\big) \subset B(Rp_1, 1), \qquad   \text{supp }\big(\eta_2^{(0)}\big), \ \text{supp }\big(\eta_2^{(1)}\big) \subset B(Rp_2, 1),
    \end{equation}
    Here, $p_1 = (-1,0,0)$, $p_2 = (1,0,0)$, and $B(x_1, a)$ is the Euclidean three-dimensional ball of radius $a$ centered at the point $x_1$.
    
    Moreover, let $\tilde F$ and $\tilde G$ be resp.~a trilinear null form and a bilinear null form, according to Definition~\ref{def:null}. For $f(t,x)$ an arbitrary smooth compactly supported function in spacetime, we have that the following estimates hold true:
    \begin{equation}\label{eq:trilinearestimates}
        \begin{aligned}
            &\int_0^\infty \int_{\Sigma_s} \Big( \big| \tilde F (d \eta_1,d^2 \eta_2)\big| \,\big| f(s,x)\big|+ \big| \tilde G (d \eta_1,d \eta_2)\big| \,\big| f(s,x)\big|  \Big)\de x \de s \\
            &\quad \leq C \frac{\big(\Vert \eta^{(0)}_1\Vert_{H^{8}(\Sigma_0)}+\Vert \eta^{(1)}_1 \Vert_{H^{7}(\Sigma_0)}\big)\big(\Vert \eta^{(0)}_2\Vert_{H^{8}(\Sigma_0)}+\Vert \eta^{(1)}_2 \Vert_{H^{7}(\Sigma_0)}\big)}{R} \sup_{\bar u_1 \in\R} \Vert f \Vert_{L^2 (C^{(1)}_{\bar u_1})}.
        \end{aligned}
    \end{equation}
    Furthermore, the same inequality holds replacing $\sup_{\bar u_1 \in\R} \Vert f \Vert_{L^2 (C^{(1)}_{\bar u_1})}$ with $\sup_{\bar u_2 \in\R} \Vert f \Vert_{L^2 (C^{(2)}_{\bar u_2})}$.
    
    Recall that, here, the notation $C^{(h)}_{\bar u_h}$ for $h \in \{1,2\}$ denotes the outgoing null cone $\{u_h = \bar u_h\} \cap \{t \geq 0\}$, where the coordinate $u_h$ has been introduced in Definition~\ref{def:riviui}.
\end{proposition}

It is possible to deduce Proposition~\ref{prop:trilinear} from the the proof of Proposition~\ref{prop:energy}, noting that the null forms $\tilde F$ and $\tilde G$ correspond to the inhomogeneities in the linear equation for $\psi_{ij}$, while the function $f(t,x)$ corresponds to the multiplier we are using in the proof of Proposition~\ref{prop:energy}.

The trilinear estimates~\eqref{eq:trilinearestimates} are useful because they gain a power of $R^{-1}$. However, we note that they are very wasteful in terms of derivatives required on the functions $\eta_1$ and $\eta_2$.

\begin{remark}
    In Proposition~\ref{prop:trilinear}, we require control on $8$ derivatives of $\eta^{(0)}_1$ and  $\eta^{(0)}_2$ as we need to estimate $2$ derivatives in $L^\infty$ of both $\eta_1$ and $\eta_2$, and we require the improved decay of Lemma~\ref{prop:decphii}.
\end{remark}

We now turn to the proof of Proposition~\ref{prop:energy}.

\bp[Proof of Proposition~\ref{prop:energy}]
Let us first commute equation~\eqref{eq:pairs} with $I \in I^{\leq n-1}_{\boldsymbol{\Gamma}^{(h)} }$. We obtain:
\begin{equation}\label{eq:pairscomm}
\begin{aligned}
    &\Box \der^I \psi_{ij} +  \sum_{H+K \subset I} \big(F_{HK}(d \der^H \phi_i, d^2 \der^K \phi_j) + F_{HK}(d \der^H \phi_j,d^2 \der^K \phi_i) \big)\\
    &\quad =  \sum_{H+K \subset I} G_{HK} (d \der^H \phi_i, d \der^K \phi_j)
\end{aligned}
\end{equation}
Here, every $F_{HK}$, $G_{HK}$ is a trilinear (resp.~bilinear) null form as in Definition~\ref{def:null}.
We now multiply the evolution equation in~\eqref{eq:pairscomm} by $\partial_t \psi_{i j}$ and integrate by parts in a spacetime region bounded by two slabs $\Sigma_0$ and $\Sigma_t$ with $t \ge R/10$. This gives us that
\begin{equation}\label{eq:enijspa}
    \begin{aligned}
    &\frac 12 \int_{\Sigma_t} |\partial \der^I \psi_{i j}|^2 \de x = \int_{\frac R {10}}^t \int_{\Sigma_s} \Big( \sum_{H+K \subset I} \big(F_{HK}(d \der^H \phi_i, d^2 \der^K \phi_j) + F_{HK}(d \der^H \phi_j,d^2 \der^K \phi_i) 
    \\
    &\quad - \sum_{H+K \subset I} G_{HK} (d \der^H \phi_i, d \der^K \phi_j) \Big) \partial_t \der^I \psi_{i j} \de x \de s =: \mathfrak{B},
    \end{aligned}
\end{equation}
where we have used the fact that, by domain of dependence, any product involving $\phi_i$ and $\phi_j$ will vanish for $t \le \frac R {10}$ (recall that $i \neq 0$ and $j \neq 0$). Indeed, along with the fact that $\psi_{i j}$ has vanishing initial data, this implies that $\int_{\Sigma_{R/10}} |\partial \der^I \psi_{i j}|^2 \de x = 0$. Now, we decompose the spacetime integration region into two regions. The first region is where $|u_i| \le {R \over 10}$ and $|u_j| \le {R \over 10}$, and the second piece is where at least one of $|u_i|$ or $|u_j|$ is larger than ${R \over 10}$. With $\mathcal{I}_{i j}$ the set of all points where $|u_i| \le {R \over 10}$ and $|u_j| \le {R \over 10}$, we have that $\mathfrak{B} = \mathfrak{B}_1 + \mathfrak{B}_2$, where
\begin{equation}
    \begin{aligned}
    &\mathfrak{B}_1 = \int_{\{ R/10\le s \le t \} \cap \mathcal{I}_{i j}} \Big( \sum_{H+K \subset I} \big(F_{HK}(d \der^H \phi_i, d^2 \der^K \phi_j) + F_{HK}(d \der^H \phi_j,d^2 \der^K \phi_i) \big)
    \\
    &\qquad - \sum_{H+K \subset I} G_{HK} (d \der^H \phi_i, d \der^K \phi_j) \Big)  \partial_t \der^I \psi_{i j} \de x \de s,
    \end{aligned}
\end{equation}
and
\begin{equation}
    \begin{aligned}
    &\mathfrak{B}_2 = \int_{\{ R/10\le s \le t \} \cap \mathcal{I}^c_{i j}} \Big( \sum_{H+K \subset I} \big(F_{HK}(d \der^H \phi_i, d^2 \der^K \phi_j) + F_{HK}(d \der^H \phi_j,d^2 \der^K \phi_i) \big)
    \\
    &\qquad - \sum_{H+K \subset I} G_{HK} (d \der^H \phi_i, d \der^K \phi_j) \Big)  \partial_t \der^I \psi_{i j} \de x \de s.
    \end{aligned}
\end{equation}
We can further decompose $\mathcal{I}_{i j}^c$ (the set of points ``far away'' from at least one of the light cones) into the set where $|u_i| \ge {R \over 10}$, which we shall call $\mathcal{E}_i^*$, and the remainder, $\mathcal{E}_i' = \mathcal{I}^c_{i j} \setminus \mathcal{E}_i^*$. We note that we must have that $|u_j| \ge {R \over 10}$ in $\mathcal{E}_i'$. Now, for $\mathfrak{B}_2$, we have that
\begin{equation}
    \begin{aligned}
    &\mathfrak{B}_2 = \int_{\{ R/10 \le s \le t \} \cap \mathcal{E}_i^*} \Big( \sum_{H+K \subset I} \big(F_{HK}(d \der^H \phi_i, d^2 \der^K \phi_j) + F_{HK}(d \der^H \phi_j,d^2 \der^K \phi_i) \big)
    \\
    &\qquad - \sum_{H+K \subset I} G_{HK} (d \der^H \phi_i, d \der^K \phi_j) \Big)  \partial_t \der^I \psi_{i j} \de x \de s \\ 
    &+ \int_{\{ R/10 \le s \le t \} \cap \mathcal{E}_i'} \Big( \sum_{H+K \subset I} \big(F_{HK}(d \der^H \phi_i, d^2 \der^K \phi_j) + F_{HK}(d \der^H \phi_j,d^2 \der^K \phi_i) \big) 
    \\
    &\qquad - \sum_{H+K \subset I} G_{HK} (d \der^H \phi_i, d \der^K \phi_j) \Big)  \partial_t \der^I \psi_{i j} \de x \de s  =: \mathfrak{B}_{21} + \mathfrak{B}_{22}.
    \end{aligned}
\end{equation}
We shall now get estimates for these integrals over $\mathcal{E}_i^*$ and $\mathcal{E}_i'$, terms $\mathfrak{B}_{21}$ and $\mathfrak{B}_{22}$. Thus, in the following, it suffices to get estimates for the first integral over $\mathcal{E}_i^*$ where $|u_i| \ge {R \over 10}$, as the estimates for the integral over $\mathcal{E}_i'$ follow in the same way after replacing $i$ with $j$ in the following argument.

Now, we have that
\begin{align*}
    &\mathfrak{B}_{21}  \leq \sup_{R/10 \le s \le t} \Vert \partial_t \der^I \psi_{i j} \Vert_{L^2 (\Sigma_s \cap \mathcal{E}_i^*)} \\
    & \quad \times \int_{R/10}^t \Big( \sum_{H+K \subset I} \Big\Vert  \big(F_{HK}(d \der^H \phi_i, d^2 \der^K \phi_j) + F_{HK}(d \der^H \phi_j,d^2 \der^K \phi_i)\Big \Vert_{L^2 (\Sigma_s \cap \mathcal{E}_i^*)}
    \\
    &\qquad + \sum_{H+K \subset I} \Big\Vert G_{HK} (d \der^H \phi_i, d \der^K \phi_j)  \Big\Vert_{L^2 (\Sigma_s \cap \mathcal{E}_i^*)}\Big) \de s.
\end{align*}
We now divide the region $\mathcal{E}^*_i$ in two further subregions:
$$
\hat{\mathcal{E}}^*_i := \mathcal{E}^*_i \cap \big\{r_i \leq  \frac 1 {10} t\big\}, \qquad \check{\mathcal{E}}^*_i := \mathcal{E}^*_i \cap \big\{r_i \geq \frac 1 {10} t\big\}
$$
Let us first focus on the region $\hat{\mathcal{E}}^*_i$. We have that, in this region, $u_i \geq C t$, and hence, using the decay properties of $\phi$ which follow from display~\eqref{eq:asymptotici}, we have
\begin{equation}
    \begin{aligned}
    &\Vert G_{HK}(d \der^H \phi_i,d \der^K \phi_j) \Vert_{L^2 (\Sigma_s \cap \hat{\mathcal{E}}_i^*)}\\
    & \qquad \leq C \Vert | \p \der^H \phi_i| \, |\p \der^K \phi_j| \Vert_{L^2 (\Sigma_s \cap \hat{\mathcal{E}}_i^*)} \leq C  \Vert \p \der^H \phi_i\Vert_{L^\infty (\Sigma_s \cap \hat{\mathcal{E}}_i^*)} \, \Vert\p \der^K \phi_j\Vert_{L^2 (\Sigma_s \cap \hat{\mathcal{E}}_i^*)}\\
    & \qquad \leq C \frac {\eps^2} {s^{2+\delta}}.
    \end{aligned}
\end{equation}
Note that here we needed to bound only $n-1+1 = n$ derivatives in $L^\infty$. Furthermore, a similar reasoning holds for the terms in $F$ (bounding $n+1$ derivatives in $L^\infty$ this time) so that, overall, we obtain
\begin{equation}\label{eq:forehat}
\begin{aligned}
&\Vert F_{HK}(d \der^H \phi_i,d^2 \der^K \phi_j) \Vert_{L^2 (\Sigma_s \cap \hat{\mathcal{E}}_i^*)}+\Vert F_{HK}(d \der^H \phi_j,d^2 \der^K \phi_i) \Vert_{L^2 (\Sigma_s \cap \hat{\mathcal{E}}_i^*)}\\
&+\Vert G_{HK}(d \der^H \phi_i,d \der^K \phi_j) \Vert_{L^2 (\Sigma_s \cap \hat{\mathcal{E}}_i^*)} \leq C \frac {\eps^2} {s^{2+\delta}}.
\end{aligned}
\end{equation}
Let us now focus on the region $\check{\mathcal{E}}^*_i$. Using Lemma~\eqref{lem:nullstruct} on the structure of classical null forms, combined with the estimates in Lemma~\ref{lem:decomposition}, specifically bounds~\eqref{eq:asymptotici} and~\eqref{eq:decl2phii}, plus the bound~\eqref{eq:goodbad}, along with the fact that $|u_i| \ge {R \over 10}$ in $\mathcal{E}_i^*$ and the fact that $r_i \geq \frac 1 {10}t$ in the region $\check{\mathcal{E}}^*_i$, we have that
\begin{equation} \label{gooduest1}
    \begin{aligned}
    &\Vert G_{HK}(d \der^H \phi_i,d \der^K \phi_j) \Vert_{L^2 (\Sigma_s \cap \check{\mathcal{E}}_i^*)}\\
    & \qquad \leq C \Vert  |\overline \p^{(i)} \der^H \phi_i| \, |\p \der^K \phi_j| \Vert_{L^2 (\Sigma_s \cap \check{\mathcal{E}}_i^*)} +C \Vert  |\p \der^H \phi_i| \, |\overline \p^{(i)} \der^K \phi_j| \Vert_{L^2 (\Sigma_s \cap \check{\mathcal{E}}_i^*)}\\
    &\qquad  \le C \eps^2 {1 \over R^\delta s^2} + C  \Big\Vert  \frac R {r_i}|\p \der^H \phi_i| \, | \p \der^K \phi_j| \Big\Vert_{L^2 (\Sigma_s \cap \check{\mathcal{E}}_i^*)} + C \Vert  |\p \der^H \phi_i| \, |\overline \p^{(j)} \der^K \phi_j| \Vert_{L^2 (\Sigma_s \cap \check{\mathcal{E}}_i^*)}\\
    &\qquad  \le C \eps^2 {1 \over R^\delta s^2} + C \frac R {s} \Vert  |\p \der^H \phi_i| \, | \p \der^K \phi_j| \Vert_{L^2 (\Sigma_s \cap \check{\mathcal{E}}_i^*)} + C \eps^2 {1 \over s^2}\\
    &\qquad  \le C \eps^2 {1 \over R^\delta s^2} + C \eps^2 \frac R {s} \frac{1}{s R^{1+\delta}} + C \eps^2 {1 \over s^2} \leq C \eps^2 \frac 1 {s^2}.
    \end{aligned}
\end{equation}
\begin{remark}
Note that again we are bounding at most $n-1+1 = n$ derivatives in $L^\infty$.
\end{remark}

Using again estimates~\eqref{eq:asymptotici} and~\eqref{eq:decl2phii}, we can get the same estimates for the two terms involving $F$ in the region $\check{\mathcal{E}}^*_i$ (notice that this time we are estimating at most $n-1+2 = n+1$ derivatives in $L^\infty$). We obtain:
\begin{equation}\label{eq:forecheck}
\begin{aligned}
&\Vert F_{HK}(d \der^H \phi_i,d^2 \der^K \phi_j) \Vert_{L^2 (\Sigma_s \cap \check{\mathcal{E}}_i^*)}+\Vert F_{HK}(d \der^H \phi_j,d^2 \der^K \phi_i) \Vert_{L^2 (\Sigma_s \cap \check{\mathcal{E}}_i^*)}\\
&+\Vert G_{HK}(d \der^H \phi_i,d \der^K \phi_j) \Vert_{L^2 (\Sigma_s \cap \check{\mathcal{E}}_i^*)} \leq C \frac {\eps^2} {s^{2}}.
\end{aligned}
\end{equation}

The reasoning is totally analogous for the term $\mathfrak{B}_{22}$, restricting to the region $\mathcal{E}'_i$. Thus, altogether, combining estimates~\eqref{eq:forehat} and~\eqref{eq:forecheck} and the estimates for $\mathfrak{B}_{22}$, we have that
\begin{equation}\label{eq:forbtwo}
    \begin{aligned}
    &|\mathfrak{B}_2| \le C \eps^2 \sup_{R/10 \le s \le t } \Vert \partial_t \der^I \psi_{i j} \Vert_{L^2 (\Sigma_s \cap \mathcal{I}^c_{i j})} \int_{R/10}^t {1 \over s^2} \de s \\
    &\qquad \le C \eps^2 \sup_{R/10 \le s \le t} \Vert \partial_t \der^I \psi_{i j} \Vert_{L^2 (\Sigma_s \cap \mathcal{I}^c_{i j})} {1 \over R}.
    \end{aligned}
\end{equation}
We note that we have been wasteful in estimating this term. Indeed, being more precise, we could have controlled the term $\Vert |\partial V^H \phi_i| |\overline{\partial}^{(j)} V^K \phi_j| \Vert_{L^2 (\Sigma_s \cap \check{\mathcal{E}}_i^*)}$ in \eqref{gooduest1} by $C \eps^2 {1 \over R^{1 + \delta}}$. This would have improved the final estimate in \eqref{gooduest1} to $C \eps^2 {1 \over R^\delta s^2}$, and a similar argument would have improved the final estimate in \eqref{eq:forecheck} to $C \eps^2 {1 \over R^\delta s^2}$. This would then give us a final estimate of $C \eps^2 {1 \over R^{1 + \delta}}$ in \eqref{eq:forbtwo}. Another term, however, will prevent us from doing better than $C \eps^2 {1 \over R}$.

Now, for $\mathfrak{B}_1$, we begin by decomposing with respect to a null frame adapted to $\phi_i$. The same argument works using a null frame for $\phi_j$ instead. Now, using Lemma~\ref{lem:nullstruct} on the structure of null forms on $G_{HK}$ and $F_{HK}$, we have that
\begin{equation}\label{eq:mbfrac}
    \begin{aligned}
    &|\mathfrak{B}_1| \le C \int_{\{R/10 \le s \le t \} \cap \mathcal{I}_{i j}} \sum_{H+K \subset I} \Big(|\overline{\partial}^{(i)} \der^H \phi_i| |\partial^2 \der^K \phi_j| + |  \partial \der^H \phi_i| | \partial \overline{\partial}^{(i)} \der^K \phi_j| \\
    &\qquad \qquad \qquad +|\overline{\partial}^{(i)} \der^H \phi_j| |\partial^2  \der^K  \phi_i|+|\partial \der^H \phi_j| |\partial \overline{\partial}^{(i)} \der^K  \phi_i| \\
    &\qquad \qquad \qquad +|\overline{\partial}^{(i)} \der^H \phi_i| | \partial \der^K \phi_j|  + |\partial \der^H \phi_i| |\overline{\partial}^{(i)} \der^K \phi_j| \Big) |\partial_t \der^I \psi_{i j}| \de x \de s 
    \end{aligned}
\end{equation}
For the terms in which $\overline{\partial}^{(i)}$ falls on $\phi_i$ or $\partial \phi_i$ (i.~e.~the first, fourth, and fifth term in the previous display), we can use the pointwise estimate~\eqref{eq:decphii} to bound the $\overline{\partial}^{(i)}$ term in $L^\infty$, along with the energy estimate~\eqref{eq:decl2phii} to bound the term containing $\partial \phi_j$ in $L^2 (\Sigma_s)$. This gives us that
\begin{equation}\label{eq:rminusone1}
    \begin{aligned}
    &\int_{\{R/10 \le s \le t \} \cap \mathcal{I}_{i j}} \sum_{H+K \subset I} \Big(|\overline{\partial}^{(i)} \der^H \phi_i| |\partial^2 \der^K \phi_j| \\
    &\qquad \qquad \qquad +|\partial \der^H \phi_j| |\partial \overline{\partial}^{(i)} \der^K  \phi_i| +|\overline{\partial}^{(i)} \der^H \phi_i| | \partial \der^K \phi_j| \Big) |\partial_t \der^I \psi_{i j}| \de x \de s \\
    & \qquad \le C \eps^2 \sup_{R/10 \le s \le t} \Vert \partial_t \der^I \psi_{i j} \Vert_{L^2 (\Sigma_s \cap \mathcal{I}_{i j})}  \int_{R/10}^t {1 \over s^2} \de s \\
    &\qquad \le C \sup_{R/10 \le s \le t} \Vert \partial_t \psi_{i j} \Vert_{L^2 (\Sigma_s \cap \mathcal{I}_{i j})} {\eps^2  \over R}.
    \end{aligned}
\end{equation}
Note that we had to estimate at most $n-1+2$ derivatives in $L^\infty$. 
To bound the remaining terms in the RHS of~\eqref{eq:mbfrac}, we must now get an estimate for the terms
\begin{equation}
    \begin{aligned}
    &\int_{\{R/10 \le s \le t \} \cap \mathcal{I}_{i j}} \sum_{H+K \subset I} \Big( |\partial \der^H \phi_i| | \partial \overline{\partial}^{(i)} \der^K \phi_j|  \\
    &\qquad +|\overline{\partial}^{(i)} \der^H \phi_j| |\partial^2  \der^K  \phi_i| + |\partial \der^H \phi_i| |\overline{\partial}^{(i)} \der^K \phi_j| \Big) |\partial_t \der^I \psi_{i j}| \de x \de s =: \mathfrak{B}_3.
    \end{aligned}
\end{equation}
We have that $|\overline{\partial}^{(i)} f| \le C \Big( {R \over s} |\partial f| +   |\overline{\partial}^{(j)} f| \Big)$  in $\{R/10 \le s \le t \} \cap \mathcal{I}_{i j}$, which implies
\begin{equation}
    \begin{aligned}
    & |\mathfrak{B}_3| \leq \int_{\{R/10 \le s \le t \} \cap \mathcal{I}_{i j}} \sum_{H+K \subset I} \Big( |\partial \der^H \phi_i| | \partial \overline{\partial}^{(j)} \der^K \phi_j|  \\
    &\qquad +|\overline{\partial}^{(j)} \der^H \phi_j| |\partial^2  \der^K  \phi_i| + |\partial \der^H \phi_i| |\overline{\partial}^{(j)} \der^K \phi_j| \Big) |\partial_t \der^I \psi_{i j}| \de x \de s\\
    &+ C \int_{\{R/10 \le s \le t \} \cap \mathcal{I}_{i j}} \sum_{H+K \subset I} \Big( |\partial \der^H \phi_i| | \partial^2 \der^K \phi_j|  \\
    &\qquad +|\p \der^H \phi_j| |\partial^2  \der^K  \phi_i| + |\partial \der^H \phi_i| |\p \der^K \phi_j| \Big) |\partial_t \der^I \psi_{i j}| \frac{R}{s} \de x \de s =: \mathfrak{B}_{31} + \mathfrak{B}_{32}.
    \end{aligned}
\end{equation}
Now, the term $\mathfrak{B}_{31}$ can be controlled exactly in the same way as in estimate~\eqref{eq:rminusone1}, as terms like $\overline{\partial}^{(j)} \der^K \phi_j$ have very strong pointwise decay. Thus, we must only control terms $\mathfrak{B}_{32}$.

We begin by noting that we can decompose:
\begin{equation}
    \begin{aligned}
    \partial_t \der^I \psi_{i j} = {1 \over 2} (\partial_{v_j} \der^I \psi_{i j} + \partial_{u_j} \der^I \psi_{i j}).
    \end{aligned}
\end{equation}
Moreover, using bound~\eqref{eq:igortrick} from Lemma~\eqref{lem:goodbad}, we can write
\begin{equation}
    \begin{aligned}
    |\partial_{u_j} \psi_{i j}| \le C {s \over R} ( |\overline{\partial}^{(i)} \psi_{i j}| + |\overline{\partial}^{(j)} \psi_{i j}|),
    \end{aligned}
\end{equation}
since we are restricting to the region $\{R/10 \le s \le t \} \cap \mathcal{I}_{i j}$. Thus, we have that
\begin{equation}\label{eq:beforelinf}
    \begin{aligned}
    &|\mathfrak{B}_{32}| \leq  C \int_{\{R/10 \le s \le t \} \cap \mathcal{I}_{i j}} \sum_{H+K \subset I} \Big( |\partial \der^H \phi_i| | \partial^2 \der^K \phi_j|  \\
    &\qquad +|\p \der^H \phi_j| |\partial^2  \der^K  \phi_i| + |\partial \der^H \phi_i| |\p \der^K \phi_j| \Big) (|\bar \partial^{(i)} \der^I \psi_{i j}| + |\bar \partial^{(j)} \der^I \psi_{i j}|) \de x \de s.
    \end{aligned}
\end{equation}
Now, we note that the terms $\bar \partial^{(i)} \der^I \psi_{i j}$ and $\bar \partial^{(j)} \der^I \psi_{i j}$ all appear in the characteristic $\p_t$ energy for the wave equation satisfied by $\psi_{ij}$ (equation~\eqref{eq:pairs}). Indeed, the terms with $\bar{\p}^{(i)}$ appear in the characteristic $\p_t$-energy through outgoing cones adapted to the $i$-th piece of initial data (the cones $C^{(i)}_{\bar u_i}$, defined by $\{u_i = \bar u_i\} \cap \{t \geq 0\}$, where the coordinate $u_i$ has been introduced in Definition~\ref{def:riviui}), whereas the terms with $\bar{\p}^{(j)}$ appear in the characteristic $\p_t$-energy through outgoing cones $C^{(j)}_{\bar u_j}$ adapted to $\phi_j$. Thus, we write the RHS of inequality~\eqref{eq:beforelinf} as two separate terms, and we shall consider different foliations of outgoing null cones to control each of those terms.

We begin with the term involving $\overline{\partial}^{(i)} V^I \psi_{i j}$. We define
\begin{equation}
    \begin{aligned}
     \mathfrak{B}_{4}:=&\int_{\{R/10 \le s \le t \} \cap \mathcal{I}_{i j}} \sum_{H+K \subset I} \Big( |\partial \der^H \phi_i| | \partial^2 \der^K \phi_j|  \\
    &\qquad +|\p \der^H \phi_j| |\partial^2  \der^K  \phi_i| + |\partial \der^H \phi_i| |\p \der^K \phi_j| \Big) |\bar \partial^{(i)} \der^I \psi_{i j}|  \de x \de s.
    \end{aligned}
\end{equation}

Using the H\"older inequality in mixed Lebesgue spaces, we now estimate $\overline{\partial}^{(i)}\der^I \psi_{i j}$ in $L^2$ of the outgoing cones $C^{(i)}_{ u_i}$ and $L^\infty$ in the $u_i$-direction, while estimating the multiplying factor in $L^2$ of the outgoing cones $C^{(i)}_{ u_i}$ and $L^1$ in the $u_i$-direction. This gives us that
\begin{equation}
    \begin{aligned}
    &\mathfrak{B}_4 \leq \sup_{- R/10 \le  u_i \le R/10} \Vert \overline{\partial}^{(i)} \der^I \psi_{i j} \Vert_{L^2 (\tilde C_{ u_i}^{(i),t,R})} \int_{- R/10}^{R/10} \Big( \Vert  |\partial \der^H \phi_i| | \partial^2 \der^K \phi_j|   \Vert_{L^2 (\tilde C_{ u_i}^{(i),t,R})} \\
    & \qquad \qquad \qquad + \Vert |\p \der^H \phi_j| |\partial^2  \der^K  \phi_i| \Vert_{L^2 (\tilde C_{ u_i}^{(i),t,R})}+ \Vert |\partial \der^H \phi_i| |\p \der^K \phi_j| \Vert_{L^2 (\tilde C_{ u_i}^{(i),t,R})} \Big) \de  u_i,
    \end{aligned}
\end{equation}
where, if $\bar u_i \in \R$, we are denoting by $\tilde C^{(i),R,t}_{\bar u_i}$ the usual outgoing cone $C^{(i)}_{\bar u_i}$ defined by $\{u_i = \bar u_i\} \cap \{t \geq 0\}$ intersected with the set $\{R/10 \le s \le t\} \cap \mathcal{I}_{ij}$.

We now wish to calculate the innermost integral (the integrals in the $L^2$ norms) using the coordinates $(u_i, v_i, u_j, \varphi)$ described in Lemma~\ref{lem:r1r2}, adapted to $\phi_i$ and $\phi_j$. We note that, in the region $\mathcal{I}_{ij} \cap \{t \geq R/10\}$, $r_i$~and~$r_j$ are comparable, as the triangle inequality implies $r_i \leq r_j + 2R$, and since $r_j \geq c R$ for some positive constant $c$, we have $r_i \leq c'r_j$, for some positive constant $c'$. This also implies that $v_i$ is comparable to $v_j$.

Now, using the pointwise bounds on $\phi_i$ and $\phi_j$ given in Lemma~\ref{prop:decphii} (bounds~\eqref{eq:decphii}), we have that, since we are always differentiating at most $n-1$ times, and the estimates~\eqref{eq:decphii} give control over $n+1$ derivatives,
\begin{align}
    &\mathfrak{B}_4 \leq C \sup_{-R/10 \le u_i \le R/10} \Vert \overline{\partial}^{(i)} \der^I \psi_{i j} \Vert_{L^2 (\tilde C_{u_i}^{(i),t,R})} \nonumber  \\
    & \qquad \times \eps^2 \int_{u_i \in [-R/10, R/10]} \left (\int_{v_i \in [R/10, \infty)} \int_{u_j \in \R} \frac{1}{ v^2_i (1 + |u_i|)^{2 + 2 \delta}v^2_j (1 + |u_j|)^{2 + 2 \delta}} \frac{r_ir_j} R \de u_j \, \de v_i  \right)^{\frac 12} \de u_i \nonumber \\\label{eq:linearimp}
    & \quad \leq C \sup_{-R/10 \le u_i \le R/10} \Vert \overline{\partial}^{(i)} \der^I \psi_{i j} \Vert_{L^2 (\tilde C_{u_i}^{(i),t,R})}  \\
    & \qquad \times \eps^2 \int_{u_i \in [-R/10, R/10]} \left (\int_{v_i \in [R/10, \infty)} \int_{u_j \in \R} \frac{1}{ v_i^2 (1 + |u_i|)^{2 + 2 \delta} (1 + |u_j|)^{2 + 2 \delta}} \frac 1 R \de u_j \, \de v_i  \right)^{\frac 12} \de u_i
    \nonumber \\ 
    & \quad \le C \eps^2 {1 \over R} \sup_{- R/10 \le u_i \le R/10} \Vert \overline{\partial}^{(i)} \der^I \psi_{i j} \Vert_{L^2 (\tilde C_{u_i}^{(i),t,R})}. \nonumber
\end{align}

\begin{remark}
    This is the most important estimate in this proposition. Here, we used the crucial fact that we have improved $u$-decay for $\phi_i$ and $\phi_j$. Note in particular that, if the $u_i$-decay for $\phi_i$ were not integrable in $u_i$, we would not be able to close the argument (as the integral in~\eqref{eq:linearimp} would diverge).
\end{remark}
Using the same argument for terms involving $\overline{\partial}^{(j)} \psi_{i j}$ in display~\eqref{eq:beforelinf}, and putting the resulting estimates together, we finally obtain
\begin{equation}
    \begin{aligned}
    &\Vert \partial \der^I \psi_{i j} \Vert_{L^2 (\Sigma_t)}^2
    \\& \quad  \le C \eps^2 {1 \over R} \Big( \sup_{R/10 \le s \le t} \Vert \partial \der^I \psi_{i j} \Vert_{L^2 (\Sigma_s)} \\
    & \qquad \qquad + \sup_{-R/10 \le u_i \le R/10} \Vert \overline{\partial}^{(i)} \der^I \psi_{i j} \Vert_{L^2 (\tilde C^{(i),R,t}_{u_i})} + \sup_{- R/10 \le u_j \le R/10} \Vert \overline{\partial}^{(j)}\der^I \psi_{i j} \Vert_{L^2 (\tilde C^{(j),R,t}_{u_j})} \Big).
    \end{aligned}
\end{equation}
The inequality in the last display holds for all $t\geq 0$. Thus, we can take the supremum in $t$ on both sides. The LHS becomes $\sup_{t\geq 0} \Vert \partial \der^I \psi_{i j} \Vert_{L^2 (\Sigma_t)}^2$. Furthermore, we obtain (note that $\psi_{ij}$ is zero whenever $u_i \leq -3R$ or $u_j \leq -3R$, by domain of dependence):
\begin{equation}
    \begin{aligned}
    &\sup_{t \geq 0} \Vert \partial \der^I \psi_{i j} \Vert_{L^2 (\Sigma_t)}^2  \\
    &\quad \leq C \eps^2 {1 \over R} \left (\sup_{t \geq 0} \Vert \partial \der^I \psi_{i j} \Vert_{L^2 (\Sigma_t)} + \sup_{u_i \in \R} \Vert \overline{\partial}^{(i)} \der^I \psi_{i j} \Vert_{L^2 (C^{(i)}_{u_i})} + \sup_{u_j \in \R} \Vert \overline{\partial}^{(j)} \der^I \psi_{i j} \Vert_{L^2 (C^{(j)}_{u_j})} \right ).
    \end{aligned}
\end{equation}
We proceed by multiplying again the evolution equation in~\eqref{eq:pairscomm} by $\partial_t \der^I \psi_{i j}$ and integrating by parts in the spacetime region $\mathcal{S}_0$ to the future of the hypersurface $\Sigma_0$ and to the past of both the hypersurface $\Sigma_t$ and the outgoing cone $C^{(i)}_{u_i}$, for $u_i \geq -3R$. In other words,
$$
\mathcal{S}_0 := \{(\bar t, \bar r_i, \bar \theta, \bar \varphi):  0 \leq \bar{t} \leq t, \quad \bar t - \bar r_i \leq u_i \}.
$$
It's straightforward to see that the error term appearing in the RHS of the resulting estimate can be controlled in the same fashion as terms $\mathfrak{B}_1$ through $\mathfrak{B}_4$. This gives us that
\begin{equation*}\label{eq:icone}
    \begin{aligned}
    &\Vert \overline{\partial}^{(i)} \der^I \psi_{i j} \Vert_{L^2 (C^{(i)}_{u_i} \cap \{ 0 \leq s \leq t\})}^2
    \\& \quad  \le C \eps^2 {1 \over R} \Big( \sup_{R/10 \le s \le t} \Vert \partial \der^I \psi_{i j} \Vert_{L^2 (\Sigma_s)} \\
    & \qquad \qquad + \sup_{-R/10 \le u_i \le R/10} \Vert \overline{\partial}^{(i)} \der^I \psi_{i j} \Vert_{L^2 (C^{(i),R,t}_{u_i})} + \sup_{- R/10 \le u_j \le R/10} \Vert \overline{\partial}^{(j)}\der^I \psi_{i j} \Vert_{L^2 (C^{(j),R,t}_{u_j})} \Big).
    \end{aligned}
\end{equation*}
Here, we denoted by $C^{(i), R, t}_{u_i}$ the usual outgoing cone $C^{(i)}_{u_i}$ of constant $u_i$ to the future of the initial hypersurface $\Sigma_0$, intersected with the set $\{R/10 \le s \le t\}$:
$$
C^{(i), R, t}_{u_i} := \{(\bar t, \bar r_i, \bar \theta, \bar \varphi): R/10 \leq \bar t \leq t, \quad \bar t - \bar r_i = u_i \}.
$$
The analogous definition holds for $C^{(j), R, t}_{u_j}$.
We can now send $t$ to $\infty$ and use the monotone convergence theorem, giving us that the LHS of display~\eqref{eq:icone} becomes $\Vert \overline{\partial}^{(i)} \der^I \psi_{i j} \Vert_{L^2 (C^{(i)}_{u_i})}^2$. We bound the RHS by the trivial estimate and we obtain, for all $u_i \in \R$,
\begin{equation*}
 \begin{aligned}
    &\Vert \overline{\partial}^{(i)} \der^I \psi_{i j} \Vert_{L^2 (C^{(i)}_{u_i})}^2  \\
    &\quad \le C \eps^2 {1 \over R} \Big( \sup_{t \geq 0} \Vert \partial \der^I \psi_{i j} \Vert_{L^2 (\Sigma_t)}+ \sup_{u_i \geq -3R} \Vert \overline{\partial}^{(i)} \der^I \psi_{i j} \Vert_{L^2 (C^{(i)}_{u_i})} + \sup_{u_j \geq -3R} \Vert \overline{\partial}^{(j)}\der^I \psi_{i j} \Vert_{L^2 (C^{(j)}_{u_j})} \Big).
\end{aligned}
\end{equation*}
Then, taking the supremum in $u_i$ gives us that
\begin{equation*}
 \begin{aligned}
    &\sup_{u_i \in \R}\Vert \overline{\partial}^{(i)} \der^I \psi_{i j} \Vert_{L^2 (C^{(i)}_{u_i})}^2  \\
    &\quad \le C \eps^2 {1 \over R} \Big( \sup_{t \geq 0} \Vert \partial \der^I \psi_{i j} \Vert_{L^2 (\Sigma_t)}  + \sup_{u_i \in \R} \Vert \overline{\partial}^{(i)} \der^I \psi_{i j} \Vert_{L^2 (C^{(i)}_{u_i})} + \sup_{u_j \in \R } \Vert \overline{\partial}^{(j)}\der^I \psi_{i j} \Vert_{L^2 (C^{(j)}_{u_j})} \Big).
\end{aligned}
\end{equation*}
Similarly, using the same argument with respect to outgoing cones adapted to $\phi_j$ gives us that
\begin{equation*}
 \begin{aligned}
    &\sup_{u_j \in \R}\Vert \overline{\partial}^{(j)} \der^I \psi_{i j} \Vert_{L^2 (C^{(j)}_{u_j})}^2 \\
    &\quad \le C \eps^2 {1 \over R} \Big( \sup_{t \geq 0} \Vert \partial \der^I \psi_{i j} \Vert_{L^2 (\Sigma_t)}  + \sup_{u_i \in \R} \Vert \overline{\partial}^{(i)} \der^I \psi_{i j} \Vert_{L^2 (C^{(i)}_{u_i})} + \sup_{u_j \in \R } \Vert \overline{\partial}^{(j)}\der^I \psi_{i j} \Vert_{L^2 (C^{(j)}_{u_j})} \Big).
\end{aligned}
\end{equation*}
Thus, we have that
\begin{equation*}
 \begin{aligned}
    &\sup_{t \geq 0} \Vert \partial \der^I \psi_{i j} \Vert_{L^2 (\Sigma_t)}^2+\sup_{u_i \in \R}\Vert \overline{\partial}^{(i)} \der^I \psi_{i j} \Vert_{L^2 (C^{(i)}_{u_i})}^2+\sup_{u_j \in \R}\Vert \overline{\partial}^{(j)} \der^I \psi_{i j} \Vert_{L^2 (C^{(j)}_{u_j})}^2
    \\& \quad  \le C \eps^2 {1 \over R} \Big( \sup_{t \geq 0} \Vert \partial \der^I \psi_{i j} \Vert_{L^2 (\Sigma_t)} + \sup_{u_i \in \R} \Vert \overline{\partial}^{(i)} \der^I \psi_{i j} \Vert_{L^2 (C^{(i)}_{u_i})} + \sup_{u_j \in \R } \Vert \overline{\partial}^{(j)}\der^I \psi_{i j} \Vert_{L^2 (C^{(j)}_{u_j})} \Big).
\end{aligned}
\end{equation*}
This implies that
\begin{equation}
    \begin{aligned}
    \sup_{t \geq 0} \Vert \partial \der^I \psi_{i j} \Vert_{L^2 (\Sigma_t)}+\sup_{u_i \in \R}\Vert \overline{\partial}^{(i)} \der^I \psi_{i j} \Vert_{L^2 (C^{(i)}_{u_i})}+\sup_{u_j \in \R}\Vert \overline{\partial}^{(j)} \der^I \psi_{i j} \Vert_{L^2 (C^{(j)}_{u_j})} \le C \eps^2 {1 \over R},
    \end{aligned}
\end{equation}
as desired.
\ep

We also note that we can prove an energy estimate for the $\p_t$ energy of $\psi_{i j}$ through the outgoing cone adapted to any of the $\phi_k$. The argument is the same as in the above, as it just involves controlling the same spacetime integral. We record the result here, as it will be important later.

\begin{lemma}\label{lem:psiijencones}
Let us assume the hypotheses of Lemma~\ref{prop:energy}.
Then, we have that, for all multi-indices $I \in I^{\leq n-1}_{ \boldsymbol{\Gamma}^{(h)} }$, $h \in \{i,j\}$, the following inequality holds true:
\begin{equation}
    \begin{aligned}
    \sup_{u_k \in \R} \Vert \bar \p^{(k)} \der^I  \psi_{i j} \Vert_{L^2 (C^{(k)}_{u_k})}  \le C \eps^2 {1 \over R},
    \end{aligned}
\end{equation}
where the constant $C$ is allowed to depend on the distance between $p_i$ and $p_j$, which we recall are the points where the initial data for $\phi_i$ and $\phi_j$, respectively, are centered when $R = 1$. Here, the integration is over outgoing cones $C^{(k)}_{u_k}$ adapted to any of the $\phi_k$. Recall that, for $\bar u_k \in \R$, $C^{(k)}_{\bar u_k}$ is defined as the set $\{u_k = \bar u_k\} \cap \{t \geq 0\}$.
\end{lemma}

We now focus on the interaction between the ``remainder'' $\phi_0$ and $\phi_i$. We use the bounds on the energy of $\phi_0$ obtained in Lemma~\ref{lem:decomposition} to control the energy of $\psi_{0i}$ and $\psi_{i0}$. We collect the $L^2$ estimates in the following lemma:

\begin{lemma}[Estimates on the energy of $\psi_{0i}$ and $\psi_{i0}$, for $i \in \{1, \ldots, N\}$]\label{lem:phi0j}
Let $\psi_{0i}$, $i \in \{1,\ldots, N\}$ the solution to the initial value problem~\eqref{eq:pairs}, with $\phi_i$, $i \in \{0, \ldots, N\}$ constructed as in Lemma~\ref{lem:decomposition}. Similarly, let $\psi_{i0}$, $i \in \{1,\ldots, N\}$ the solution to the initial value problem~\eqref{eq:pairs}, with $\phi_i$, $i \in \{0, \ldots, N\}$ constructed as in Lemma~\ref{lem:decomposition} (recall that the construction depends on the parameter $\eps$, which is chosen small enough so that the conclusion of Lemma~\ref{prop:decphii} holds true). Then, for all $t \geq 0$ and $I \in I^{\leq n-1}_{\boldsymbol{K}^{(i)}}$, the following inequalities hold true:
\begin{equation}\label{eq:enphi0i}
    \Vert \p \der^I \psi_{i0}\Vert_{L^2(\Sigma_t)} \leq C(N,d_\Pi) \varepsilon^2 R^{-\frac 32}, \qquad \Vert \p \der^I \psi_{0i}\Vert_{L^2(\Sigma_t)} \leq C(N,d_\Pi) \varepsilon^2  R^{-\frac 32}.
\end{equation}

\end{lemma}

\begin{proof}
Let's first prove the claim for $\psi_{i0}$ when restricted to $t \in [0, R^{20}]$. 
Recall the form of equation~\eqref{eq:pairs}, specialized to the case $j =0$:
\begin{equation}\label{eq:phi0j}
\Box \psi_{i0} +  F(d \phi_i, d^2 \phi_0) + F(d \phi_0,d^2 \phi_i) = 2 G (d \phi_0, d \phi_i).
\end{equation}
We now commute such equation with $\der^I$, where $I \in I^{\leq n-1}_{\boldsymbol{K}^{(i)}}$:
\begin{equation*}
\begin{aligned}
    &\Box \der^I \psi_{i0} + \sum_{H, K \subset I}(F_{HK}(d \der^H \phi_i, d^2 \der^K \phi_0) + F_{HK}(d \der^H \phi_0, d^2 \der^K \phi_i)) \\
    &= \sum_{H, K \subset I}G_{HK}(d \der^H \phi_0, d \der^K \phi_i).
\end{aligned}
\end{equation*}
Here, $F_{HK}$ and $G_{HK}$ are a collection of (resp.~trilinear and bilinear) null forms as in Definition~\ref{def:null}.

The energy inequality now implies, along with the fundamental null form inequality (Lemma~\ref{lem:nullstruct}):

\begin{equation*}
    \begin{aligned}
    &\p_t \Vert \p \der^I \psi_{i0}\Vert_{L^2(\Sigma_t)} \\
    &\quad \leq C \sum_{H, K \subset I}\Vert \bar \p^{(i)} \der^H \phi_i \Vert_{L^\infty(\Sigma_t)} \Vert \p^2 \der^K \phi_0 \Vert_{L^2(\Sigma_t)}+\Vert \p \der^H \phi_i \Vert_{L^2(\Sigma_t)} \Vert \bar \p^{(i)} \p \der^K \phi_0 \Vert_{L^\infty(\Sigma_t)} \\
    & \quad+ C  \sum_{H, K \subset I}\Vert \p^2 \der^H \phi_i \Vert_{L^2(\Sigma_t)} \Vert \bar\p^{(i)} \der^K \phi_0 \Vert_{L^\infty(\Sigma_t)}+\Vert \bar \p^{(i)}\p \der^H \phi_i \Vert_{L^\infty(\Sigma_t)} \Vert \p \der^K \phi_0 \Vert_{L^2(\Sigma_t)}\\
    &\quad + C \sum_{H, K \subset I}\Vert \bar \p^{(i)} \der^H \phi_i \Vert_{L^\infty(\Sigma_t)} \Vert \p \der^K \phi_0 \Vert_{L^2(\Sigma_t)}+\Vert \p \der^H \phi_i \Vert_{L^2(\Sigma_t)} \Vert \bar \p^{(i)} \der^K \phi_0 \Vert_{L^\infty(\Sigma_t)}
    \end{aligned}
\end{equation*}

We then integrate the previous display in time, using the bounds~\eqref{eq:asymptotici} and~\eqref{eq:linfimpphi0} (valid for all $K_2 \in I^{\leq n}_{\boldsymbol{K}^{(i)}}$):
$$
|\bar \p^{(i)} \der^{K_2} \phi_i| \leq C \varepsilon \frac{1}{(1+v_i^2)(1+|u_i|)^\delta}, \qquad |\bar \p^{(i)} \der^{K_2} \phi_0| \leq C(N,d_\Pi) \varepsilon \frac{R^{- \frac 32}}{(1+v_i^2)(1+|u_i|)^\delta},
$$
along with the $L^2$ bounds~\eqref{eq:decl2phii} and~\eqref{eq:l2impphi0} and the fact that $\psi_{i0}$ has zero initial data. This enabels us to deduce the claim for $\psi_{i0}$:
\begin{equation*}
    \Vert \p \der^I \psi_{i0}\Vert_{L^2(\Sigma_t)} \leq C(N,d_\Pi) \varepsilon^2 R^{-\frac 32}.
\end{equation*}
In a totally analogous manner, we obtain the claim for $\psi_{0i}$.
\end{proof}

\subsection{\texorpdfstring{$L^\infty$}{sup} estimates on the linear equation}\label{sub:linfpsiij}

Having proved suitable $L^2$ estimates for $\psi_{ij}$, our goal is now to use them to deduce $L^\infty$ estimates on the solution $\psi_{ij}$ to the linear equation~\eqref{eq:pairs}. With this in mind, we will employ the $L^2$ estimates just derived in Proposition~\ref{prop:energy}, together with the $R$-weighted Klainerman--Sobolev estimates of Section~\ref{sec:sobolevs} (we recall that these are modifications of estimates first showed  in~\cite{Klainerman1985}). These estimates account for the fact that some of the vector fields carry ${1 \over R}$ weights. We have the following proposition.

\begin{proposition}\label{prop:linfty}
	Let $n \geq 4$, $n \in \N$. Let the smooth function $\psi_{ij}$ arise as a solution to the initial value problem~\eqref{eq:pairs}, where each of the $\phi_i$'s ($i \in \{0, \ldots, N\}$) is constructed according to Lemma~\ref{lem:decomposition}.
	In particular, $\psi_{ij}$ satisfies estimate~\eqref{eq:enpsiij} from Proposition~\ref{prop:energy}, and bound~\eqref{eq:enphi0i} from Lemma~\ref{lem:phi0j}. In these conditions, there exists a constant $C >0$ such that $\psi_{ij}$ satisfies the following $L^\infty$ estimates, for all $t \geq 0$, and for all  $K \in I^{\leq n-4}_{\boldsymbol{K}_R}$:
	\begin{align}
	& \Vert \p \der^K \psi_{ij} \Vert_{L^\infty (\Sigma_{t})} \le {C \eps^{2} \over {R^{\frac 12}(1+t)}} \qquad  \text{ for } i \neq j \text{ and } (i,j) \in \{0, \ldots, N\} \times \{0, \ldots, N\},\label{eq:linf1}\\
	&\Vert \overline{\partial}^{(i)} \der^K \psi_{ij} \Vert_{L^\infty (\Sigma_{t})} \le {C \varepsilon^{2} R^3 \over {(1+t)}^{3 \over 2}}, \qquad \text{ for } i \neq j \text{ and } (i,j) \in \{0, \ldots, N\} \times \{0, \ldots, N\}, \label{eq:linf2}\\
	&|\partial \der^K \psi_{ij}|(t, u_i, \omega) \le {C \varepsilon^{2} R^2 \over {t} (1 + |u_i|)^{1 \over 2}}, \qquad \text{ for } i \neq j \text{ and } (i,j) \in \{0, \ldots, N\} \times \{0, \ldots, N\}.\label{eq:linf3}
	\end{align}
 Here, we considered $(u_i, v_i, \theta_i, \varphi_i)$ coordinates introduced in Definition~\ref{def:riviui} (recall that $t = \frac 1 2(u_i + v_i)$).
\end{proposition}
\begin{proof}
    Let us first focus on the bound~\eqref{eq:linf1}. Let us initially restrict to the region $\mathcal{A}$ for which $\rho = \sqrt{y^2 +z^2} \geq \frac t {10}$. 
    The modified Sobolev inequality of Lemma~\ref{lem:sobspheres} (inequality~\eqref{eq:sobspheres}) now implies:
    \begin{align*}
	|f(t,x,y,z)|^2 \leq C \frac R {t^2}  \sum_{I \in I^{\leq 3}_{\Ga^{(h)}}} \Vert \der^I f \Vert_{L^2(\Sigma_t)}^2.
	\end{align*}
	Here, $h \in \{i,j\}$. We set $f := \partial_a \der^K\psi_{ij}$, with $K \in I^{\leq n-4}_{\boldsymbol{K}_R}$, and $a \in \{t,x,y,z\}$.
	We then have
	\begin{align*}
	|\p \der^K \psi_{ij}|^2 \leq C \frac R {t^2}  \sum_{a\in \{t,x,y,z\}}\sum_{I \in I^{\leq 3}_{\Ga^{(h)}}} \Vert \der^I \p_a \der^K \psi_{ij} \Vert_{L^2(\Sigma_t)}^2\leq C \frac R {t^2}  \sum_{J \in I^{\leq 3}_{\Ga^{(h)}}} \Vert \p \der^J\der^K \psi_{ij}\Vert_{L^2(\Sigma_t)}^2.
	\end{align*}
	The last inequality is obtained by ``commuting out'' the $\p_a$ derivative, keeping in mind that Lie brackets of elements of $\boldsymbol{T}$ and $\boldsymbol{\Gamma}^{(h)}$ are in $\boldsymbol{T}$. We finally employ the spacelike estimates~\eqref{eq:enpsiij} and~\eqref{eq:enphi0i}, to deduce the claim~\eqref{eq:linf1} restricted to the region~$\mathcal{A}$ (recall that, in particular, we have $\boldsymbol{K}_R \subset \boldsymbol{\Gamma}^{(h)}$).
	
	Let us now focus on the region $\mathcal{A}^c$. Recall the hyperboloidal coordinates $(\tau, \alpha, x, \varphi)$ defined in display~\eqref{eq:xflathyp}. Furthermore, recall the hyperboloids $H_{\bar \tau}$ defined by $H_{\bar \tau} := \{\tau= \bar \tau\}$. We now commute equation~\eqref{eq:pairs} with $\der^I$, where $I \in I^{\leq n-1}_{\boldsymbol{\Gamma}^{(h)}}$, with $h \in \{i,j\}$. We then multiply the commuted equation~\eqref{eq:pairs} by $\p_t \der^I \psi_{ij}$ and integrate in the spacetime region between $\Sigma_0$ and $H_{ \tau}$. The inhomogeneous error terms which arise from this estimate are treated exactly as in the proof of Proposition~\ref{prop:energy}. We now use Lemma~\ref{lem:unifspacelike} in Appendix~\ref{sec:emst} to deduce that the future boundary term on $H_\tau$ controls all derivatives of $\der^I \phi$ in a non-degenerate manner (the lemma follows from the fact that every hypersurface $H_{ \tau} \cap \mathcal{A}^c$ is uniformly spacelike). We thus arrive at the estimate:
	\begin{equation}\label{eq:energyhyp}
	    \Vert \p \der^I \psi_{ij} \Vert_{L^2(H_{\tau} \cap  \mathcal{A}^c)} \leq C \frac{\eps^2}{R}.
	\end{equation}
	Now, we use the Sobolev embedding on hyperboloids (Lemma~\ref{lem:sobhyp}) to deduce:
	\begin{equation}
        |f(t,x,y,z)| \leq  \frac C t  \sum_{I \in I^{\leq 3}_{\Ga^{(h)}}}\Vert  \der^I f\Vert_{L^2(H_\tau \cap \mathcal{A}^c)}
    \end{equation}
    where $t^2 = y^2 +z^2 +\tau^2$, and $(t,x,y,z)$ belongs to $H_\tau \cap \mathcal{A}^c$. Setting $f = \p_a \der^K \psi_{ij}$, with $a \in \{t,x,y,z\}$ and $K\in I^{\leq n-4}_{\boldsymbol{K}_R}$, we have
    \begin{align*}
	    |\p \der^K \psi_{ij}| \leq C \frac 1 {t}  \sum_{a\in \{t,x,y,z\}}\sum_{I \in I^{\leq 3}_{\Ga^{(h)}}} \Vert \der^I \p_a \der^K \psi_{ij} \Vert_{L^2(H_\tau \cap \mathcal{A}^c)}\leq C \frac 1 {t}  \sum_{J \in I^{\leq 3}_{\Ga^{(h)}}} \Vert \p \der^J\der^K \psi_{ij}\Vert_{L^2(H_\tau \cap \mathcal{A}^c)}.
	\end{align*}
	The last inequality again follows from the fact that Lie brackets of elements of $\boldsymbol{\Gamma}^{(h)}$ and $\boldsymbol{T}$ are in $\boldsymbol{T}$. We now use estimate~\eqref{eq:energyhyp} to conclude the proof of claim~\eqref{eq:linf1}.
	
	Finally, claims~\eqref{eq:linf2} and~\eqref{eq:linf3} follow directly from the $L^2$ estimate~\eqref{eq:enpsiij}, combined with the classical Klainerman--Sobolev inequality with $R$-weights of Lemma~\ref{lem:classicalks} (and the standard argument involving integration along a line of constant $v_i$-coordinate, cf.~{\bf Step 3} of the proof of Theorem~\ref{thm:nonlinear} in Section~\ref{sec:mainproof}).
\end{proof}

\section{Proof of Theorem~\ref{thm:nonlinear}}\label{sec:mainproof}
In this section, we will close the argument and use all the results obtained so far to conclude global existence to the nonlinear equation~\eqref{eq:nlw}.

\begin{proof}[Proof of Theorem~\ref{thm:nonlinear}]
We begin by noting that it suffices to prove uniform estimates assuming that $R \ge R_0$, for some positive number $R_0$. Indeed, if $R \le R_0$, we can restrict $\eps$ to a smaller value, depending on $R_0$, and use the classical theory to conclude global stability. Thus, in the following, we shall without loss of generality use the fact that, restricting to the case $R \geq R_0$, we have the inequality $\log{(R)} \le C R^\delta$ for some uniform positive constant $C$.

We start from the equation satisfied by $\Psi$:
\begin{equation}\label{eq:Psi}
    \begin{aligned}
        &\Box \Psi +F(d \Psi, d^2 \Psi) \\
        &+ \sum_{\substack{i,j = 0,\ldots, N\\i\neq j}}(F(d \psi_{ij}, d^2 \Psi)+F(d \Psi, d^2 \psi_{ij})) + \sum_{\substack{g,h,i,j = 0,\ldots, N\\g \neq h, i\neq j}}F(d \psi_{gh}, d^2 \psi_{ij})\\  
        & +\sum_{i=0}^N\sum_{\substack{g,h = 0,\ldots, N\\g\neq h}} ( F(d \phi_i, d^2 \psi_{gh}) + F(d \psi_{gh}, d^2 \phi_i))\\  
        & +\sum_{i=0}^N ( F(d \phi_i, d^2 \Psi) + F(d \Psi, d^2 \phi_i)) \\
         &\qquad =G(d \Psi, d\Psi) \\
         &\qquad + \sum_{\substack{i,j = 0,\ldots, N\\i\neq j}}(G(d \psi_{ij}, d \Psi)+G(d \Psi, d \psi_{ij}))+ \sum_{\substack{g,h,i,j = 0,\ldots, N\\g\neq h, i\neq j}}G(d \psi_{gh}, d \psi_{ij})\\  
        & \qquad +\sum_{i=0}^N\sum_{\substack{g,h = 0,\ldots, N\\g\neq h}} ( G(d \phi_i, d \psi_{gh}) + G(d \psi_{gh}, d \phi_i))\\  
        & \qquad +\sum_{i=0}^N ( G(d \phi_i, d \Psi) + G(d \Psi, d \phi_i)),\\
        & \Psi|_{t=0} = 0,\\
        & \p_t\Psi|_{t=0} = 0.
    \end{aligned}
\end{equation}
 We shall use the estimates obtained for $\phi_i$ and $\psi_{ij}$ above to prove estimates and global existence for $\Psi$. We will thus have global existence and estimates for the long time behavior of $\phi$, as we recall the definition
 \begin{equation}
     \Psi := \phi - \sum_{i=0}^N \phi_i - \sum_{\substack{i,j \in \{0, \ldots, N\} \\ i \neq j}}\psi_{ij},
 \end{equation}
and all the $\phi_i$'s and $\psi_{ij}$'s are global.

For ease of notation, let us suppose for the remainder of the proof that $G$ is identically $0$, as estimating the terms in $G$ is exactly analogous to estimating the terms in $F$, and it requires fewer derivatives.

In order to solve this equation, we shall now set up a continuity argument in the parameter $T$, which we define to be the maximal time for which the initial value problem~\eqref{eq:Psi} admits a solution on the set $[0,T]\times \R^3$ and satisfies the following bootstrap estimates\footnote{Note that, in the course of our argument, we will need to be able to control at most $\lfloor N_0/2 \rfloor+2$ derivatives of $\Psi$ in $L^\infty$. This suggests that we should require $|J| \leq \lfloor N_0/2 \rfloor+1$ in the bootstrap assumptions.} on $[0,T]$,
for all $I \in I^{\leq N_0}_{\boldsymbol{K}_R}$ and all $J \in I^{\leq  \lfloor N_0/2 \rfloor+1}_{\boldsymbol{K}_R}$:

\begin{align}
&\Vert \p \der^I \Psi \Vert_{L^2(\Sigma_t)} \leq \eps^{ 3 -\delta} R^{-\frac 32 + \delta} \quad \hspace{120pt} \text{ for } \quad t \in [0,T], \label{eq:bstp1}   \\
&\Vert (1+|u_i|)^{-\frac 12 - \frac \delta 2} \bar \p^{(i)} \der^I \Psi \Vert_{L^2([0,T]\times \R^3)} \leq \eps^{ 3 -\delta}  R^{-\frac 32 + \delta}  \quad \hspace{8pt} \text{ for } \quad t \in [0,T], \quad i \in \{0, \ldots, N\}, \label{eq:bstp2}  \\
&|\p \der^J  \Psi(t, r, \theta,\varphi)| \leq \eps^{ 3 -\delta}  (1+t)^{-1} R^{-\frac 1 2 +\delta}  \quad \hspace{55pt} \text{ for } \quad t \in [0,T], \label{eq:bstp3}   \\
&|\bar \p^{(i)} \der^J \Psi(t, r, \theta,\varphi)| \leq \eps^{ 3 -\delta}  (1+t)^{-\frac 32 } R^{\frac 32 + \delta}  \quad \hspace{50pt} \text{ for } \quad t \in [R^{20}, T], \quad i \in \{0, \ldots, N\}, \label{eq:bstp4}\\
&|\p \der^J \Psi(t, r, \theta,\varphi)| \leq  \eps^{ 3 -\delta}  (1+v_i)^{-1}(1+|u_i|)^{-\frac 12} R^{\frac 12 + \delta}  \quad  \text{ for } \quad t \in [R^{20}, T], \quad i \in \{0, \ldots, N\}. \label{eq:bstp5}
\end{align}

Here, we adopted the convention that the interval $[a,b]$, with $a>b$, is the empty set. By the local existence theory for a single quasilinear equation, we know that $T > 0$ (note that, for this first non-emptiness step, the precise value of $T$ here is allowed to depend on $R$).

We will then proceed to improve these bootstrap estimates. Namely, we will prove, under the bootstrap assumptions~\eqref{eq:bstp1}--\eqref{eq:bstp5}, the following bounds for all $I \in I^{\leq N_0}_{\boldsymbol{K}_R}$ and all $J \in I^{\leq  N_0-3}_{\boldsymbol{K}_R}$:

\begin{align}
&\Vert \p \der^I \Psi \Vert_{L^2(\Sigma_t)} \leq \eps^{ 3 -\frac {3 \delta} 4} R^{-\frac 32 + \delta} \quad \hspace{111pt} \text{ for } \quad t \in [0,T], \label{eq:bstpi1}   \\
&\Vert (1+|u_i|)^{-\frac 12 - \frac \delta 2} \bar \p^{(i)} \der^I \Psi \Vert_{L^2([0,T]\times \R^3)} \leq \eps^{ 3 -\frac {3 \delta} 4}  R^{-\frac 32 + \delta}  \quad \hspace{8pt} \text{ for } \quad t \in [0,T], \ \  i \in \{0, \ldots, N\},\label{eq:bstpi2}  \\
&|\p \der^J  \Psi(t, r, \theta,\varphi)| \leq \eps^{ 3 -\frac {7 \delta} 8}  (1+t)^{-1} R^{-\frac 1 2 +\delta}  \quad \hspace{46pt}  \text{ for } \quad  t \in [0,T], \label{eq:bstpi3}   \\
&|\bar \p^{(i)} \der^J \Psi(t, r, \theta,\varphi)| \leq \eps^{ 3 - \frac {7 \delta} 8} (1+t)^{-\frac 32 } R^{\frac 32 + \delta}  \quad \hspace{50pt} \text{ for } \quad t \in [R^{20}, T], \ \ i \in \{0, \ldots, N\}, \label{eq:bstpi4}\\
&|\p \der^J \Psi(t, r, \theta,\varphi)| \leq \eps^{ 3 - \frac{7 \delta} 8}  (1+v_i)^{-1}(1+|u_i|)^{-\frac 12} R^{\frac 12 + \delta}  \quad  \text{ for } \quad t \in [R^{20}, T], \ \ i \in \{0, \ldots, N\}. \label{eq:bstpi5}
\end{align}
This will imply that the initial value problem~\eqref{eq:Psi} admits a global-in-time solution, upon choosing $N_0 \geq 7$ (this choice ensures that $N_0 - 3 \geq \big\lfloor \frac{N_0} 2 \big\rfloor + 1$).

\begin{remark}\label{rmk:numberder}
    Note that, in the course of our argument, we will also need to be able to estimate the functions $\psi_{ij}$ in $L^\infty$, which is done through an application of Proposition~\ref{prop:linfty}. More precisely, we will require (in the worst-case scenario) $L^\infty$ bounds for at most $N_0 + 2$ derivatives of $\psi_{ij}$ (this is noted in the analysis of term $\boldsymbol{(a_{12}})$ below). This means that we will have to set, in the statement of Proposition~\ref{prop:linfty}, $n-3 = N_0 + 2$, which means we have to require $n \geq N_0 + 5 = 12$. On the other hand, since we require bounds on $n+7$ derivatives in $L^\infty$ from Lemma~\ref{lem:decomposition}, this translates to $n \geq 12+7 = 19$, which is the number of derivatives we require in the statement of the main theorem (Theorem~\ref{thm:main}).
\end{remark}

\vspace{10pt}
\begin{remark}\label{rmk:improvedR}
Under the additional assumption that $\phi_0 \equiv 0$, looking ahead to the proof of the large data theorem (Theorem~\ref{thm:largedata}), we note that the estimates we will show are actually better in terms of the parameter $R$, and read as follows, for all $I \in I^{\leq N_0}_{\boldsymbol{K}_R}$ and all $J \in I^{\leq  N_0 -3}_{\boldsymbol{K}_R}$: 
\begin{align}
    &\Vert \p \der^I \Psi \Vert_{L^2(\Sigma_t)} \leq \eps^{ 3 -\delta} R^{-\frac 32 +\frac 3 4 \delta} \quad  \hspace{111pt} \text{ for } \quad t \in [0,T], \label{eq:Rbstpi1}   \\
    &\Vert (1+|u_i|)^{-\frac 12 - \frac \delta 2} \bar \p^{(i)} \der^I \Psi \Vert_{L^2([0,T]\times \R^3)} \leq \eps^{ 3 -\delta} R^{-\frac 32 +\frac 3 4 \delta}   \quad  \hspace{8pt}\text{ for } \quad t \in [0,T], \ i \in \{1, \ldots, N\},\label{eq:Rbstpi2}  \\
    &|\p \der^J  \Psi(t, r, \theta,\varphi)| \leq \eps^{ 3 -\delta} (1+t)^{-1} R^{-\frac 1 2 + \frac 78 \delta}  \quad  \hspace{46pt} \text{ for } t \in [0,T], \label{eq:Rbstpi3}   \\
    &|\bar \p^{(i)} \der^J \Psi(t, r, \theta,\varphi)| \leq \eps^{ 3 -\delta} (1+t)^{-\frac 32 } R^{\frac 32 + \frac 78 \delta}  \quad \hspace{50pt} \text{ for } \quad t \in [R^{20}, T], \ i \in \{1, \ldots, N\}, \label{eq:Rbstpi4}\\
    &|\p \der^J \Psi(t, r, \theta,\varphi)| \leq \eps^{ 3 -\delta}  (1+v_i)^{-1}(1+|u_i|)^{-\frac 12} R^{\frac 12 + \frac 78 \delta}  \quad  \text{ for } \quad t \in [R^{20}, T], \ i \in \{1, \ldots, N\}. \label{eq:Rbstpi5}
\end{align}
\end{remark}

\vspace{10pt}

We divide the proof in several {\bf Steps}.
In {\bf Step 0}, we will introduce some preliminary calculations. In {\bf Step 1}, we will prove estimate~\eqref{eq:bstpi1} by a $\p_t$-energy estimate, whereas, in {\bf Step 2}, we will prove estimates~\eqref{eq:bstpi2}, again by a $\p_t$-energy estimate. As we shall see, the only difference between {\bf Step 1} and {\bf Step 2} is in how the boundary terms are treated, as the bulk terms in the respective estimates will be bounded in essentially the same way. {\bf Step 2} will be moreover divided in several parts. We will first estimate the ``nonlinear'' term $F(d \Psi, d^2 \Psi)$, we will then proceed to estimate the ``mixed'' terms 
$$
\sum_{\substack{i,j = 0,\ldots, N\\i\neq j}}(F(d \psi_{ij}, d^2 \Psi)+F(d \Psi, d^2 \psi_{ij})) +\sum_{i=0}^N ( F(d \phi_i, d^2 \Psi) + F(d \Psi, d^2 \phi_i)).$$
Finally, we will estimate the ``inhomogenous'' terms  
$$
\sum_{i=0}^N\sum_{\substack{g,h = 0,\ldots, N\\g\neq h}} ( F(d \phi_i, d^2 \psi_{gh}) + F(d \psi_{gh}, d^2 \phi_i))  + \sum_{\substack{g,h,i,j = 0,\ldots, N\\g \neq h, i\neq j}}F(d \psi_{gh}, d^2 \psi_{ij}).
$$
To conclude the proof, in {\bf Step 3} an easy application of the Sobolev lemmas in Section~\ref{sec:sobolevs} will be sufficient to show estimates~\eqref{eq:bstpi3}--\eqref{eq:bstpi5} from~\eqref{eq:bstpi1} and~\eqref{eq:bstpi2}.

\vspace{20pt}
{\bf Step 0}.
Let $n \in \N$, $n \geq 0$. We now define $E_n [f] (\Sigma_t)$ for any smooth function $f$ to be the following energy integral:
\begin{equation}
\begin{aligned}
E_{n+1}^2 [f] (\Sigma_t) := \sum_{I \in I_{\boldsymbol{K}_R}^{\le n}} \Vert\partial \der^I f\Vert_{L^2 (\Sigma_t)}^2.
\end{aligned}
\end{equation}
We now consider the family of (truncated) the null cones $\hat C^{(i)}_{\bar u_i, T}$:
these are defined as follows, for $\bar u_i \in \R$:
$$
\hat C^{(i)}_{\bar u_i, T} := C^{(i)}_{\bar u_i} \cap \{0 \le t \le T\},
$$
where the usual outgoing cone $C^{(i)}_{\bar u_i}$ is defined as $\{u_i = \bar u_i\} \cap \{t \geq 0\}$.

We define $E_n [f] (\hat C^{(i)}_{\bar u_i, T})$ for a smooth function $f$ as follows:
\begin{equation}\label{eq:nullendef}
\begin{aligned}
& E_{1}^2 [f] (\hat C^{(i)}_{\bar u_i}) := \int_{\hat C^{(i)}_{\bar u_i, T}} Q(\p_{v_i},\partial_t) \, r_i^2 \, \de v_i \, \de \theta_i \, \de \varphi_i, \\
& E_{n+1}^2 [f] (\hat C^{(i)}_{\bar u_i, T}) :=\sum_{I \in I_{\boldsymbol{K}_R}^{\le n}} E_{1}^2 [\der^I f] (\hat C^{(i)}_{\bar u_i, T}).
\end{aligned}
\end{equation}
Here, $Q(\cdot, \cdot)$ is the stress--energy--momentum tensor associated to the linear wave equation as defined in Section~\ref{sec:emst}, and $\p_{v_i}$ is defined as a coordinate vector field arising from the coordinates $(u_i, v_i, \theta_i, \varphi_i)$ defined in Definition~\ref{def:riviui}. It is moreover a properly normalized Lorentzian normal to the cones ${\hat C^{(i)}_{\bar u_i, T}}$. 

We also recall that the energy integrals in display~\eqref{eq:nullendef} give control over good derivatives, i.e. there exists a positive constant $C$ such that the inequality holds true:
\begin{equation}
    E_{n+1}^2 [f] (\hat C^{(i)}_{\bar u_i, T}) \geq C \sum_{I \in I^{\leq n}_{\boldsymbol{K}_R}}\int_{\hat C^{(i)}_{\bar u_i, T}} |\bar \p^{(i)} \der^I f|^2 \, r_i^2 \, \de \hat v_i \, \de \theta_i \, \de \varphi_i.
\end{equation}

We now commute equation~\eqref{eq:Psi} with $\der^I$, where $I \in I^{\leq N_0}_{\boldsymbol{K}_R}$. We obtain (recall that we set $G = 0$ for ease of argument):
\begin{align}
        &\Box \der^I \Psi + \sum_{H+K \subset I} \Big( F_{HK}(d \der^H \Psi, d^2 \der^K \Psi)\Big) \nonumber\\
        &+ \sum_{H+K \subset I} \Big(\sum_{\substack{i,j = 0,\ldots, N\\i\neq j}}(F_{HK}(d \der^H \psi_{ij}, d^2 \der^K \Psi)+F_{HK}(d \der^H \Psi, d^2 \der^K \psi_{ij}))\nonumber\\
        & \qquad \qquad \qquad +\sum_{i=0}^N ( F_{HK}(d \der^H \phi_i, d^2 \der^K \Psi) + F_{HK}(d \der^H \Psi, d^2 \der^K \phi_i)) \Big) \label{eq:Psicomm} \\  
        & + \sum_{H+K \subset I} \Big(\sum_{i=0}^N\sum_{\substack{g,h = 0,\ldots, N\\g\neq h}} ( F_{HK}(d \der^H\phi_i, d^2 \der^K \psi_{gh}) + F_{HK}(d \der^H \psi_{gh}, d^2 \der^K \phi_i))\nonumber\\  
        &\qquad \qquad + \sum_{\substack{g,h,i,j = 0,\ldots, N\\g \neq h, i\neq j}}F_{HK}(d \der^H \psi_{gh}, d^2 \der^K \psi_{ij}) \Big)= 0.\nonumber
\end{align}
We can now proceed to {\bf Step 1} of the proof.
\vspace{20pt}

{\bf Step 1}. We now wish to perform a $\p_t$-energy estimate on equation~\eqref{eq:Psicomm}. To this end, we note that the following lemma holds true.

\begin{lemma}[Main lemma on spacelike $L^2$ estimates]\label{lem:spacel2} Let $F$ be a trilinear null form as in Definition~\ref{def:null}. There exists a positive $\eps_0$ and a positive constant $C >0$ such that the following holds.
For any smooth function $\tilde \Psi$ satisfying the following initial value problem on the set $[0,T]\times \R^3$:
\begin{equation}\label{eq:quassimpli}
\begin{aligned}
    &\Box \tilde \Psi + F(d g, d^2 \tilde \Psi) = h,\\
    &\tilde \Psi|_{t=0} = 0,\\
    &\p_t \tilde \Psi|_{t=0} = 0,
\end{aligned}
\end{equation}
with $g$ smooth satisfying the bound $|\p g_j| \leq \eps_0$, and with $h$ smooth, we have that the inequality holds, for all $t \in [0,T]$:
\begin{equation}
\begin{aligned}
    &E_1[\tilde \Psi](\Sigma_t) \leq C \int_0^{t} \int_{\R^3}\big( |F^{\alpha\beta\gamma} \ \p_\beta \p_\alpha g \ \p_t \tilde \Psi  \  \p_\gamma \tilde \Psi| +  |F^{\alpha \beta\gamma} \ \p_t \p_\alpha g \ \p_\beta  \tilde \Psi \p_\gamma \tilde \Psi| \big)\, \de x \de s\\
    & \qquad + C\int_0^{t} \int_{\R^3}  |\p_t \tilde \Psi| |h|\, \de x \de s+ C E_1[\tilde \Psi](\Sigma_0).
\end{aligned}
\end{equation}
\end{lemma}

\begin{proof}[Proof of Lemma~\ref{lem:spacel2}]
    First of all, let us multiply equation~\eqref{eq:quassimpli} by $\p_t \tilde \Psi$. Let us write the tensor $F$ in components as $F^{\alpha\beta\gamma}$. Recall that, without loss of generality, we can assume that $F$ is symmetric in the last two indices: $F^{\alpha\beta\gamma} = F^{\alpha\gamma\beta}$. Then, we have, 
\begin{equation*}
    \begin{aligned}
    &\p_t \tilde \Psi \  F^{\alpha\beta\gamma} \ \p_\alpha g  \ \p_\beta \p_\gamma  \tilde \Psi\\
    &= \p_\beta \big(F^{\alpha\beta\gamma} \  \p_\alpha g \ \p_t \tilde \Psi  \  \p_\gamma  \tilde \Psi\big) -  F^{\alpha\beta\gamma} \  \p_\beta \p_\alpha g \ \p_t \tilde \Psi  \  \p_\gamma \tilde \Psi\\
    &\quad -\frac 12 \p_t \big(\p_\alpha g \ F^{\alpha \beta\gamma} \ \p_\beta  \tilde \Psi \p_\gamma \tilde \Psi\big) + \frac 12  \p_t \p_\alpha g\ F^{\alpha \beta\gamma} \ \p_\beta  \tilde \Psi \p_\gamma \tilde \Psi.
    \end{aligned}
\end{equation*}
We then let $t_1 \geq 0$, and integrate the resulting equation on the region $\{0 \leq t \leq t_1\} \cap \{t+ r/2 \leq A\}$, for $A > 0$ large. Note that the boundary of the region $\{t+ r/2 \leq A\}$ is strictly spacelike. This means that, possibly restricting $\varepsilon_0$ to be smaller, we have the following inequality:
\begin{equation}
\begin{aligned}
    &\frac 12 \int_{\Sigma_{t_1} \cap \{r \leq 2(A-t_1)\}} (|\p \tilde \Psi|^2  -  |F^{\alpha\beta\gamma} \p_\alpha g  \p_\beta \tilde \Psi \p_\gamma \tilde \Psi|)\, \de x  \\
    & \qquad \leq \int_0^{t_1} \int_{\R^3}\big( |F^{\alpha\beta\gamma} \ \p_\beta \p_\alpha g \ \p_t \tilde \Psi  \  \p_\gamma \tilde \Psi| + \frac 12  |F^{\alpha \beta\gamma} \ \p_t \p_\alpha g \ \p_\beta  \tilde \Psi \p_\gamma \tilde \Psi| \big)\, \de x \de t\\
    & \qquad + \int_0^{t_1} \int_{\R^3}  |\p_t \tilde \Psi| |h|\, \de x \de t + C E_1[\tilde \Psi](\Sigma_0).
\end{aligned}
\end{equation}
Using now the fact that $|\p g| \leq \eps_0$, we conclude that: 
\begin{equation}
\begin{aligned}
    &\int_{\Sigma_{t_1} \cap \{r \leq 2(A-t_1)\}} |\p \tilde \Psi|^2\de x \\
    & \qquad \leq C \int_0^{t_1} \int_{\R^3}\big( |F^{\alpha\beta\gamma} \ \p_\beta \p_\alpha g \ \p_t \tilde \Psi  \  \p_\gamma \tilde \Psi| + \frac 12  |F^{\alpha \beta\gamma} \ \p_t \p_\alpha g \ \p_\beta  \tilde \Psi \p_\gamma \tilde \Psi| \big)\, \de x \de t\\
    & \qquad + C \int_0^{t_1} \int_{\R^3}  |\p_t \tilde \Psi| |h|\, \de x \de t+ C E_1[\tilde \Psi](\Sigma_0).
\end{aligned}
\end{equation}
We now conclude by the monotone convergence theorem, upon sending $A \to \infty$.
\end{proof}

Let us now rewrite the commuted equation~\eqref{eq:Psicomm} highlighting the top-order terms:

\begin{align}
        &\Box \der^I \Psi + F\Big(\sum_{i=0}^N d \phi_i + \sum_{\substack{i,j=\{0, \ldots, N\} \\i \neq j}} d\psi_{ij} + d \Psi, d^2 \der^I \Psi\Big)\nonumber\\
        &+ \sum_{\substack{H+K \subset I \\ K \neq I}} F_{HK}(d \der^H \Psi, d^2 \der^K \Psi)\nonumber \\
        &+ \sum_{\substack{H+K \subset I \\ K \neq I }} \Big(\sum_{\substack{i,j = 0,\ldots, N\\i\neq j}}(F_{HK}(d \der^H \psi_{ij}, d^2 \der^K \Psi)+F_{HK}(d \der^H \Psi, d^2 \der^K \psi_{ij}))\nonumber\\
        & \qquad \qquad \qquad +\sum_{i=0}^N ( F_{HK}(d \der^H \phi_i, d^2 \der^K \Psi) + F_{HK}(d \der^H \Psi, d^2 \der^K \phi_i)) \Big) \label{eq:Psicommhigh} \\
        &+\sum_{\substack{i,j = 0,\ldots, N\\i\neq j}} F(d \Psi, d^2 \der^I \psi_{ij}) + \sum_{i=0}^N F(d \Psi, d^2 \der^I \phi_i)\nonumber \\
        & + \sum_{H+K \subset I} \Big(\sum_{i=0}^N\sum_{\substack{g,h = 0,\ldots, N\\g\neq h}} ( F_{HK}(d \der^H\phi_i, d^2 \der^K \psi_{gh}) + F_{HK}(d \der^H \psi_{gh}, d^2 \der^K \phi_i))\nonumber\\  
        &\qquad \qquad + \sum_{\substack{g,h,i,j = 0,\ldots, N\\g \neq h, i\neq j}}F_{HK}(d \der^H \psi_{gh}, d^2 \der^K \psi_{ij}) \Big)= 0.\nonumber
\end{align}

We now apply Lemma~\ref{lem:spacel2}, with $\tilde \Psi := \der^I \Psi$, and
$$
g := \sum_{i=0}^N \phi_i + \sum_{\substack{i,j=\{0, \ldots, N\} \\i \neq j}} \psi_{ij} + \Psi.
$$
Note that, by the linear estimates of Proposition~\ref{prop:linfty} and by the bootstrap assumptions, we can assume that this $g$ is in the conditions of the above lemma, i.~e. $|\p g| \leq \eps_0$.

We then have the following estimate:
\begin{equation}\label{eq:prel2space}
\begin{aligned}
    &E_1[\der^I \Psi](\Sigma_t) \leq C \int_0^{t} \int_{\R^3}\big( |F^{\alpha\beta\gamma} \ \p_\beta \p_\alpha g \ \p_t \der^I \Psi  \  \p_\gamma \der^I \Psi| +  |F^{\alpha \beta\gamma} \ \p_t \p_\alpha g \ \p_\beta  \der^I \Psi \p_\gamma \der^I \Psi| \big)\, \de x \de s\\
    & \qquad + C\int_0^{t} \int_{\R^3}  |\p_t \der^I \Psi| |h|\, \de x \de s+ C E_1[\der^I \Psi](\Sigma_0).
\end{aligned}
\end{equation}
Here, $h$ is composed of the terms contained in lines 2 to 7 of display~\eqref{eq:Psicommhigh}. Expanding all the terms in display~\eqref{eq:prel2space} now gives:
\begin{align}
    &E_1[\der^I \Psi](\Sigma_t)- C E_1[\der^I \Psi](\Sigma_0)\nonumber\\
    &\leq C \int_0^{t} \int_{\R^3}\big(\underbrace{ |F^{\alpha\beta\gamma} \ \p_\beta \p_\alpha \Psi\ \p_t \der^I \Psi  \  \p_\gamma \der^I \Psi|}_{(a_1)} + \underbrace{ |F^{\alpha \beta\gamma} \ \p_t \p_\alpha \Psi \ \p_\beta  \der^I \Psi \p_\gamma \der^I \Psi|}_{(a_2)} \big)\, \de x \de s\nonumber\\
    &+C \sum_{i = 0}^N \int_0^{t} \int_{\R^3}\big( \underbrace{|F^{\alpha\beta\gamma} \ \p_\beta \p_\alpha \phi_i \ \p_t \der^I \Psi  \  \p_\gamma \der^I \Psi|}_{(a_3)} +  \underbrace{|F^{\alpha \beta\gamma} \ \p_t \p_\alpha \phi_i \ \p_\beta  \der^I \Psi \p_\gamma \der^I \Psi|}_{(a_4)} \big)\, \de x \de s\nonumber\\
    &+C \sum_{\substack{i,j \in \{0, \ldots, N\}\\i \neq j}} \int_0^{t} \int_{\R^3}\big( \underbrace{|F^{\alpha\beta\gamma} \ \p_\beta \p_\alpha \psi_{ij} \ \p_t \der^I \Psi  \  \p_\gamma \der^I \Psi|}_{(a_5)} +  \underbrace{|F^{\alpha \beta\gamma} \ \p_t \p_\alpha \psi_{ij} \ \p_\beta  \der^I \Psi \p_\gamma \der^I \Psi|}_{(a_6)} \big)\, \de x \de s\nonumber\\
    & \qquad + C \sum_{\substack{H+K \subset I \\ K \neq I}} \int_0^{t} \int_{\R^3}  \underbrace{|\p_t \der^I \Psi| |F_{HK}(d \der^H \Psi, d^2 \der^K \Psi)|}_{(a_7)}\, \de x \de s \nonumber \\
    &+ C \sum_{\substack{H+K \subset I \\ K \neq I }} \int_0^t \int_{\R^3} \Big(\sum_{\substack{i,j = 0,\ldots, N\\i\neq j}}(\underbrace{|F_{HK}(d \der^H \psi_{ij}, d^2 \der^K \Psi)|}_{(a_8)}+\underbrace{|F_{HK}(d \der^H \Psi, d^2 \der^K \psi_{ij})|}_{(a_9)})\label{eq:masterl2space} \\
    & \qquad \qquad \qquad +\sum_{i=0}^N ( \underbrace{|F_{HK}(d \der^H \phi_i, d^2 \der^K \Psi)|}_{(a_{10})} + \underbrace{|F_{HK}(d \der^H \Psi, d^2 \der^K \phi_i)|}_{(a_{11})}) \Big)|\p_t \der^I \Psi| \de x \de s\nonumber \\  
    & \qquad + C \sum_{\substack{i,j = 0,\ldots, N\\i\neq j}} \int_0^{t} \int_{\R^3}  \underbrace{|\p_t \der^I \Psi| |F(d \Psi, d^2 \der^I \psi_{ij})|}_{(a_{12})}\, \de x \de s + C \sum_{i=0}^N \int_0^{t} \int_{\R^3} \underbrace{ |\p_t \der^I \Psi| |F(d \Psi, d^2 \der^I \phi_i)|}_{(a_{13})}\, \de x \de s\nonumber\\
    & + C \sum_{H+K \subset I} \int_0^t \int_{\R^3}\Big(\sum_{i=0}^N\sum_{\substack{g,h = 0,\ldots, N\\g\neq h}} ( \underbrace{|F_{HK}(d \der^H\phi_i, d^2 \der^K \psi_{gh})|}_{(a_{14})} + \underbrace{|F_{HK}(d \der^H \psi_{gh}, d^2 \der^K \phi_i)|}_{(a_{15})})\nonumber\\  
    &\qquad \qquad + \sum_{\substack{g,h,i,j = 0,\ldots, N\\g \neq h, i\neq j}}\underbrace{|F_{HK}(d \der^H \psi_{gh}, d^2 \der^K \psi_{ij}) |}_{(a_{16})}\Big) |\p_t \der^I \Psi|\de x \de s.\nonumber
\end{align}

We now proceed to estimate the terms in the previous display one by one. Because inequality~\eqref{eq:masterl2space} controls the square of the energy, we note that we must recover the square of the bootstrap assumption~\eqref{eq:bstpi1}. We shall be wasteful when deriving our estimates in terms of the parameter $R$. Indeed, to close the bootstrap argument we must only recover $R^{-3 + 2 \delta}$ (this is what we need for~\eqref{eq:bstpi1}), but we shall keep track of which estimates ``have room in $R$''. We shall do this by bounding these terms by a factor of $R^{-3+\frac 3 2 \delta}$ instead of $R^{-3 + 2 \delta}$. This will be needed in the proof of Theorem \ref{thm:largedata} in Section \ref{sec:largedata}. See also Remark~\ref{rmk:improvedR}.
We shall also be wasteful in terms of the parameter $\eps$. Terms which gain an improvement in powers of $\eps$ will be bounded by $\eps^6$.

\begin{enumerate}
\item[$\boldsymbol{(a_1) + (a_2)}$.] We have the following estimates (as usual, we assume that the interval $[a,b]$, with $a > b$, is the empty set):
\begin{align}
    &\int_0^{t} \int_{\R^3}\big( |F^{\alpha\beta\gamma} \ \p_\beta \p_\alpha \Psi\ \p_t \der^I \Psi  \  \p_\gamma \der^I \Psi| + |F^{\alpha \beta\gamma} \ \p_t \p_\alpha \Psi \ \p_\beta  \der^I \Psi \p_\gamma \der^I \Psi| \big)\, \de x \de s \nonumber\\ 
    & \quad \leq C \Big(\int_{[0,R^{20}]}\Vert \p^2 \Psi \Vert_{L^\infty(\Sigma_s)} \de s\Big) \sup_{s \in [0,R^{20}]} \Vert \p \der^I \Psi \Vert^2_{L^2(\Sigma_s)}\nonumber \\
    &\qquad + C \int_{[R^{20}, t]} \int_{\R^3}|\p^2 \Psi|\, | \bar \p^{(0)}\der^I \Psi| \,| \p \der^I \Psi| \de x \de s\nonumber \\
    &\qquad + C\Big( \int_{[R^{20}, t]} \Vert\bar \p^{(0)} \p \Psi\Vert_{L^\infty(\Sigma_s)} \de s \Big) \sup_{s \in [R^{20},t]} \Vert \p \der^I \Psi \Vert^2_{L^2(\Sigma_s)}\nonumber \\
    & \quad \leq C \eps^{9 - 3\delta} R^{-3 + 2\delta} R^{-\frac 12 +\delta} \log R \nonumber\\
    &\quad + C \eps^{3 -\delta} R^{-\frac 32 + \delta} \Big(\int_{[R^{20}, t]} \int_{\R^3}(1+s)^{1+\delta}|\p^2 \Psi|^2 \,|\bar \p^{(0)} \der^I \Psi|^2 \de x \de s \Big)^{\frac 12}\nonumber \\
    & \qquad + C \eps^{9 - 3\delta} R^{-3+ 2\delta} \int_{R^{20}}^t R^{\frac 32 + \delta}(1+s)^{-\frac 32}\de s\nonumber\\
    &\quad \leq C \eps^7 R^{-3}. \label{terms12}
\end{align}

Here, in the first inequality we used the lemma on the structure of null forms (Lemma~\ref{lem:nullstruct}), in the second inequality we used the bootstrap assumptions~\eqref{eq:bstp1}--\eqref{eq:bstp5}, plus the Cauchy--Schwarz inequality on the second term (multiplying and dividing by $(1+s)^{\frac{1+\delta} 2}$), and finally in the last inequality we used estimate~\eqref{eq:bulkest2} from Lemma \ref{lem:shortghost}.
\item[$\boldsymbol{(a_3)+(a_4)}$.] We need to estimate the following expression:
\begin{equation*}
    \sum_{i = 0}^N \int_0^{t} \int_{\R^3}\big( |F^{\alpha\beta\gamma} \ \p_\beta \p_\alpha \phi_i \ \p_t \der^I \Psi  \  \p_\gamma \der^I \Psi| + |F^{\alpha \beta\gamma} \ \p_t \p_\alpha \phi_i \ \p_\beta  \der^I \Psi \p_\gamma \der^I \Psi|\big)\, \de x \de s.
\end{equation*}
For all $i \in \{0, \ldots, N\}$, we have, by Lemma~\ref{lem:nullstruct},
\begin{equation*}
\begin{aligned}
    &\int_0^{t} \int_{\R^3}\big( |F^{\alpha\beta\gamma} \ \p_\beta \p_\alpha \phi_i \ \p_t \der^I \Psi  \  \p_\gamma \der^I \Psi| + |F^{\alpha \beta\gamma} \ \p_t \p_\alpha \phi_i \ \p_\beta  \der^I \Psi \p_\gamma \der^I \Psi|\big)\, \de x \de s\\
    & \quad \leq C\int_0^t \int_{\R^3}\big( | \p \bar \p^{(i)} \phi_i| \ |\p \der^I \Psi |  \,  |\p \der^I \Psi| + | \p \p \phi_i| \ |\bar \p^{(i)} \der^I \Psi |  \,  |\p \der^I \Psi| \big)\, \de x \de s.
\end{aligned}
\end{equation*}
Now, by bounds~\eqref{eq:asymptotici} and~\eqref{eq:linfimpphi0}, we know that $|\p \bar \p^{(i)}\phi_i| \leq C(N,d_\Pi) \varepsilon (1+t)^{-2}$. This, together with the bootstrap assumption~\eqref{eq:bstp1}, implies that
\begin{equation*}
\begin{aligned}
    &\int_0^t \int_{\R^3} | \p \bar \p^{(i)} \phi_i| \ |\p \der^I \Psi |  \,  |\p \der^I \Psi| \de x \de s \\
    &\leq C(N,d_\Pi) \eps^{7-2\delta}R^{-3 + 2 \delta}\int_0^t \frac 1 {(1+s)^2}\de s \leq C(N, d_{\Pi}) \eps^{7 - 2 \delta} R^{-3 + 2 \delta}.
\end{aligned}
\end{equation*}
We note that this term does not have any ``room'' in terms of the parameter $R$. We also explicitly marked the dependence of the constant $C$ on the quantity $d_\Pi$ (which is defined in equation~\eqref{eq:dpidef}), as well as on the number $N$.

Now, for the remaining term, we have, applying H\"older's inequality, the bootstrap assumption~\eqref{eq:bstp1}, and the estimate~\eqref{eq:bulkest2} from Lemma \ref{lem:shortghost},
\begin{equation*}
\begin{aligned}
    &\int_0^t \int_{\R^3}  | \p \p \phi_i| \ |\bar \p^{(i)} \der^I \Psi |  \,  |\p \der^I \Psi| \, \de x \de s \leq C\Big(\int_0^t \frac 1 {(1+s)^{1+\delta}}\Vert \p \der^I \Psi \Vert^2_{L^2(\Sigma_s)} \de s\Big)^{\frac 12} \\
    & \quad \times \Big( \int_0^t \int_{\R^3}(1+s)^{1+\delta} |\p\p\phi_i|^2 |\bar \p^{(i)} \der^I \Psi|^2\de x \de s \Big)^{\frac 12} \leq C(N, d_\Pi) \eps^{7 - 2 \delta} R^{-3 + 2 \delta}.
\end{aligned}
\end{equation*}
This term similarly does not have any space in $R$. Adding these terms gives us that 
\begin{equation} \label{terms34}
    \begin{aligned}
        \boldsymbol{(a_3)}+\boldsymbol{(a_4)}\leq C(N, d_\Pi) \eps^{7 - 2 \delta} R^{-3 + 2 \delta} \le \eps^6 R^{-3 + 2 \delta},
    \end{aligned}
\end{equation}
where we have used the fact that $\eps$ can depend on $N$ and on $d_\Pi$. This suffices to bound terms $\boldsymbol{(a_3)}$ and $\boldsymbol{(a_4)}$.

\begin{remark}\label{rmk:improved34}
    Note that, if $\phi_0$ was identically $0$, all the functions $\psi_{ij}$ would be supported in the set $\{ t \geq R/10$\}. This would in particular imply that, in this case,
    $$
    \boldsymbol{(a_3)}+\boldsymbol{(a_4)} \leq C(N, d_\Pi) \eps^{7-2\delta}R^{-3}.
    $$
\end{remark}

\item[$\boldsymbol{(a_5)+(a_6)}$.] In this case, we have the estimates:
\begin{equation*}
\begin{aligned}
    &\int_0^{t} \int_{\R^3}\big( |F^{\alpha\beta\gamma} \ \p_\beta \p_\alpha \psi_{ij} \ \p_t \der^I \Psi  \  \p_\gamma \der^I \Psi| +  |F^{\alpha \beta\gamma} \ \p_t \p_\alpha \psi_{ij} \ \p_\beta  \der^I \Psi \p_\gamma \der^I \Psi| \big)\, \de x \de s\\
    & \quad \leq \int_0^{R^{20}} \int_{\R^3} |\p^2 \psi_{ij}| |\p \der^I \Psi|^2 \de x \de s+ \int_{R^{20}}^{t} \int_{\R^3}|\p  \bar \p^{(i)} \psi_{ij}| \ |\p \der^I \Psi|^2\de x \de s  \\
    &\qquad + \int_{R^{20}}^{t} \int_{\R^3}| \p^2 \psi_{ij}| \ |\bar \p^{(i)}  \der^I \Psi| \, | \p \der^I \Psi| \, \de x \de s\\
    & \quad \leq C(N, d_\Pi) \eps^7 R^{-3} + C\eps^{3 - \delta} R^{- \frac 3 2+ \delta} \Big(\int_{R^{20}}^t \int_{\R^3} (1+s)^{1+\delta} |\p^2\psi_{ij}|^2 \, |\bar \p^{(i)} \der^I \Psi |^2 \de x \de s \Big)^{\frac 12}.
\end{aligned}
\end{equation*}
Here, we used the fundamental lemma on null forms (Lemma~\ref{lem:nullstruct}), plus the bounds in Proposition~\ref{prop:linfty}, together with the bootstrap assumptions and the H\"older inequality in the last line. We then bound:
\begin{equation*}
    \Big(\int_{R^{20}}^t \int_{\R^3} (1+s)^{1+\delta} |\p^2\psi_{ij}|^2 \, |\bar \p^{(i)} \der^I \Psi |^2 \de x \de s \Big)^{\frac 12} \leq C(N, d_\Pi) \eps^7 R^{-3},
\end{equation*}
where again we used the estimates for $\psi_{ij}$ contained in Proposition~\ref{prop:linfty}, together with estimate~\eqref{eq:bulkest2} from Lemma \ref{lem:shortghost}. Summing gives us
\begin{equation} \label{terms56}
    \begin{aligned}
        \boldsymbol{(a_5)+(a_6)} \leq C(N, d_\Pi) \eps^7 R^{-3} \le \eps^6 R^{-3},
    \end{aligned}
\end{equation}
where we have once again used that $\eps$ can depend on $N$ and $d_\Pi$.

\item[$\boldsymbol{(a_7)}$.] We have to estimate the term
\begin{equation*}
    \sum_{\substack{H+K \subset I \\ K \neq I}} \int_0^{t} \int_{\R^3}  |\p_t \der^I \Psi| |F_{HK}(d \der^H \Psi, d^2 \der^K \Psi)|\, \de x \de s.
\end{equation*}
This can be controlled in the same way as $\boldsymbol{(a_1) + (a_2)}$. We obtain:
\begin{equation}\label{terms7}
    \boldsymbol{(a_7)} \leq C \eps^7 R^{-3}.
\end{equation}

\item[$\boldsymbol{(a_8)+(a_9)}$.] We need to bound the terms
\begin{equation*}
    \sum_{\substack{H+K \subset I \\ K \neq I }} \int_0^t \int_{\R^3} \Big(\sum_{\substack{i,j = 0,\ldots, N\\i\neq j}}(|F_{HK}(d \der^H \psi_{ij}, d^2 \der^K \Psi)|+|F_{HK}(d \der^H \Psi, d^2 \der^K \psi_{ij})|)\Big)|\p_t \der^I \Psi| \de x \de s.
\end{equation*}
These terms can be dealt with exactly as in the case of terms $\boldsymbol{(a_5)+(a_6)}$, always estimating $\psi_{ij}$ in $L^\infty$. We note that we need to bound at most $N_0 +1$ derivatives of $\psi_{ij}$ in $L^\infty$. We obtain:
\begin{equation}\label{terms89}
    \boldsymbol{(a_8)+(a_9)} \leq C(N, d_\Pi) \eps^7 R^{-3}.
\end{equation}

\item[$\boldsymbol{(a_{10})+(a_{11})}$.] In this case, we need to bound the terms
\begin{equation*}
    \sum_{\substack{H+K \subset I \\ K \neq I }} \int_0^t \int_{\R^3} \Big(\sum_{i=0}^N ( |F_{KH}(d \der^H \phi_i, d^2 \der^K \Psi)| + |F_{KH}(d \der^H \Psi, d^2 \der^K \phi_i)|) \Big)|\p_t \der^I \Psi| \de x \de s.
\end{equation*}
The same reasoning as the one for terms $\boldsymbol{(a_3)+(a_4)}$ will give the required bound. Again, we need to be careful as we always estimate $\phi_i$ in $L^\infty$, and in the worst case we need to be able to estimate $N_0 +1$ derivatives of $\phi_i$. We obtain:
\begin{equation} \label{terms1011}
    \begin{aligned}
        \boldsymbol{(a_{10})}+\boldsymbol{(a_{11})}\leq C(N, d_\Pi) \eps^{7 - 2 \delta} R^{-3 + 2 \delta}.
    \end{aligned}
\end{equation}
\begin{remark}\label{rmk:improved1011}
    As in the term $\boldsymbol{(a_{3})}+\boldsymbol{(a_{4})}$, we note that, under the additional assumption $\phi_0 \equiv 0$, we have the improved estimate (since in that case the support of $\Psi$ is contained in the set $\{t \geq R/10\}$):
    \begin{equation}
        \boldsymbol{(a_{10})}+\boldsymbol{(a_{11})} \leq C(N, d_\Pi) \eps^{7-2\delta} R^{-3}.
    \end{equation}
\end{remark}
\item[$\boldsymbol{(a_{12})}$.] This is the term
\begin{equation*}
    \sum_{\substack{i,j = 0,\ldots, N\\i\neq j}} \int_0^{t} \int_{\R^3} |\p_t \der^I \Psi| |F(d \Psi, d^2 \der^I \psi_{ij})|\, \de x \de s.
\end{equation*}
This can be dealt with exactly in the same way as $\boldsymbol{(a_5)+(a_6)}$, but note that we need to estimate $N_0 + 2$ derivatives of $\psi_{ij}$ in $L^\infty$. We obtain:
\begin{equation} \label{terms12d}
    \begin{aligned}
        \boldsymbol{(a_{12})} \leq C(N, d_\Pi) \eps^7 R^{-3} \le \eps^6 R^{-3},
    \end{aligned}
\end{equation}
since we are allowed to choose $\eps$ small in terms of $N$ and $d_\Pi$.

\item[$\boldsymbol{(a_{13})}$.] We need to bound the term
\begin{equation*}
     \sum_{i=0}^N \int_0^{t} \int_{\R^3}  |\p_t \der^I \Psi| |F(d \Psi, d^2 \der^I \phi_i)|\, \de x \de s.
\end{equation*}
This can be dealt with exactly in the same way as $\boldsymbol{(a_3)+(a_4)}$, but note that we need to estimate $N_0 + 2$ derivatives of $\phi_{i}$ in $L^\infty$. We obtain:
\begin{equation} \label{terms13}
    \begin{aligned}
        \boldsymbol{(a_{13})}\leq C(N, d_\Pi) \eps^{7 - 2 \delta} R^{-3 + 2 \delta}.
    \end{aligned}
\end{equation}
\begin{remark}\label{rmk:improved13}
    Note again that, under the additional assumption $\phi_0 \equiv 0$, we have that $\Psi$ is supported in the set $\{t \geq R/10\}$. In particular, in that case, we have the improved estimate:
\begin{equation} \label{terms13i}
\begin{aligned}
    \boldsymbol{(a_{13})}\leq C(N, d_\Pi) \eps^{7 - 2 \delta} R^{-3}.
\end{aligned}
\end{equation}
\end{remark}
\item[$\boldsymbol{(a_{14})+(a_{15})}$.]
 In this case, we need to estimate the terms:
\begin{equation*}
    \sum_{H+K \subset I} \int_0^t \int_{\R^3}\Big(\sum_{i=0}^N\sum_{\substack{g,h = 0,\ldots, N\\g\neq h}} ( |F_{HK}(d \der^H\phi_i, d^2 \der^K \psi_{gh})| + |F_{HK}(d \der^H \psi_{gh}, d^2 \der^K \phi_i)|)\Big) |\p_t \der^I \Psi|\de x \de s.
\end{equation*}
Let us first suppose that either $g=0$ or $h=0$. We have, using Lemma~\ref{lem:nullstruct} on the structure of null forms, combined with H\"older's inequality,
\begin{equation*}
    \begin{aligned}
        &\int_0^t \int_{\R^3}\Big( |F_{HK}(d \der^H\phi_i, d^2 \der^K \psi_{gh})| + |F_{HK}(d \der^H \psi_{gh}, d^2 \der^K \phi_i)|\Big) |\p_t \der^I \Psi|\de x \de s\\
        & \leq C\sum_{J_1, J_2 \in I^{\leq N_0}_{\boldsymbol{K}_R}}\sup_{s \in [0, R^{20}]} \Vert \p \der^I \Psi \Vert_{L^2(\Sigma_s)}  \sup_{s \in [0, R^{20}]} \Vert \p^{\leq 2} \der^{J_1} \psi_{gh} \Vert_{L^2(\Sigma_s)} \int_0^{R^{20}} \Vert \p^{\leq 2} \der^{J_2} \phi_i \Vert_{L^\infty(\Sigma_s)} \de s\\
        & \quad + C\sum_{J_1, J_2 \in I^{\leq N_0}_{\boldsymbol{K}_R}} \int_{R^{20}}^t \int_{\R^3}|\bar \p^{(i)} \p^{\leq 1} \der^{J_1} \phi_i |\, |\p^{\leq 2} \der^{J_2} \psi_{gh} | \,|\p \der^I \Psi| \de x \de s\\
        &\quad + C\sum_{J_1, J_2 \in I^{\leq N_0}_{\boldsymbol{K}_R}} \int_{R^{20}}^t \int_{\R^3}|\p^{\leq 2} \der^{J_1} \phi_i |\, |\bar \p^{(i)} \p^{\leq 1} \der^{J_2} \psi_{gh} | \,|\p \der^I \Psi| \de x \de s.
    \end{aligned}
\end{equation*}
We subsequently use the bootstrap assumptions~\eqref{eq:bstp1}, the linear estimates on $\psi_{gh}$ (recalling that either $g=0$ or $h=0$) in Proposition~\ref{prop:linfty}, and the estimates on $\phi_i$ in Lemma~\ref{prop:decphii} to bound the last display by
\begin{equation*}
\begin{aligned}
    & C(N, d_\Pi) \Big( \eps^{6 - \delta} R^{-\frac 32 + \delta} R^{-\frac 32} \log{(R)} +  \sum_{J_2 \in I^{\leq N_0}_{\boldsymbol{K}_R}} \int_{R^{20}}^t \int_{\R^3}{\eps \over (1 + s)^2} \, |\p^{\leq 2} \der^{J_2} \psi_{gh} | \,|\p \der^I \Psi| \de x \de s \Big)\\
    &\quad + C\eps^{3 - \delta} R^{-\frac 32 + \delta} \sum_{J_1, J_2 \in I^{\leq N_0}_{\boldsymbol{K}_R}} \Big(\int_{R^{20}}^t \int_{\R^3}|\p^{\leq 2} \der^{J_1} \phi_i |^2\, |\bar \p^{(i)} \p \der^{J_2} \psi_{gh} |^2 (1+s)^{1+\delta}\de x \de s\Big)^{\frac 12}\\
    &\leq C(N, d_\Pi) \eps^{6 - \delta} R^{-3 + \delta} \log{(R)}\\
    &\quad+ C(N, d_\Pi) \eps^{4 - \delta}R^{-\frac 32 + \delta} \sum_{J_1, J_2 \in I^{\leq N_0}_{\boldsymbol{K}_R}} \Big(\int_{R^{20}}^t \int_{\R^3}\frac 1 {(1+s)^{1-\delta}(1+|u_i|)^{2+2\delta}}\, |\bar \p^{(i)} \p \der^{J_2} \psi_{gh} |^2 \de x \de s\Big)^{\frac 12}\\
    &\leq C(N, d_\Pi) \eps^{6 - \delta} R^{-3 + \delta} \log{(R)}.
\end{aligned}
\end{equation*}
Here, we also used estimate~\eqref{eq:bulkest2} from Lemma \ref{lem:shortghost}. We note that this term has very little ``room'' in both the parameters $\eps$ and $R$. It is in fact this term which determines the best powers in $R$ and $\eps$ that we can use. Once again, because $\eps$ is allowed to depend on $N$ and $d_\Pi$, we have that, upon possibly restricting $\eps$ to a smaller value,
\begin{equation} \label{terms14151}
    \begin{aligned}
        C(N, d_\Pi) \eps^{6 - \delta} R^{-3 + \delta} \log{(R)} \le \eps^{6 - {3 \delta \over 2}} R^{-3 + \delta} \log{(R)},
    \end{aligned}
\end{equation}
giving us the desired result.

On the other hand, if both $g$ and $h$ are different from $0$, we know that, by an easy domain of dependence argument, $\psi_{gh}$ is supported in the set $\{t \geq R/10\}$. Then, we proceed to estimate, using also Lemma~\eqref{lem:nullstruct},
\begin{align*}
        &\int_0^t \int_{\R^3}\Big( |F_{HK}(d \der^H\phi_i, d^2 \der^K \psi_{gh})| + |F_{HK}(d \der^H \psi_{gh}, d^2 \der^K \phi_i)|\Big) |\p_t \der^I \Psi|\de x \de s\\
        & = \int_{R/10}^t \int_{\R^3}\Big( |F_{HK}(d \der^H\phi_i, d^2 \der^K \psi_{gh})| + |F_{HK}(d \der^H \psi_{gh}, d^2 \der^K \phi_i)|\Big) |\p_t \der^I \Psi|\de x \de s\\
        & \leq C\sum_{J_1, J_2 \in I^{\leq N_0}_{\boldsymbol{K}_R}} \int_{R/{10}}^t \int_{\R^3}|\bar \p^{(i)} \p^{\leq 1} \der^{J_1} \phi_i |\, |\p^{\leq 2} \der^{J_2} \psi_{gh} | \,|\p \der^I \Psi| \de x \de s\\
        &\quad + C\sum_{J_1, J_2 \in I^{\leq N_0}_{\boldsymbol{K}_R}} \int_{R/{10}}^t \int_{\R^3}|\p^{\leq 2} \der^{J_1} \phi_i |\, |\bar \p^{(i)} \p^{\leq 1} \der^{J_2} \psi_{gh} | \,|\p \der^I \Psi| \de x \de s\\
        &\leq C(N, d_\Pi)\sum_{J_2 \in I^{\leq N_0}_{\boldsymbol{K}_R}} \int_{R/{10}}^t \int_{\R^3}\frac{\eps}{s^2}|\p^{\leq 2} \der^{J_2} \psi_{gh} | \,|\p \der^I \Psi| \de x \de s\\
        &\quad + C(N, d_\Pi)\sum_{J_1, J_2 \in I^{\leq N_0}_{\boldsymbol{K}_R}} \sup_{s \geq 0} \Vert \partial \der^I \Psi \Vert_{L^2 (\Sigma_s)} \int_{R / 10}^t \Vert | \partial^{\leq 2} \der^{J_1} \phi_i| \, | \overline{\partial}^{(i)} \partial^{\leq 1} \der^{J_2} \psi_{g h}| \Vert_{L^2 (\Sigma_s)} \de s \\
        &\leq C(N, d_\Pi) \eps^{6 - \delta} R^{-3} + C(N, d_\Pi) \eps^{6 - \delta} R^{-3 +\delta}.
\end{align*}

The last inequality follows from Lemma \ref{lem:shortghost}, estimate~\eqref{eq:bulkest1}. We note that this term also does not have much ``room'' in the parameters $R$ and $\eps$. It does, however, have ``room'' of size $\frac \delta 2$ in the parameter $R$. Once again, because $\eps$ is allowed to depend on $N$ and $d_\Pi$, we have that, upon possibly restricting $\eps$ to a smaller value,
\begin{equation} \label{terms14152}
    \begin{aligned}
        C(N, d_\Pi) \eps^{6 -\delta} R^{-3} + C(N, d_\Pi) \eps^{6 - {\delta}} R^{-3 + \delta} \le  \eps^{6 - {3 \delta \over 2}} R^{-3 + \delta}.
    \end{aligned}
\end{equation}
This concludes the bounds on $\boldsymbol{(a_{14})+(a_{15})}$:
\begin{equation}
    \boldsymbol{(a_{14})+(a_{15})} \leq \eps^{6 - \frac{3\delta}{2}} R^{-3 + \delta}.
\end{equation}

\item[$\boldsymbol{(a_{16})}$.] We have to estimate the following expression:
\begin{equation*}
    \sum_{H+K \subset I} \int_0^t \int_{\R^3}\Big(\sum_{\substack{g,h,i,j = 0,\ldots, N\\g \neq h, i\neq j}}|F_{HK}(d \der^H \psi_{gh}, d^2 \der^K \psi_{ij}) |\Big) |\p_t \der^I \Psi|\de x \de s.
\end{equation*}
Let us break up the integral in two pieces, as usual:
\begin{equation*}
\begin{aligned}
    &\int_0^t \int_{\R^3}\Big(\sum_{\substack{g,h,i,j = 0,\ldots, N\\g \neq h, i\neq j}}|F_{HK}(d \der^H \psi_{gh}, d^2 \der^K \psi_{ij}) |\Big) |\p_t \der^I \Psi|\de x \de s\\
    &\leq \int_0^{R^{20}} \int_{\R^3}\Big(\sum_{\substack{g,h,i,j = 0,\ldots, N\\g \neq h, i\neq j}}\underbrace{|F_{HK}(d \der^H \psi_{gh}, d^2 \der^K \psi_{ij}) |}_{(a)}\Big) |\p_t \der^I \Psi|\de x \de s\\
    &\quad + \int_{[R^{20},t]} \int_{\R^3}\Big(\sum_{\substack{g,h,i,j = 0,\ldots, N\\g \neq h, i\neq j}}\underbrace{|F_{HK}(d \der^H \psi_{gh}, d^2 \der^K \psi_{ij}) |}_{(b)}\Big) |\p_t \der^I \Psi|\de x \de s.
\end{aligned}
\end{equation*}
Here, as usual, we adopt the convention that the interval $[a,b]$, with $a > b$, is the empty set. We focus first on term $(a)$. We have, by the linear estimates in Lemma~\ref{prop:linfty},
\begin{equation*}
    \begin{aligned}
        &(a)\leq C(N, d_
        \Pi)\sum_{J_1 \in I^{\leq N_0}_{\boldsymbol{K}_R}}\int_0^{R^{20}} \frac{ \eps^2}{ (1+t)R^{\frac 12}}\Vert \p^2 \der^{J_1} \psi_{ij} \Vert_{L^2(\Sigma_s)} \Vert \p \der^I \Psi \Vert_{L^2(\Sigma_s)}\de s \\
        & \quad \leq C(N, d_\Pi) \eps^{7 - \delta} R^{-1} R^{-\frac 12} \log{(R)} R^{-\frac 32 +\delta} \leq C(N, d_\Pi) \eps^{7 - \delta} R^{-3 + \delta} \log{(R)}.
    \end{aligned}
\end{equation*}
Using the fact that $\eps$ can depend on $N$ and $d_\Pi$, we get that
\begin{equation}
    \begin{aligned}
        C(N, d_\Pi) \eps^{7 - \delta} R^{-3 + \delta} \log{(R)} \le \eps^6 R^{-3 + 2\delta}.
    \end{aligned}
\end{equation}

Focusing now on term $(b)$, we have, again using Lemma~\ref{lem:nullstruct},
\begin{equation*}
    \begin{aligned}
         &\int_{[R^{20},t]} \int_{\R^3}|F_{HK}(d \der^H \psi_{gh}, d^2 \der^K \psi_{ij}) |\,|\p_t \der^I \Psi|\de x \de s\\
         & \leq C\int_{[R^{20},t]} \int_{\R^3}|\p \der^H \psi_{gh}| \, | \p \bar \p^{(0)} \der^K \psi_{ij} |\,|\p_t \der^I \Psi|\de x \de s\\
         & \quad +C\int_{[R^{20},t]} \int_{\R^3}|\bar \p^{(0)} \der^H \psi_{gh}| \, | \p^2 \der^K \psi_{ij} |\,|\p_t \der^I \Psi|\de x \de s\\
         & \leq C(N, d_\Pi) \int_{[R^{20},t]} \frac{\eps^2}{s^{\frac 32}} R^{\frac 32} \Vert \p \der^K \psi_{gh} \Vert_{L^2(\Sigma_s)} \Vert \p \der^I \Psi \Vert_{L^2(\Sigma_s)} \de s\\
         & \quad + C(N, d_\Pi) \int_{[R^{20},t]} \frac{\eps^2}{s^{\frac 32}} R^{\frac 32} \Vert \p \der^K \psi_{gh} \Vert_{L^2(\Sigma_s)} \Vert \p \der^I \Psi \Vert_{L^2(\Sigma_s)} \de s\\
         & \leq C(N, d_\Pi) \eps^{7 - \delta} R^{-3}.
    \end{aligned}
\end{equation*}
Here, we used the classical Klainerman--Sobolev inequality in Lemma~\ref{lem:classicalks} (which is wasteful in terms of $R$-weights), the bounds on energies of $\psi_{gh}$ and $\psi_{ij}$ in Proposition~\ref{prop:energy}, and finally the bootstrap assumption~\ref{eq:bstp1}. Summing gives us that
\begin{equation}\label{terms16}
    \begin{aligned}
        \boldsymbol{(a_{16})}\leq C(N, d_\Pi) \eps^{7 - \delta} R^{-3}\le \eps^6 R^{-3+2 \delta},
    \end{aligned}
\end{equation}
since we can choose $\eps$ to be small, depending on $d_\Pi$ and $N$.
\end{enumerate}

Having completed {\bf Step 1} of the proof, we now proceed to {\bf Step 2}.

\vspace{20pt}

{\bf Step 2}. We now proceed to recover the averaged characteristic energy estimate \eqref{eq:bstpi2}. Averaging the outgoing characteristic energy in the $u$-direction allows us to control nonlinear interactions in which it is more convenient to estimate ``good derivatives'' of the solution in $L^2$ of the outgoing null cones (this is the case when, for example, all of the commutation vector fields are applied to the ``good derivative'' in the nonlinear error terms). When deriving estimates for quantities which are differentiated at the top order (in our case, when they are differentiated $N_0$ times), however, this will cause problems, as the background Minkowski structure differs from the causal structure induced by the quasilinear wave equation at hand. Indeed, because the light cones associated to the metric defined by the solution only asymptotically become the Minkowskian light cones, we obtain error terms involving the incoming derivatives on the outgoing cones that we must control in $L^2$. Note that the incoming derivative is not intrinsic to the outgoing light cone, and as such it cannot be controlled by the na\"ive energy estimate. However, since we are seeking to prove an \emph{averaged} estimate, the error term produced in this way can be controlled in the same way as other errors. Note that the fact that the ``true'' causal structure asymptotically approaches the Minkowski causal structure is encoded in the fact that these error terms will gain ``good weights''. In the following discussion, we shall treat the situation for arbitrary solutions to appropriately perturbed wave equations, and later specify to the equation at hand.

Recall that, for $\bar u_i \in \R$, and $i \in \{0, \ldots, N\}$, we defined the cones $C^{(i)}_{\bar u_i} := \{u_i = \bar u_i\} \cap \{t \geq 0\}$, and the associated ``good'' derivatives $\bar \p^{(i)}$ in Definition~\ref{def:shorthand}.

\begin{lemma} \label{avcharest1}
    Let $\tilde{\Psi}$ be a sufficiently smooth solution, which is moreover decaying at infinity, to the equation $\Box \tilde{\Psi} = F^{\alpha \beta \gamma} \partial_\alpha h \, \partial_\beta \partial_\gamma \tilde{\Psi} + H$, where $F$ satisfies the classical null condition, and $h$, as well as $H$, are sufficiently smooth functions. Then, there exists some $\eps_0 > 0$ such that, for all $i \in \{0, \ldots, N\}$, we have the estimates
    \begin{equation} \label{avcharest11}
        \begin{aligned}
            \int_0^t \int_{\Sigma_s} |\overline{\partial}^{(i)} \tilde{\Psi}|^2 {1 \over (1 + |u|)^{1 + \delta}} \de x \de s \le C \Vert \partial \tilde{\Psi} \Vert_{L^2 (\Sigma_0)}^2 + C \int_0^t \int_{\Sigma_s} |H| |\partial_t \tilde{\Psi}| \de x \de s
            \\ + C\int_0^t \int_{\Sigma_s} |\overline{\partial} \partial^{\leq 1} h| |\partial \tilde{\Psi}|^2 \de x \de s + C\int_0^t \int_{\Sigma_s} |\partial^{\leq 1} \partial h| |\overline{\partial}^{(i)} \tilde{\Psi}| |\partial \tilde{\Psi}| \de x \de s,
        \end{aligned}
    \end{equation}
    as long as $|\partial h| \le \eps_0$, and where $C$ can only depend on $\delta$.
\end{lemma}
\bp
We restrict to the case $i = 0$, as the other cases are completely analogous. Note that we adopt the convention $\bar \p := \bar \p^{(0)}$, where $\bar \p^{(0)}$ was defined in Definition~\ref{def:shorthand}. We also recall the usual functions $t$, $u$, $v$, and $r$. Furthermore, we denote: $C_u := C^{(0)}_u$. We shall finally denote by $C_{\bar u}^{\bar t}$ the portion of the cone $C_{\bar u}$ between $\Sigma_0$ and $\Sigma_{\bar t}$\,: $C_{\bar u}^{\bar t}:= C_{\bar u} \cap \{0\leq t \leq \bar t\}$.

We shall do an energy estimate using $\partial_t$ as a multiplier, integrating on the region bounded between $\Sigma_0$, $C_u^t$, and $\Sigma_t$. We shall move the boundary term on the $C_u^t$ cone (arising from integration of $\Box \tilde{\Psi}$) on the LHS of the estimate thus obtained and we shall move all the other terms on the RHS. Just as in Lemma~\ref{lem:spacel2}, the boundary flux through $\Sigma_0$ and the error integral arising from $H$ can be controlled. More precisely, letting $\mathcal{R}_{\bar u, \bar t} := \{u \leq \bar u\} \cap \{ 0 \leq t \leq \bar t\}$, we have
\begin{equation}\label{eq:avchar1}
\begin{aligned}
   &\int_{C_u^t}|\bar \p^{(0)} \tilde \Psi|\de v \de \omega  + \int_{\Sigma_t \cap \mathcal{R}_{u,t}}|\p \tilde \Psi|^2 \de x \leq  C \Big|\int_{\mathcal{R}_{u,t}} \p_t \tilde \Psi \, \Box \tilde \Psi \de x \de s \Big|  + C\int_{\Sigma_0}|\p \tilde \Psi|^2 \de x\\
   &\leq C \underbrace{\Big|\int_{\mathcal{R}_{u,t}}  F^{\alpha \beta \gamma} \partial_\alpha h \, \partial_\beta \partial_\gamma \tilde{\Psi}\, \p_t \tilde \Psi \, \de x \de s \Big|}_{(a)}  + C \int_{\Sigma_t \cap \mathcal{R}_{u,t}}|H| |\p_t \tilde \Psi| \de x \de s+ C \Vert \p \tilde \Psi \Vert^2_{L^2(\Sigma_0)}
 \end{aligned}
\end{equation}
We then wish to estimate term $(a)$ in the previous display. To that end, we note the following identity:
\begin{equation*}
    \begin{aligned}
        2 F^{\alpha \beta \gamma} \partial_\alpha h \, \partial_\beta \partial_\gamma \tilde{\Psi} \partial_t \tilde{\Psi} = F^{\alpha \beta \gamma} \partial_\beta (\partial_\alpha h \, \partial_t \tilde{\Psi} \partial_\gamma \tilde{\Psi}) - F^{\alpha \beta \gamma} \partial_\alpha \partial_\beta h \, \partial_t \tilde{\Psi} \partial_\gamma \tilde{\Psi} - \partial_t ( F^{\alpha \beta \gamma} \partial_\alpha h \, \partial_\beta \tilde{\Psi} \partial_\gamma \tilde{\Psi})
        \\ + F^{\alpha \beta \gamma} (\partial_t \partial_\alpha h \, \partial_\beta \tilde{\Psi} \partial_\gamma \tilde{\Psi}) + F^{\alpha \beta \gamma} \partial_\gamma (\partial_\alpha h \, \partial_\beta \tilde{\Psi} \partial_t \tilde{\Psi}) - F^{\alpha \beta \gamma} (\partial_\gamma \partial_\alpha h \, \partial_\beta \tilde{\Psi} \partial_t \tilde{\Psi}).
    \end{aligned}
\end{equation*}
This implies, using the fundamental lemma on the structure of null forms (Lemma~\ref{lem:nullstruct}):
\begin{equation}\label{eq:avchar2}
\begin{aligned}
    &(a) \leq C\int_0^t \int_{\Sigma_s} |\overline{\partial} \partial h| |\partial \tilde{\Psi}|^2 \de x \de s + C\int_0^t \int_{\Sigma_s} |\partial \partial h| |\overline{\partial} \tilde{\Psi}| |\partial \tilde{\Psi}| \de x \de s\\
    & \quad + \Big|\int_{\mathcal{R}_{u,t}} \underbrace{\big( F^{\alpha \beta \gamma} \partial_\beta (\partial_\alpha h \, \partial_t \tilde{\Psi} \partial_\gamma \tilde{\Psi}) - \partial_t ( F^{\alpha \beta \gamma} \partial_\alpha h \, \partial_\beta \tilde{\Psi} \partial_\gamma \tilde{\Psi}) + F^{\alpha \beta \gamma} \partial_\gamma (\partial_\alpha h \, \partial_\beta \tilde{\Psi} \partial_t \tilde{\Psi})  \big)}_{(b)} \, \de x \de s \Big|.
\end{aligned}
\end{equation}
We now note that the term $(b)$ is in divergence form, and we now wish to integrate it by parts. For $\eps_0$ sufficiently small, we can once again absorb the error integrals through $\Sigma_t$ coming from term $(b)$ (after integration by parts) in the LHS, just as in Lemma~\ref{lem:spacel2}. Moreover, because the resulting integral over $\Sigma_t$ has a good sign, we will simply drop it.

All that remains is to control the error terms arising from $(b)$ that are fluxes through $C_u^t$ (after integration by parts). These are the terms:
\begin{equation}
    \begin{aligned}
        \underbrace{\Big| \int_{C^t_u} \big( F^{\alpha \beta \gamma} N_\beta (\partial_\alpha h \partial_t \tilde{\Psi} \partial_\gamma \tilde{\Psi}) - N_0 F^{\alpha \beta \gamma} (\partial_\alpha h \partial_\beta \tilde{\Psi} \partial_\gamma \tilde{\Psi}) + F^{\alpha \beta \gamma} N_\gamma (\partial_\alpha h \partial_\beta \tilde{\Psi} \partial_t \tilde{\Psi}) \big) r^2 \de v \de \omega \Big|}_{(c)}.
    \end{aligned}
\end{equation}
Here, $N_\beta$ is defined as follows. Note that the Euclidean unit normal to the outgoing cone $u = \text{const}$ is given by ${1 \over \sqrt{2}} \partial_u = {1 \over \sqrt{2}} (\partial_t - \partial_r) = {1 \over \sqrt{2}} \partial_t - {x^i \over \sqrt{2} r} \partial_i$. $N_\beta$ is then defined as the one-form arising from lowering the index of ${1 \over \sqrt{2}} \partial_u$ by means of the \emph{Minkowski} metric.

We then have that, using the null condition on term $(c)$, 
\begin{equation}\label{eq:avchar3}
    (c) \leq \int_{C^t_u}  \left (|\overline{\partial} h| |\partial \tilde{\Psi}|^2 + |\partial h| |\overline{\partial} \tilde{\Psi}| |\partial \tilde{\Psi}| \right ) r^2 \de v \de \omega.
\end{equation}
We then combine estimates~\eqref{eq:avchar1},~\eqref{eq:avchar2} and~\eqref{eq:avchar3}, multiply by ${1 \over (1 + |u|)^{1 + \delta}}$ and integrate in $u$ for $u \in \R$. This yields the claim.
\ep

We now turn to the main content of {\bf Step 2}, which is to recover the integrated estimate~\eqref{eq:bstpi2}. Let $i \in \{0, \ldots, N\}$, and apply Lemma~\ref{avcharest1} to equation~\eqref{eq:Psicomm}, choosing
$$
h := \sum_{i=0}^N \phi_i + \sum_{\substack{i,j=\{0, \ldots, N\} \\i \neq j}} \psi_{ij} + \Psi, \qquad \tilde \Psi := \der^I \Psi
$$
(with $I \in I^{\leq N_0}_{\boldsymbol{K}_R}$). Furthermore, we apply said lemma to bound the averaged characteristic energy adapted to the light cone associated with the $i$-th piece of data. We obtain:
 \begin{equation} 
        \begin{aligned}
            \int_0^t \int_{\Sigma_s} {|\overline{\partial}^{(i)}\der^I \Psi|^2  \over (1 + |u_i|)^{1 + \delta}} \de x \de s \le C \Vert \partial \der^I \Psi \Vert_{L^2 (\Sigma_0)}^2 + C \int_0^t \int_{\Sigma_s} |H| |\partial_t \der^I \Psi| \de x \de s
            \\ + C\int_0^t \int_{\Sigma_s} |\overline{\partial} \partial^{\leq 1} h| |\partial \der^I \Psi |^2 \de x \de s + C\int_0^t \int_{\Sigma_s} |\partial^{\leq 1} \partial h| |\overline{\partial} \der^I \Psi| |\partial \der^I \Psi| \de x \de s,
        \end{aligned}
    \end{equation}
with the appropriate choice of $h$ and $H$ arising from equation~\eqref{eq:Psicomm}. The error terms corresponding to $H$ and $h$ are identical to the ones we dealt with in {\bf Step 1}. Hence, these error integrals can be handled in the same way as in {\bf Step 1}

We have the following conclusion:
\begin{equation}
    \Vert (1+|u_i|)^{-\frac 12 - \frac \delta 2} \bar \p^{(i)} \der^I \Psi \Vert_{L^2([0,T]\times \R^3)} \leq \eps^{ 3 -\frac {3 \delta} 4}  R^{-\frac 32 + \delta}  \quad \hspace{8pt} \text{ for } \quad t \in [0,T], \ \  i \in \{0, \ldots, N\}.
\end{equation}
\begin{remark}
    Note that, under the additional assumption $\phi_0 \equiv 0$, we obtain the following estimate with an improvement in terms of the parameter $R$:
    \begin{equation}
         \Vert (1+|u_i|)^{-\frac 12 - \frac \delta 2} \bar \p^{(i)} \der^I \Psi \Vert_{L^2([0,T]\times \R^3)} \leq \eps^{ 3 -\delta}  R^{-\frac 32 + \frac 3 4 \delta}  \quad \hspace{8pt} \text{ for } \quad t \in [0,T], \ \  i \in \{0, \ldots, N\}.
    \end{equation}
\end{remark}

This concludes {\bf Step 2} of the proof. We now turn to {\bf Step 3}, in which we show the $L^\infty$ estimates.

\vspace{20pt}

{\bf Step 3}. In this step, we are going to deduce the improved pointwise estimates~\eqref{eq:bstpi3}--\eqref{eq:bstpi5}. From {\bf Step 1} and {\bf Step 2}, we know that, for all $I \in I^{\leq N_0}_{\boldsymbol{K}_R}$,
\begin{align*}
&\Vert \p \der^I \Psi \Vert_{L^2(\Sigma_t)} \leq \eps^{ 3 -\frac {3\delta} 4} R^{-\frac 32 + \delta} \quad  \text{ for } \quad t \in [0,T].
\end{align*}
Now, using the Sobolev Lemma~\ref{lem:bettersob1} and Lemma~\ref{lem:bettersob2}, we obtain immediately, for all $J \in I^{\leq N_0-3}_{\boldsymbol{K}_R}$:
\begin{align*}
|\p \der^J  \Psi(t, r, \theta,\varphi)| \leq C(d_\Pi) \eps^{ 3 -\frac {3\delta}4}  (1+t)^{-1} R^{-\frac 1 2 +\delta}  \quad  \text{ for } t \in [0,T].
\end{align*}
The constant depends only on $d_\Pi$ because the rescaled vector fields introduce bad weights that depend only on $d_\Pi$. Because $\eps$ is allowed to depend on $d_\Pi$, we can possibly restrict to a smaller value of $\eps$ and obtain
$$
|\p \der^J  \Psi(t, r, \theta,\varphi)| \leq \eps^{ 3 -\frac {7\delta}8}  (1+t)^{-1} R^{-\frac 1 2 +\delta}  \quad  \text{ for } t \in [0,T].
$$
This proves the improved bound~\eqref{eq:bstpi3}.

As for bound~\eqref{eq:bstpi5}, we use the classical Klainerman--Sobolev estimates~\ref{lem:classicalks}, and we obtain, as an application of inequality~\eqref{eq:ks1}, for all multi-indices $J \in I^{\leq N_0 -2}_{\boldsymbol{K}_R \cup \boldsymbol{K}^{(i)}_R}$:
\begin{align}\label{eq:ksfirst}
|\p \der^J \Psi(t, r, \theta,\varphi)| \leq  C(N, d_\Pi) \eps^{ 3 -\frac {3\delta}{4}}  (1+v_i)^{-1}(1+|u_i|)^{-\frac 12} R^{\frac 12 + \delta}  \quad  \text{ for } \quad t \in [0, T], \quad i \in \{0, \ldots, N\}.
\end{align}
Upon possibly restricting $\eps$ to a smaller value, we infer the bound~\eqref{eq:bstpi5}. Moreover, this implies, in particular, the claim~\eqref{eq:bstpi4} in the region $|u_i| \geq c t$, for some $c \in (0,1)$.

Note now that, if $|u_i| \leq ct$, then also $|u_i| \leq c v_i$. Integrating the above display~\eqref{eq:ksfirst} on a line of constant $v_i$ coordinate, noting that $\Psi$ has vanishing initial data, we then have, for all multi-indices $J \in I^{\leq N_0 -2}_{\boldsymbol{K}_R\cup \boldsymbol{K}^{(i)}_R}$:
\begin{equation*}
    |\der^J \Psi(t, r, \theta,\varphi)| \leq  C(N, d_\Pi) \eps^{ 3 -\frac {3\delta} 4}  (1+v_i)^{-\frac 12} R^{\frac 12 + \delta}  \quad  \text{ for } \quad t \in [0, T], \quad i \in \{0, \ldots, N\}.
\end{equation*}
Let us now recall the inequality:
\begin{equation*}
    |\bar \p^{(i)} f| \leq C \frac{R}{1+v_i} \sum_{H\in \boldsymbol{K}^{(i)}_R} |\der^H f|.
\end{equation*}
Using this estimate, we finally have, for all $J \in I^{\leq N_0 -3}_{\boldsymbol{K}_R}$:
\begin{align*}
    |\bar \p^{(i)} \der^J \Psi(t, r, \theta,\varphi)| \leq C(N, d_\Pi) \eps^{ 3 -\frac {3\delta}{4}} (1+t)^{-\frac 32 } R^{\frac 32 + \delta}  \quad  \text{ for } \quad t \in [R^{20}, T], \quad i \in \{0, \ldots, N\}
\end{align*}
Upon possibly restricting to a smaller value of $\eps$, we deduce bound~\eqref{eq:bstpi4}. This concludes the proof of the Theorem.
\end{proof}

We now record the calculations that involve using the averaged characteristic energy estimates in controlling the terms in the proof of Theorem \ref{thm:nonlinear}.
\begin{lemma}\label{lem:shortghost}
	Let $f, h : \R^{3 + 1} \rightarrow \R$ be smooth functions, and let $i \in \{0, \ldots, N\}$. Let us consider the usual coordinates $(u_i,v_i,\theta_i, \varphi_i)$ and $(t,r_i,\theta_i, \varphi_i)$ (see Definition~\ref{def:riviui}). Moreover, with $R \ge 10$ and $\alpha, \beta, \gamma, \mu > 0$ parameters, let $f$ satisfy the bulk bound
	\begin{equation}
	\begin{aligned}
    \int_{\{ t \ge 0 \}} |\overline{\partial}^{(i)} f|^2 (1 + |u_i|)^{-1 - {\delta \over 2}} \de x \de s \le C_1^2 R^{2 \alpha},
	\end{aligned}
	\end{equation}
	and let $h$ satisfy the pointwise bound
	\begin{equation}
	\begin{aligned}
	|\p h| \le {C_2 R^\beta \over (1 + v_i) (1 + |u_i|)^\gamma}
	\end{aligned}
	\end{equation}
	with $\gamma > \delta$. Here, as usual, we used the notation for $\overline{\p}^{(i)}$ (the ``good derivatives'') introduced in Definition~\ref{def:shorthand}.
	
	Then, the following inequality holds true:
	\begin{align}\label{eq:bulkest1}
	&\int_{R^\mu}^\infty \Vert |\overline{\partial}^{(i)} f| \, |\partial h| \Vert_{L^2(\Sigma_t)} \de t \le C C_1 C_2 R^{\alpha + \beta + \mu \left ({\delta \over 2} - 2 \gamma \right )}.
	\end{align}
	Here, $C$ is some constant that does not depend on $C_1$, $C_2$, $R$, $\alpha$, $\beta$, $\gamma$, $\mu$, $f$, or $h$.
	
    Moreover, if we also assume that $\gamma \leq \frac 1 2 + \frac \delta 2$, we have the following inequality:
    \begin{align}
    \label{eq:bulkest2}
    &\left (\int_{R^\mu}^\infty \int_{\Sigma_t} (1 + t)^{1 + \delta} |\overline{\partial}^{(i)} f|^2 |\partial h|^2 \de x \de t \right )^{{1 \over 2}} \le C C_1 C_2 R^{\alpha + \beta + \mu \left (\delta - \gamma \right )}.
	\end{align}
	Here, again, $C$ is some constant that does not depend on $C_1$, $C_2$, $R$, $\alpha$, $\beta$, $\gamma$, $\mu$, $f$, or $h$.
\end{lemma}
\bp[Proof of Lemma~\ref{lem:shortghost}] Let us restrict to the case $i=0$ (without loss of generality), recalling that we adopt the convention $|\overline{\p}f| := |\overline{\p}^{(0)} f|$, and $u:= u_0, v:= v_0, r := r_0$. Let us first focus on the proof of estimate~\eqref{eq:bulkest1}. We have the following pointwise bound:
$$
|\p h|^2 (1+|u|)^{1+\frac \delta 2} \leq C^2_2  \frac{R^{2\beta}}{(1+t)^{1+2\gamma-\frac \delta 2}}.
$$
This implies that
\begin{equation*}
\begin{aligned}
&\int_{R^\mu}^\infty \left (\int_{\Sigma_t} |\overline{\partial} f|^2 |\partial h|^2 \de x \right )^{{1 \over 2}} \de t = \int_{R^\mu}^\infty \left (\int_{\Sigma_t} |\overline{\partial} f|^2 (1+|u|)^{-1 - {\delta \over 2}} (1+|u|)^{1 + {\delta \over 2}} |\partial h|^2 \de x \right )^{{1 \over 2}} \de t \\ 
&\qquad \le \int_{R^\mu}^\infty {C_2 R^\beta \over (1 + t)^{{1 \over 2} + \gamma - {\delta \over 4}}} \left (\int_{\Sigma_t} |\overline{\partial} f|^2 (1+|u|)^{-1 - {\delta \over 2}} \de x \right )^{{1 \over 2}} \de t \\
&\qquad \le C C_2 R^\beta \left (\int_{R^\mu}^\infty {1 \over (1 + t)^{1 + 2 \gamma - {\delta \over 2}}} \de t\right )^{{1 \over 2}} \left (\int_{R^\mu}^\infty \int_{\Sigma_t} |\overline{\partial} f|^2 (1+|u|)^{-1 - {\delta \over 2}} \de x \de t \right )^{{1 \over 2}} \\
&\qquad \le C C_1 C_2 R^{\alpha + \beta + \mu \left ({\delta \over 2} - 2 \gamma \right )},
\end{aligned}
\end{equation*}
as desired.

Let us then focus on the proof of bound~\eqref{eq:bulkest2}. We have the following pointwise bound, under the additional assumption that $\gamma \leq \frac 1 2 + \frac \delta 2$:
$$
|\p h|^2 (1+|u|)^{1 + \delta} (1 + t)^{1 + \delta} \leq C^2_2  \frac{R^{2\beta}}{(1+t)^{2\gamma - 2 \delta}}.
$$
This implies that
\begin{equation*}
\begin{aligned}
& \left (\int_{R^\mu}^\infty \int_{\Sigma_t} (1 + t)^{1 + \delta} |\overline{\partial} f|^2 |\partial h|^2 \de x \de t \right )^{{1 \over 2}} 
\\ & =\left (\int_{R^\mu}^\infty \int_{\Sigma_t} |\overline{\partial} f|^2 (1 + t)^{1 + \delta} (1+|u|)^{-1 - \delta} (1+|u|)^{1 + \delta} |\partial h|^2 \de x \de t \right )^{{1 \over 2}} \\ 
&\qquad \le {C_2 R^\beta \over R^{\mu (\gamma - \delta)}} \left (\int_{R^\mu}^\infty \int_{\Sigma_t} |\overline{\partial} f|^2 (1+|u|)^{-1 - \delta} \de x \de t \right )^{{1 \over 2}} \\
&\qquad \le C C_1 C_2 R^{\alpha + \beta + \mu (\delta - \gamma)}.
\end{aligned}
\end{equation*}
This proves inequality~\eqref{eq:bulkest2} and concludes the proof of the lemma.
\ep

\section{Proof of Theorem~\ref{thm:largedata}}\label{sec:largedata}

\begin{proof}[Proof of Theorem~\ref{thm:largedata}]
 Let $L > 0$ be given (this number corresponds to the total energy of the initial data we are going to focus on). Consider $N \in \N$ to be determined later, and let $N_1 \in \N$, $N_1 \geq 13$ ($N_1$ is the number of derivatives we require on the initial data). Consider moreover a collection of functions $(\phi_i^{(0)}, \phi_i^{(1)})$, $i \in \{1, \ldots, N\}$ which satisfies the following properties:
 \begin{align}
 &\text{supp}(\phi^{(0)}_i) \subset B(0,1), \qquad \text{supp}(\phi^{(1)}_i) \subset B(0,1),\\
 &\Vert \phi_{i}^{(0)}\Vert_{H^{N_1}(B(0,1))} \leq \eps_0, \qquad  \Vert \phi_{i}^{(1)}\Vert_{H^{N_1-1}(B(0,1))} \leq \eps_0,\\
 & \eps_1 \leq \Vert \phi_{i}^{(0)}\Vert_{H^{1}(B(0,1))}, \qquad \eps_1 \leq \Vert \phi_{i}^{(1)}\Vert_{L^2(B(0,1))},\quad \text{for all } i \in \{1, \ldots, N\}.
 \end{align}
 Here, $B(0,1)$ is the three-dimensional Euclidean ball of radius one centered at the origin.
 Moreover, $\eps_0 > 0$ is such that the global stability results of Lemma \ref{prop:decphii} hold true for compactly supported data in the unit ball whose $H^{N_1}$ norms are of size at most $2 \eps_0$, with $N_1 \geq 19$ (the number $19$ is chosen so that the bootstrap argument for the nonlinear equation in Section~\ref{sec:mainproof} carries over). Such initial data can easily be seen to exist. Importantly, we note that $\eps_1$ is independent of $N$.
 
Let now $N = \big\lfloor {100 L \over \eps_1}\big\rfloor$. We then take $N$ points $p_i$ on the unit sphere in $\Sigma_0$ which are roughly equidistributed. We could take, for example, $N$ points equidistributed on the unit circle $\{x= 0 \} \cap \{y^2 + z^2 =1\}$. In this case, the largest distance between two such points is bounded above by $2$, and the smallest pairwise distance between any two distinct such points is bounded below by $2 \sin (\pi /N) \geq \frac \pi N$ (if $N$ is sufficiently large). We note that the ratio between the largest and smallest pairwise distance between the points $p_i$, which we recall is denoted by $d_\Pi$, is a function of $N$ alone.

Now, we scale up by a factor $R > 1$ to be chosen momentarily in terms of $N$ and $d_\Pi$. We emphasize that, since $N$ is a function of $L$ and since $d_\Pi$ is a function of $N$, the value of $R$ will only depend on $L$.

For $R$ sufficiently large, we note that the unit balls around each point $w_i := R \, p_i$ will be pairwise disjoint. We define the collection of translated functions $(\tilde \phi_{i}^{(0)},\tilde \phi_{i}^{(1)})$ for $i \in \{1, \ldots, N\}$, as follows:
\begin{equation}
    \tilde \phi_{i}^{(0)}(x) := \phi^{(0)}_i(x-w_i), \qquad \tilde \phi_{i}^{(1)}(x) := \phi^{(1)}_i(x-w_i), \quad  \text{for all } i \in \{1, \ldots, N\}.
\end{equation}
This implies that that both $\tilde \phi_{i}^{(0)}$ and $\tilde \phi_{i}^{(1)}$ are supported in the unit ball centered at $w_i$, for all $i \in \{1, \ldots, N\}$. We also note that, letting
$$
\phi^{(0)}:= \sum_{i=1}^N \tilde \phi^{(0)}_i, \qquad \phi^{(1)}:= \sum_{i=1}^N \tilde \phi^{(1)}_i,
$$
we have trivially, for $R$ sufficiently large, that
$\Vert \phi^{(0)}\Vert_{H^{1}(\Sigma_0)} \geq L$, and that $\Vert \phi^{(1)}\Vert_{L^2(\Sigma_0)} \geq L$.

We then wish to show that, for $R$ sufficiently large, there exists a global-in-time solution to the initial value problem:
\begin{equation}
    \begin{aligned}
    &\Box \phi + F (d \phi, d^2 \phi) =  G (d \phi, d \phi),\\
    & \phi|_{t=0} = \phi^{(0)},\\
    & \p_t \phi|_{t=0} = \phi^{(1)}.
    \end{aligned}
\end{equation}
Now, because of how $\eps_0$ was chosen, we note that the following initial value problem admits a global-in-time solution $\phi_i$, for all $i \in \{1, \ldots, N\}$:
\begin{equation}
    \begin{aligned}
    &\Box \phi_{i} + F (d \phi_i, d^2 \phi_i) =  G (d \phi_i, d \phi_i),\\
    & \phi_i|_{t=0} = \tilde \phi^{(0)}_{i},\\
    & \p_t \phi_i|_{t=0} = \tilde \phi^{(1)}_{i},
    \end{aligned}
\end{equation}
Furthermore, every $\phi_i$ falls off according to the decay rates described in Lemma \ref{prop:decphii}:
$$
|\bar \p^{(i)} \der^{J_1} \phi_i| \leq C \varepsilon_0 \frac{1}{(1+r_i^2)(1+|u_i|)^\delta}, \quad |\p \der^{J_1} \phi_i| \leq C \varepsilon_0 \frac{1}{(1+v_i)(1+|u_i|)^{1+\delta}}.
$$
Here, $J_1 \in I^{\leq N_1 - 7}_{\boldsymbol{K}_R}$.

Thus, with $\psi_{i j}$ defined as in Section \ref{sub:seconditerate}, the trilinear estimates in Section \ref{sub:improvedenergy} give us that every function $\psi_{i j}$ satisfies the estimates~\eqref{eq:enpsiij} with $\eps_0$ in place of $\eps$. The energy estimates satisfied by the functions $\psi_{i j}$ are
\begin{equation}
    \begin{aligned}
        \sup_{t\geq 0} \Vert \partial \der^{J_2} \psi_{i j} \Vert_{L^2 (\Sigma_t)} \le C(N, d_\Pi) {\eps_0^2 \over R},
    \end{aligned}
\end{equation}
for all $J_2 \in I^{\leq N_1 - 8 }_{\boldsymbol{K}_R}$.

Meanwhile, the pointwise estimates satisfied by the functions $\psi_{i j}$ are, for all $J_3 \in I^{\leq N_1 - 11}_{\boldsymbol{K}_R}$:
\begin{align}
	& \Vert \p \der^{J_3} \psi_{ij} \Vert_{L^\infty (\Sigma_{t})} \le {C(N, d_\Pi) \eps_0^{2} \over {R^{\frac 12}(1+t)}} \qquad  \text{ for } i \neq j \text{ and } (i,j) \in \{1, \ldots, N\} \times \{1, \ldots, N\},\label{eq:LELinfty1}\\
	&\Vert \overline{\partial}^{(i)} \der^{J_3} \psi_{ij} \Vert_{L^\infty (\Sigma_{t})} \le {C(N, d_\Pi) \eps_0^{2} R^3 \over {(1+t)}^{3 \over 2}}, \qquad \text{ for } i \neq j \text{ and } (i,j) \in \{1, \ldots, N\} \times \{1, \ldots, N\}, \label{eq:LELinfty2}\\
	&|\partial \der^{J_3} \psi_{ij}|(t, u_i, \omega) \le {C(N, d_\Pi) \eps_0^{2} R^2 \over {t} (1 + |u_i|)^{1 \over 2}}, \qquad \text{ for } i \neq j \text{ and } (i,j) \in \{1, \ldots, N\} \times \{1, \ldots, N\}.\label{eq:LELinfty3}
\end{align}

We note that, in the above display, neither $i$ nor $j$ can be $0$, because the data are compactly supported in the $N$ balls, meaning that the remainder (what we called $\phi_0$) arising from the parts of initial data which are located far away from all the centers $w_i$ is not present.

We now repeat the proof of Theorem \ref{thm:nonlinear}. This may not be possible unless of $R$ is large enough. However, for $R$ large enough, we can use the fact that most of the estimates in Section~\ref{sec:mainproof} (in particular, those concerning terms $\boldsymbol{(a_1)}$ through $\boldsymbol{(a_{16})}$) have ``room'' in the parameter $R$. This is manifest for a subset of those terms and has already been noted in the proof of Theorem~\ref{thm:nonlinear}. In addition, for a few specific terms, we need to use the fact that, in the case relevant to the proof of Theorem~\ref{thm:largedata}, we have that $\phi_0 \equiv 0$ (see Remark~\ref{rmk:improvedR}). This allows us to close a bootstrap argument.  
More precisely, we start from the following bootstrap assumptions, valid for all $I \in I^{\leq N_0}_{\boldsymbol{K}_R}$ and all $J \in I^{\leq \lfloor N_0/2\rfloor +1 }_{\boldsymbol{K}_R}$, where $N_0\in \N, N_0 \geq 7$:
\begin{align}
&\Vert \p \der^I \Psi \Vert_{L^2(\Sigma_t)} \leq \eps_0^{ 3 -\delta} R^{-\frac 32 + \delta} \quad \hspace{120pt} \text{ for } \quad t \in [0,T], \label{eq:LDbstp1}   \\
&\Vert (1+|u_i|)^{-\frac 12 - \frac \delta 2} \bar \p^{(i)} \der^I \Psi \Vert_{L^2([0,T]\times \R^3)} \leq \eps_0^{ 3 -\delta}  R^{-\frac 32 + \delta}  \quad \hspace{8pt} \text{ for } \quad t \in [0,T], \quad i \in \{1, \ldots, N\}, \label{eq:LDbstp2}  \\
&|\p \der^J  \Psi(t, r, \theta,\varphi)| \leq \eps_0^{ 3 -\delta}  (1+t)^{-1} R^{-\frac 1 2 +\delta}  \quad \hspace{55pt} \text{ for } \quad t \in [0,T], \label{eq:LDbstp3}   \\
&|\bar \p^{(i)} \der^J \Psi(t, r, \theta,\varphi)| \leq \eps_0^{ 3 -\delta}  (1+t)^{-\frac 32 } R^{\frac 32 + \delta}  \quad \hspace{50pt} \text{ for } \quad t \in [R^{20}, T], \quad i \in \{1, \ldots, N\}, \label{eq:LDbstp4}\\
&|\p \der^J \Psi(t, r, \theta,\varphi)| \leq  \eps_0^{ 3 -\delta}  (1+v_i)^{-1}(1+|u_i|)^{-\frac 12} R^{\frac 12 + \delta}  \quad  \text{ for } \quad t \in [R^{20}, T], \quad i \in \{1, \ldots, N\}. \label{eq:LDbstp5}
\end{align}

These bootstrap assumptions are the same as those in Section \ref{sec:mainproof}, the only difference being the presence of $\eps_0$ in place of $\eps$. We then seek to improve these bootstrap assumptions, proving the following, for all $I \in I^{\leq N_0}_{\boldsymbol{K}_R}$ and all $J \in I^{\leq N_0 -3 }_{\boldsymbol{K}_R}$:
\begin{align}
    &\Vert \p \der^I \Psi \Vert_{L^2(\Sigma_t)} \leq \eps_0^{ 3 -\delta} R^{-\frac 32 +\frac 3 4 \delta} \quad  \hspace{111pt} \text{ for } \quad t \in [0,T], \label{eq:aRbstpi1}   \\
    &\Vert (1+|u_i|)^{-\frac 12 - \frac \delta 2} \bar \p^{(i)} \der^I \Psi \Vert_{L^2([0,T]\times \R^3)} \leq \eps_0^{ 3 -\delta} R^{-\frac 32 +\frac 3 4 \delta}   \quad  \hspace{8pt}\text{ for } \quad t \in [0,T], \ i \in \{1, \ldots, N\},\label{eq:aRbstpi2}  \\
    &|\p \der^J  \Psi(t, r, \theta,\varphi)| \leq \eps_0^{ 3 -\delta} (1+t)^{-1} R^{-\frac 1 2 + \frac 78 \delta}  \quad  \hspace{46pt} \text{ for } t \in [0,T], \label{eq:aRbstpi3}   \\
    &|\bar \p^{(i)} \der^J \Psi(t, r, \theta,\varphi)| \leq \eps_0^{ 3 -\delta} (1+t)^{-\frac 32 } R^{\frac 32 + \frac 78 \delta}  \quad \hspace{50pt} \text{ for } \quad t \in [R^{20}, T], \ i \in \{1, \ldots, N\}, \label{eq:aRbstpi4}\\
    &|\p \der^J \Psi(t, r, \theta,\varphi)| \leq \eps_0^{ 3 -\delta}  (1+v_i)^{-1}(1+|u_i|)^{-\frac 12} R^{\frac 12 + \frac 78 \delta}  \quad  \text{ for } \quad t \in [R^{20}, T], \ i \in \{1, \ldots, N\}. \label{eq:aRbstpi5}
\end{align}

Looking at how the terms in~\eqref{eq:masterl2space} were controlled, we note that most of the terms have ``room'' in the parameter $R$ (by a positive power of $R$). Indeed, the inequalities in~\eqref{terms12}, \eqref{terms56}, \eqref{terms7}, \eqref{terms89}, \eqref{terms12d}, \eqref{terms14151}, \eqref{terms14152}, and~\eqref{terms16} still hold with $\eps_0$ instead of $\eps$. In those inequalities, we now use that we can absorb the constant $C(N, d_\Pi)$ by negative powers of $R$ (instead of using positive powers of $\eps_0$). The worst case bound in such inequalities is therefore replaced by
\begin{equation}
    \begin{aligned}
        C(N, d_\Pi) \eps_0^{6 - 2 \delta} R^{-3 + \delta} \log (R) \le \eps_0^{6 - 2 \delta} R^{-3 + {3 \delta \over 2}},
    \end{aligned}
\end{equation}
for $R$ sufficiently large. This recovers the improved bootstrap assumption~\eqref{eq:aRbstpi1} for $R$ sufficiently large.

The remaining terms, namely
$$
\boldsymbol{(a_3) + (a_4)}, \quad \boldsymbol{(a_{10}) + (a_{11})}, \quad \boldsymbol{(a_{13})}
$$
need to be handled differently. See Remarks~\ref{rmk:improved34},~\ref{rmk:improved1011} and~\ref{rmk:improved13}. We are going to show in detail here how to obtain the improvement in $R$ solely for the term $\boldsymbol{(a_3) + (a_4)}$, as the other two are very similar. Indeed, we note that the inequality~\eqref{terms34} has no ``room'' in the parameter $R$, in the presence of a nonzero $\phi_0$.

We then use the fact that the data are compactly supported in $N$ balls, meaning that $\phi_0 = 0$. Because of this, we note that the $\psi_{i j}$ are supported in the set $\big\{t \ge {R \over 10}\big\}$, and similarly, we have that $\Psi$ is supported in $\big\{t \ge {R \over 10}\big\}$. Thus, we obtain
\begin{equation}
    \begin{aligned}
    & \boldsymbol{(a_3) + (a_4)}\\
    &\leq \sum_{i = 0}^N \int_{R / 10}^{t} \int_{\R^3}\big( |F^{\alpha\beta\gamma} \ \p_\beta \p_\alpha \phi_i \ \p_t \der^I \Psi  \  \p_\gamma \der^I \Psi| + |F^{\alpha \beta\gamma} \ \p_t \p_\alpha \phi_i \ \p_\beta  \der^I \Psi \p_\gamma \der^I \Psi|\big)\, \de x \de s.
    \end{aligned}
\end{equation}
Once again, we have that
\begin{equation}
    \begin{aligned}
        &\int_{R / 10}^{t} \int_{\R^3}\big( |F^{\alpha\beta\gamma} \ \p_\beta \p_\alpha \phi_i \ \p_t \der^I \Psi  \  \p_\gamma \der^I \Psi| + |F^{\alpha \beta\gamma} \ \p_t \p_\alpha \phi_i \ \p_\beta  \der^I \Psi \p_\gamma \der^I \Psi|\big)\, \de x \de s\\
        & \quad \leq C\int_{R / 10}^t \int_{\R^3}\big( | \p \bar \p^{(i)} \phi_i| \ |\p \der^I \Psi |  \,  |\p \der^I \Psi| + | \p \p \phi_i| \ |\bar \p^{(i)} \der^I \Psi |  \,  |\p \der^I \Psi| \big)\, \de x \de s.
    \end{aligned}
\end{equation}
Now, we bound the terms in the same way as before, using in addition the fact that the region of integration is restricted to $s \ge {R \over 10}$. We get that
\begin{equation}
    \begin{aligned}
    \int_{R / 10}^t \int_{\R^3} | \p \bar \p^{(i)} \phi_i| \ |\p \der^I \Psi |  \,  |\p \der^I \Psi| \de x \de s \leq C \eps_0^{7-2\delta}R^{-3 + 2 \delta}\int_{R / 10}^t \frac 1 {(1+s)^2}\de s \leq C \eps_0^{7 - 2 \delta} R^{-4 + 2 \delta}.
    \end{aligned}
\end{equation}
We note that this term is now better by one power of $R$ than it was before. Similarly, we have that
\begin{equation}
    \begin{aligned}
        &\int_{R / 10}^t \int_{\R^3}  | \p \p \phi_i| \ |\bar \p^{(i)} \der^I \Psi |  \,  |\p \der^I \Psi| \, \de x \de s \leq C\Big(\int_0^t \frac 1 {(1+s)^{1+\delta}}\Vert \p \der^I \Psi \Vert^2_{L^2(\Sigma_s)} \de s\Big)^{\frac 12} \\
        & \quad \times \Big( \int_{R / 10}^t \int_{\R^3}(1+s)^{1+\delta} |\p\p\phi_i|^2 |\p \bar \p^{(i)} \der^I \Psi|^2\de x \de s \Big)^{\frac 12} \leq C \eps_0^{7 - 2 \delta} R^{-3}.
    \end{aligned}
\end{equation}
Thus, after summing, we get that the new inequality replacing~\eqref{terms34} is
\begin{equation}
    \begin{aligned}
        C(N, d_\Pi) \eps_0^{6 - 2 \delta} R^{-3} \le \eps_0^{6-2\delta} R^{-3 + \frac 3 2\delta},
    \end{aligned}
\end{equation}
where we have absorbed $C(N, d_\Pi)$ in the term $R^{-\frac 32\delta}$. 

This shows that we can control all of the error integrals in order to recover the bootstrap assumptions~\eqref{eq:aRbstpi1} and~\eqref{eq:aRbstpi2}, upon restricting $R$ to be large (depending on $L$).

The pointwise bootstrap assumptions are recovered in exactly the same way as in Section~\ref{sec:mainproof}. Indeed, we have that, for $J \in I^{\leq N_0 -3}_{\boldsymbol{K}_R}$, 
\begin{equation}
    \begin{aligned}
       & |\p \der^J  \Psi(t, r, \theta,\varphi)| \leq C(N, d_\Pi) \eps_0^{ 3 - \delta}  (1+t)^{-1} R^{-\frac 1 2 +\frac 34 \delta}\le \eps_0^{3 - \delta} R^{-{1 \over 2} + {7 \delta \over 8}},
    \end{aligned}
\end{equation}
where we have once again absorbed $C(N, d_\Pi)$ by $R^{-{\delta \over 10}}$. This shows~\eqref{eq:aRbstpi3}. The other pointwise estimates~\eqref{eq:aRbstpi4}, and \eqref{eq:aRbstpi5} follow in a similar way. This completes the proof of the large data theorem (Theorem~\ref{thm:largedata}).
\end{proof}

\appendix

\section{Trace lemmas}
We record the following trace lemmas that were needed in the paper.
\subsection{Trace lemma on \texorpdfstring{$\Sigma_t$}{t = const.}}
\begin{lemma}\label{lem:radialtrace} There exists a positive constant $C$ such that the following holds. Let $f : \R^3 \rightarrow \R$ be a smooth function that decays sufficiently rapidly at infinity. We take polar coordinates $(r,\omega)$ with $\omega \in \mathbb{S}^2$. Then, we have that
	\begin{equation}
	\begin{aligned}
	\Vert f \Vert_{L^2 (S_r)} \le C \Vert f \Vert_{L^2 (\R^3)} + C \Vert \partial f \Vert_{L^2 (\R^3)},
	\end{aligned}
	\end{equation}
	where $S_r$ is the sphere of radius $r$, with $r \ge 1$.
\end{lemma}
\bp[Proof of Lemma~\ref{lem:radialtrace}]
We integrate in the $r$ direction using the fundamental theorem of calculus. We have that
\begin{equation}
\begin{aligned}
h^2 (r,\omega) = \int_r^\infty 2 h (s,\omega) \partial_r h (s,\omega) \de s.
\end{aligned}
\end{equation}
Integrating the previous display over $\mathbb{S}^2$ gives us that
\begin{equation}
\begin{aligned}
\int_{\mathbb{S}^2} h^2 (r,\omega) \de \omega = \int_r^\infty 2 \int_{\mathbb{S}^2} h(s,\omega) \partial_r h (s,\omega) \de \omega \de s.
\end{aligned}
\end{equation}
Using the Cauchy--Schwarz inequality, we obtain
\begin{equation}
\begin{aligned}
\int_{\mathbb{S}^2} h^2 (r,\omega) \de \omega \le 2 \left (\int_r^\infty \int_{\mathbb{S}^2} h^2 (s,\omega) \de \omega \de s \right )^{{1 \over 2}} \left (\int_r^\infty \int_{\mathbb{S}^2} (\partial_r h)^2 (s,\omega) \de \omega \de s \right )^{{1 \over 2}}.
\end{aligned}
\end{equation}
Applying this to the function $h(r,\omega) = r f(r,\omega)$ gives us the desired result.
\ep

\subsection{Trace lemma on \texorpdfstring{$H_t$}{Ht}}

\begin{lemma}\label{lem:tracehyp}
Recall the coordinates $(\tau, \alpha, x, \varphi)$ introduced in display~\eqref{eq:xflathyp}. Recall moreover the definition of the hyperboloids $H_\tau$ (from Definition~\ref{def:hypdef}). There exists a positive constant $C$ such that the following holds. Let $\bar x \in \R$. Moreover, let $f : H_\tau \to \R$ be a smooth, compactly supported function. Then, we have that
\begin{equation}
\begin{aligned}
\Vert f\Vert_{L^2 (H_\tau \cap \{x = \bar x\})}^2 \le C \Vert f\Vert^2_{L^2 (H_\tau)}+ C\Vert\partial_x f\Vert^2_{L^2 (H_\tau)}.
\end{aligned}
\end{equation}
Here, the $L^2$ spaces are defined with respect to the induced volume form on the submanifolds considered.
\end{lemma}
\bp[Proof of Lemma~\ref{lem:tracehyp}]
We have that
\begin{equation}
\begin{aligned}
f^2 (\tau, \alpha, \bar x, \varphi) \le 2 \int_{\bar x}^\infty |f (\tau, \alpha, x, \varphi)| |\partial_x f (\tau, \alpha,x, \varphi)| \de x.
\end{aligned}
\end{equation}
Integrating along the set $H_\tau \cap \{x = \bar x\}$ (i.e., integrating in the $\alpha$ and $\varphi$ variables), and using the Cauchy--Schwarz inequality implies that
\begin{equation}
    \begin{aligned}
        \Vert f\Vert_{L^2 (H_\tau \cap \{x = \bar x\})}^2 \le C \Vert f\Vert^2_{L^2 (H_\tau)}+ C\Vert\partial_x f\Vert^2_{L^2 (H_\tau)},
    \end{aligned}
\end{equation}
as desired.
\ep

\section{Energy estimates and the hyperboloidal foliation}\label{sec:emst}
We recall the stress--energy--momentum tensor associated to the wave equation on Minkowski space (here, $m$ denotes the Minkowski metric):
\begin{equation}
\begin{aligned}
Q_{\mu \nu} [\phi] = \partial_\mu \phi \partial_\nu \phi - {1 \over 2} m_{\mu \nu} \, \partial^\gamma \phi \partial_\gamma \phi.
\end{aligned}
\end{equation}
Let now $D \subset \R^4$ be a bounded, open domain with piecewise smooth boundary $\p D$, such that every smooth piece of $\p D$ is spacelike. In particular, this implies that the Lorentzian unit outer normal to $\p D$, denoted by $N_{\p D}$, is well defined. Let now $X$ be a smooth vector field. Using the fact that $\nabla_\mu Q^{\mu \nu} = (\Box \phi) \p^\nu \phi$ and the divergence theorem, we now have
\begin{equation}
\begin{aligned}
\int_D (\Box \phi) X \phi \de x \de t + \int_D (\nabla_\mu X_\nu )  Q^{\mu\nu}\de x \de t = \int_{\partial D} Q(X,N_{\partial D}) \de \sigma(\partial D).
\end{aligned}
\end{equation}
Here, $\de \sigma(\p D)$ is the volume form associated to the induced Riemannian metric on the bondary $\p D$. When $X$ is a Killing field, the second term in the previous display vanishes, and we are left with
\begin{equation}
\begin{aligned}
\int_D (\Box \phi) X \phi \, \de x \de t = \int_{\partial D} Q(X,N_{\partial D}) \de \sigma(\partial D).
\end{aligned}
\end{equation}
When $X = T = \p_t$, we obtain the usual $\p_t$ energy.

We require the following fact on the $\p_t$ energy flux through the hyperboloidal foliation $H_{\tau}$. 
\begin{lemma}\label{lem:unifspacelike}
Recall the coordinates $(\tau, \alpha, x, \varphi)$ introduced in display~\eqref{eq:xflathyp}. Recall moreover the definition of the hyperboloids $H_\tau$ (from Definition~\ref{def:hypdef}).
There exists a positive constant $C$ such that the following inequality holds true:
\begin{equation}
    \int_{H_\tau \cap \{\rho \leq t/10\}} Q(X,N_{H_\tau}) \de \sigma(H_\tau) \geq C \Vert \p \phi \Vert^2_{L^2(H_\tau \cap \{\rho \leq t/10\})}.
\end{equation}
Here, $\de \sigma(H_\tau)$ is the volume form associated to the induced Riemannian metric on $H_\tau$, and the $L^2$ norm on the RHS is defined with respect to the induced volume form (pullback of the ambient volume form) on the hypersurface $H_\tau$.
\end{lemma}
\begin{proof}[Sketch of proof]
The proof follows in a straightforward manner from the fact that the hypersurface $H_\tau \cap \{\rho \leq t/10\}$ is uniformly spacelike, and from expanding the stress--energy--momentum tensor $Q$ in components.
\end{proof}

\bibliography{bumps.bib}
\bibliographystyle{plain}

\end{document}